\documentclass[11pt,a4paper,reqno]{amsart}

\usepackage{amsfonts,amsmath,amssymb,amsxtra,url,float} 
\usepackage[colorlinks,linkcolor=RoyalBlue,anchorcolor=Periwinkle,citecolor=Orange,urlcolor=Green]{hyperref} 
\usepackage[usenames,dvipsnames]{xcolor} 
\usepackage{enumitem}
\usepackage{geometry} 
\usepackage{graphicx} 
\usepackage{subfigure} 
\usepackage{tikz}
\usetikzlibrary{matrix}
\usetikzlibrary{patterns}
\usetikzlibrary{shadings}\usepgflibrary{shadings} 
\usepackage[all]{xy} 
\usepackage{mathdots} 
\usepackage{yhmath} %
\usepackage{dsfont} 
\usepackage{cite}
\usepackage{mathrsfs} 
\usepackage{multicol} 
\usepackage{changepage}

\usepackage[final]{changes} 
\definechangesauthor[name={Yu-Zhe Liu}, color=cyan]{yzliu}

\numberwithin{figure}{section}

\newtheorem{theorem}{Theorem}[section]
\newtheorem{lemma}[theorem]{Lemma}
\newtheorem{corollary}[theorem]{Corollary}
\newtheorem{main theorem}[theorem]{Main Theorem}
\newtheorem{proposition}[theorem]{Proposition}
\newtheorem{definition}[theorem]{Definition}
\newtheorem{construction}[theorem]{Construction}
\newtheorem{remark}[theorem]{Remark}
\newtheorem{example}[theorem]{Example}

\usetikzlibrary{arrows}

\numberwithin{equation}{section}

\def\<{\langle} 
\def\>{\rangle} 
\def\NN{\mathbb{N}} 
\def\ZZ{\mathbb{Z}} 

\newcommand{\Pic}{F{\tiny{IGURE}}\ }
\newcommand{\modcat}{\mathsf{mod}}
\newcommand{\ind}{\mathsf{ind}}
\newcommand{\kk}{\mathds{k}} 
\newcommand{\Q}{\mathcal{Q}} 
\newcommand{\I}{\mathcal{I}} 
\newcommand{\per}{\mathsf{per}} 
\newcommand{\Hom}{\mathrm{Hom}} %
\newcommand{\End}{\mathrm{End}} %
\renewcommand{\H}{\mathrm{H}} %
\newcommand{\SURF}{\mathbf{S}} 
\newcommand{\Surf}{\mathcal{S}} 
\newcommand{\bSurf}{\partial\mathcal{S}} 
\newcommand{\M}{\mathcal{M}} 
\newcommand{\MM}{\mathfrak{M}} 
\newcommand{\gbullet}{{\color{ForestGreen}\bullet}} 
\newcommand{\rbullet}{{\color{red}\circ}} 
\newcommand{\E}{\mathcal{E}} 
\newcommand{\Dgreen}{\Delta_{\color{ForestGreen}\bullet}} 
\newcommand{\Dred}{\Delta_{\color{red}\circ}} 
\newcommand{\dualDgreen}{\Delta_{{\color{ForestGreen}\bullet}}^{\star}} 
\newcommand{\tDgreen}{\widetilde{\Delta}_{{\color{ForestGreen}\bullet}}} 
\newcommand{\dualDred}{\Delta_{{\color{red}\circ}}^{\star}} 
\newcommand{\tDred}{\widetilde{\Delta}_{{\color{red}\circ}}} 
\newcommand{\OEP}{\mathrm{OEP}} 
\newcommand{\PP}{\mathcal{P}} 
\newcommand{\PGD}{\mathrm{PGD}} 
\newcommand{\GD}{\mathrm{GD}} 
\newcommand{\SURFextra}{\mathbf{S}^{\mathcal{E}}_{\color{ForestGreen}\bullet}} 
\newcommand{\SURFgreen}{\mathbf{S}_{\color{ForestGreen}\bullet}} 
\newcommand{\innerSurf}{\mathcal{S}\backslash\partial\mathcal{S}} 
\newcommand{\SURFred}[1]{\mathbf{S}^{#1}_{\color{red}\circ}} 
\newcommand{\tc}{\tilde{c}} 
\newcommand{\tgamma}{\tilde{\gamma}} 
\newcommand{\teta}{\tilde{\eta}} 
\newcommand{\Y}{\mathcal{Y}} 
\newcommand{\F}{\mathcal{F}} 
\newcommand{\Int}{\mathrm{Int}} 
\newcommand{\ii}{\mathfrak{i}} 
\newcommand{\silt}{\mathrm{silt}} 
\newcommand{\PC}{\mathrm{PC}} 
\newcommand{\CC}{\mathrm{CC}} 
\newcommand{\AC}{\widetilde{\mathrm{AC}}} 
\newcommand{\stautilt}{\mathrm{s}\tau\text{-}\mathrm{tilt}} 
\newcommand{\ta}{\mathrm{t}} 
\newcommand{\nhta}{\mathrm{nht}} 
\newcommand{\hta}{\mathrm{ht}} 
\newcommand{\sa}{\mathrm{s}} 
\newcommand{\ncsa}{\mathrm{ncs}} 
\newcommand{\Tile}{\mathrm{Tile}} 
\newcommand{\T}{\mathcal{T}} 
\newcommand{\Tri}{\mathrm{Tri}} 

\newcommand{\A}{\overrightarrow{\mathbb{A}}} %
\newcommand{\gD}{D_{\color{ForestGreen}\bullet}}
\newcommand{\gDI}{D^{\mathrm{I}}_{\color{ForestGreen}\bullet}}
\newcommand{\gDII}{D^{\mathrm{II}}_{\color{ForestGreen}\bullet}}
\newcommand{\gtD}{\widetilde{D}_{\color{ForestGreen}\bullet}}
\newcommand{\gtDI}{\widetilde{D}^{\mathrm{I}}_{\color{ForestGreen}\bullet}}
\newcommand{\gtDII}{\widetilde{D}^{\mathrm{II}}_{\color{ForestGreen}\bullet}}
\newcommand{\tarc}{\tilde{a}}
\newcommand{\m}{\mathsf{m}}
\renewcommand{\c}{\mathfrak{c}}
\newcommand{\opname}[1]{\operatorname{\mathsf{#1}}}

\newcommand{\proj}{\opname{proj}\nolimits}
\newcommand{\sfT}{\mathsf{T}}
\newcommand{\nh}{\mathrm{nh}}
\newcommand{\rmc}{\mathrm{ct}}

\begin{document}

\title[silted algebras of hereditary gentle algebras]{Classification silted algebras for a quiver of Dynkin type $\mathbb{A}_{n}$ via geometric models}
\thanks{$^{\ast}$Corresponding author.}
\author{Yu-Zhe Liu}
\address{School of Mathematics and Statistics, Guizhou University, Guiyang 550025, P. R. China, ORCID: 0009-0005-1110-386X}
\email{yzliu3@163.com}
\author{Houjun Zhang$^{\ast}$}
\address{School of Science, Nanjing University of Posts and Telecommunications, Nanjing 210023, P. R. China}
\email{zhanghoujun@njupt.edu.cn}


\begin{abstract}
Let $\Q$ be the quiver of Dynkin type $\mathbb{A}_n$ with linear orientation and $A_{n}=k\Q$. 
In this paper, we give a complete classification of the silted algebras of type $A_{n}$ by using the geometric models of gentle algebras. We show that any finite-dimensional algebra is a silted of type $A_{n}$ if and only if it is a tilted of type $A_{n}$ or a tilted algebra of type $A_{m}\times A_{n-m}$ for any positive integer $1\leq m\leq n-1$. Based on the classification, we obtain a formula for computing the number of silted algebras of type $A_{n}$.
\end{abstract}
\maketitle
\section{Introduction}
Silted algebras were introduced by Buan and Zhou \cite{BuanZhou16} as endomorphism
algebras of 2-term silting complexes over finite-dimensional hereditary algebras. Let $\Q$ be the quiver of Dynkin type $\mathbb{A}_n$ with linear orientation and $A_{n}=k\Q$. 
We give a classification of the silted algebras of type $A_{n}$. It is well-known that $A_{n}$ is a gentle algebra. Thus, we do this by using the geometric models of gentle algebras.

As a generalization of iterated tilted algebras of type $\mathbb{A}_{n}$ \cite{AssemHappel81}, gentle algebras have been widespread study and occur in many areas of mathematics such as dimer models \cite{Broomhead12}, Lie algebras \cite{Bocklandt12} and discrete derived categories
\cite{BroomheadPauksztelloPloog17,Vossieck01}.
Recently, gentle algebras have been associated to triangulations or dissections of surfaces, in connection with cluster algebras \cite{AssemBrustleCharbonneauPlamondon10,BrustleZhang11},
or with Fukaya categories of surfaces \cite{HaidenKatzarkovKontsevich17,LekiliPolishchuk20}.
Notice that any gentle algebra can be obtained from a dissection of a surface,
this has led Baur and Coelho Sim\~{o}es to realised gentle algebras as tiling algebras and gave a geometric model for their module categories \cite{BaurCoelho21}.

Let $A$ be a gentle algebra and $S$ be a silting complex in the perfect derived category $\per(A)$. According to the  geometric models for bounded derived categories of gentle algebras which introduced by Haiden, Katzarkov and Kontsevich \cite{HaidenKatzarkovKontsevich17} (also see Opper, Plamondon and Schroll \cite{OpperPlamondonSchroll18}), the collection of arcs corresponding to $S$ forms a full formal arc system of the surface. Moreover, the silting complex $S$ can induce a homologically smooth graded gentle algebra $A^S$ associated to the surface model of $A$. Recall that in \cite{HaidenKatzarkovKontsevich17}, Haiden, Katzarkov and Kontsevich proved that the perfect derived category of a homologically smooth graded gentle algebra is triangle equivalent to the partially wrapped Fukaya category of a graded oriented smooth surface. Thus, $\per(A)$ and $\per(A^S)$ are both triangle equivalent to the homologically Fukaya category of the surface. Then we have a triangle equivalence
from $\per(A)$ to $\per(A^S)$ which sending $S$ to $A^S$. So we obtain that $$\H^0(A^S)=\End_{\per(A^S)}(A^S)=\End_{\per(A)}(S).$$
Now let $S$ be a 2-term silting complex over $A_n$, then $H^0(A_{n}^S)$ is a silted algebra of type $A_n$. Moreover, we have

\begin{theorem} [Theorem \ref{coro-gamma(0,l)}]
Let $A$ be a finite-dimensional algebra. Then the following are equivalent.
\begin{itemize}
  \item [\rm(1)] $A$ is a silted algebra of type $A_{n}$;
  \item [\rm(2)] $A$ is a tilted algebra of type $A_{n}$ or a tilted algebra of type $A_{m}\times A_{n-m}$ for any positive integer $1\leq m\leq n-1$.
\end{itemize}
\end{theorem}
Based on the classification above, in order to compute the number of silted algebras of type $A_{n}$, we need to describe and compute the number of tilted algebras of type $A_{n}$.

Let $A$ be a skew-gentle algebra and $S(A)$ be the surface of $A$. He, Zhou and Zhu \cite{HeZhouZhu23} established a bijection between the set of generalized dissections (see Definition \ref{def:GD}) of $S(A)$ and the set of isoclasses of basic support $\tau$-tilting $A$-modules. In particular, for the gentle algebra $A_{n}$, we show that all the generalized dissections are triangulation. This give us a bijection between the set of all triangulations of the surface $S(A_n)$ and the set of isoclasses of basic support $\tau$-tilting modules over $A_n$. 

Let $\Gamma$ be a generalized dissection of $S(A_n)$. Then $\Gamma$ can induce a homologically smooth algebra $A^{\Gamma}$ which is the same as the endomorphism algebra of module corresponding to $\Gamma$. In order to give a characterization of $A^{\Gamma}$, we need the following definition.

\begin{definition}\rm
Let $\Gamma$ be a generalized dissection of $S(A_n)$. If all the edges of a triangle belong to $\Gamma$, then we call this triangle ${\it complete}$.
\end{definition}

Denote by $P(1)$ the projective module of vertex $1$ in $\Q$ and by $c_{P(1)}$ the permissible curve corresponding to $P(1)$. 
As we all know, if $T$ is a $\tau$-tilting module over $A_n$. Then $P(1)$ is a direct summand of $T$. Now let $\Tri_{\Gamma}(S(A_n))$ be the set of all triangles obtained by $\Gamma$ dividing $S(A_n)$. Then we have
\begin{theorem}[Proposition \ref{prop-non-hereditary}]\label{thm:geometriv-characterization-of-tilted-algebras}
Let $\Gamma$ be a generalized dissection and $c_{P(1)}\in\Gamma$.
\begin{itemize}
  \item[\rm(1)] If $\Tri_{\Gamma}(S(A_n))$ contains no complete triangle, then $A^{\Gamma}$ is a hereditary tilted algebra of type $A_{n}$.
  \item[\rm(2)] If $\Tri_{\Gamma}(S(A_n))$ contains at least one complete triangle, then $A^{\Gamma}$ is a non-hereditary tilted algebra of type $A_{n}$.
\end{itemize}
\end{theorem}

By Theorem \ref{thm:geometriv-characterization-of-tilted-algebras} (2),  we have the following result which is very useful to compute the number of tilted algebras of type $A_n$.

\begin{proposition}[Theorem \ref{prop-tri. contains sp. tile}]
There is a bijection between the following sets:
\begin{itemize}
  \item[\rm(1)] the set of all triangulations which contains $c_{P(1)}$ and has at least one complete triangle, 
  \item[\rm(2)] the set of all non-hereditary tilted algebras of type $A_n$.
\end{itemize} 
\end{proposition}

Now we can compute the number of silted algebras of type $A_n$. Denote by $a_{\sa}(A_n)$ the number of silted algebras of type $A_n$, by $a_{\ta}(A_n)$ the number of tilted algebras of type $A_n$ and by $a_{\ncsa}(A_n)$ the number of silted algebras of type $A_n$ that are not tilted algebras of type $A_n$. Then

\begin{theorem} \label{thm-1}
\begin{itemize}
{\rm \item[(1)] (Theorem \ref{thm-tilt alg})}
The number of tilted algebras of type $A_n$ is
\[a_{\ta}(A_n) = \frac{1}{n+1}\tbinom{2n}{n}  + (1+(-1)^{n-1}) 2^{\frac{n-5}{2}} - 2^{n-2}. \]
\end{itemize}
\begin{itemize}
{\rm \item[(2)] (Theorem \ref{thm-slit alg})}
The number of silted algebras of type $A_n$ is
\begin{align}
a_{\sa}(A_n) & = a_{\ta}(A_n) + a_{\ncsa}(A_n). \nonumber
\end{align}
where \[ a_{\ncsa}(A_n) = \frac{1}{2}\sum_{n'+n''=n \atop 1\le n'\ne n''\le n-1} a_{\ta}(A_{n'})a_{\ta}(A_{n''})
+ \frac{1+(-1)^{n}}{2} \cdot \left(\begin{smallmatrix}A_{\lfloor \frac{n}{2}\rfloor}+1 \\ 2\end{smallmatrix}\right).\]
\end{itemize}
\end{theorem}

This paper is organized as follows. In Section \ref{sec:preliminarys}, we give the preliminaries on silting theory and the geometric models of gentle algebras. In Section \ref{sect-tilted}, we give a geometric characterization of the tilted algebras of type $A_n$. In Section \ref{sec:end-silt}, we give a geometric characterization and a classification of the silted algebras of type $A_n$. In Section \ref{sect-number-silted}, we compute the number of silted algebras of type $A_n$.

\medskip

\noindent\emph{Notations and conventions.} Throughout this paper, let $\kk$ be an algebraically closed field.  Let $A$ be a finite-dimensional $\kk$-algebra. Denote by $\modcat A$ the category of right $A$-modules and by $\H^{i}(A)$ the $i$-th cohomology of $A$ for any integer $i$. For arbitrary two arrows $\alpha$ and $\beta$ of the quiver $\Q$,
if the target $t(\alpha)$ of $\alpha$ is the same as the source $s(\beta)$ of $\beta$, then the product of $\alpha$ and $\beta$ is denoted by $\alpha\beta$.
For simplicity, let $\Surf$ be a smooth oriented surface without puncture such that its boundary $\bSurf$ is non-empty. Moreover, let $\sharp X$ be the number of the elements in the set $X$. Let $M$ be an $A$-module. $|M|$ denotes the number of non-isomorphic indecomposable direct summands of $M$. Moreover, let $\Q$ be the quiver of Dynkin type $\mathbb{A}_n$ with linear orientation: $$\begin{xy}
(-10,0)*+{1}="1",
(0,0)*+{2}="2",
(12,0)*+{\cdots}="3",
(28,0)*+{n-1}="4",
(42,0)*+{n}="5",
\ar"1";"2", \ar"2";"3", \ar"3";"4", \ar"4";"5",
\end{xy}$$ and $A_{n}=\kk \Q$.

\medskip
\noindent\emph{Acknowledgements.} The authors would like to thank Dong Yang for his consistent encouragement and support. They are grateful to Anya Nordskova for carefully reading the manuscript and pointing out an error.
Yu-Zhe Liu is supported by the National Natural Science Foundation of China (Grant No. 12171207),
Guizhou Provincial Basic Research Program (Natural Science) (Grant No. ZK[2024]YiBan066)
and Scientific Research Foundation of Guizhou University (Grant Nos. [2022]53, [2022]65, [2023]16).
Houjun Zhang acknowledges support by the National Natural Science Foundation of China (Grant No. 12301051) and by
Natural Science Research Start-up Foundation of Recruiting Talents of Nanjing University of Posts and Telecommunications (Grant No.
NY222092).

\section{The geometric models of gentle algebras: module categories and derived categories}\label{sec:preliminarys}
In this section we recall some basic notions and properties of silting theory and the geometric models of gentle algebras.
We refer the reader to \cite{KellerVossieck88,AiharaIyama12,BuanZhou16,AdachiIyamaReiten14,BaurCoelho21,OpperPlamondonSchroll18,HaidenKatzarkovKontsevich17} for more details.

\subsection{Tilting and silting theory}

\begin{definition}\rm
Let $\Q$ be a finite, connected, and acyclic quiver and $A=\kk \Q$. An algebra $B$ is said to be {\it tilted} of type $A$ if there exists a tilting module $T$ over $A$ such that $B=\End (T_{A})$.
\end{definition}

\begin{definition}\rm
Let $A$ be a finite-dimensional algebra and $M$ be an $A$-module.
\begin{itemize}
\item[\rm(1)] $M$ is called {\it $\tau$-rigid} if $\Hom_{A}(M,\tau M)=0$, and
$M$ is called {\it $\tau$-tilting} if $M$ is $\tau$-rigid and $|M|=|A|$.
\item[\rm(2)] $M$ is called {\it support $\tau$-tilting} if there exists an idempotent
$e$ of $A$ such that $M$ is a $\tau$-tilting $A/ \langle e \rangle$-module.
\end{itemize}
\end{definition}

Denote by $\mathrm{s\tau}$-$\mathrm{tilt}A$ the set of isomorphism classes of basic support $\tau$-tilting $A$-modules. 

\begin{definition}\label{def:silting-objects}\rm
Let $A$ be a finite-dimensional algebra and $P$ be a complex in the bounded homotopy category of finitely generated projective $A$-modules $K^{b}(\proj A)$.
 \begin{itemize}
\item[\rm(1)] $P$ is called {\it presilting} if $\Hom_{K^{b}(\proj A)}(P,P[i])=0$ for $i>0$.
\item[\rm(2)] $P$ is called {\it silting} if it is presilting and generates $K^{b}(\proj A)$ as a triangulated category.
\item[\rm(3)] $P$ is called {\it 2-term } if it only has non-zero terms in degrees $-1$ and $0$.
\end{itemize}
\end{definition}

Note that 2-term silting complex $P$ is tilting if $\Hom_{K^{b}(\proj A)}(P,P[-1])=0$. We denote by $\mathrm{2}$-$\mathrm{silt}A$ the set of isomorphism classes of basic 2-term silting complexes over $A$. 

\begin{definition}\label{def:silted-algebras}\rm
Let $\Q$ be an acyclic quiver and $A=\kk\Q$. We call an algebra $B$ {\it silted} of type $A$ if there exists a 2-term silting complex $P$ over $A$ such that $B=\End_{K^{b}(\proj A)} (P)$.
\end{definition}

A finite-dimensional algebra $A$ is called {\it shod} \cite{CoelhoLanzilotta99} if for each indecomposable $A$-module $X$, either $\mathrm{pd} X \leq1$ or $\mathrm{id} X \leq 1$ (that is: the projective or the injective dimension is at most one). $A$ is called {\it strictly shod} if $A$ is shod and $\mathrm{gl.dim} A = 3$. Notice that any silted algebra is shod. Furthermore, Buan and Zhou  gave a classification of the connected finite dimensional silted algebras (see \cite[Theorem 2.13]{BuanZhou16}) as following:

\begin{theorem} \label{thm:classification-silted}
Let $A$ be a connected finite-dimensional algebra. Then the following are equivalent:
\begin{itemize}
\item[\rm (1)] $A$ is a silted algebra;
\item[\rm (2)] $A$ is a tilted algebra or a strictly shod algebra.
\end{itemize}
\end{theorem}

\subsection{The geometric model for module categories of gentle algebras}

\begin{definition} \rm
We say a finite-dimensional algebra $A$ is a {\it gentle algebra} if $A\cong \kk\Q/\I$ such that:
\begin{itemize}
  \item[(G1)] Each vertex of $\Q$ is the source of at most two arrows and the target of at most two arrows.

  \item[(G2)] For each arrow $\alpha:x\to y$ in $\Q$, there is at most one arrow $\beta$
  whose source (resp., target) is $y$ (resp., $x$) such that $\alpha\beta\in\I$ (resp., $\beta\alpha\in\I$).

  \item[(G3)] For each arrow $\alpha:x\to y$ in $\Q$, there is at most one arrow $\beta$
  whose source (resp., target) is $y$ (resp., $x$) such that $\alpha\beta\notin\I$ (resp., $\beta\alpha\notin\I$).

  \item[(G4)] $\I$ is generated by paths of length $2$.
\end{itemize}
\end{definition}
Recall that a {\it graded gentle algebra} is a gentle algebra with a grading $|\cdot|: \Q_{1} \to \mathbb{Z}$,
i.e., each arrow is assigned with an integer such that $|\omega_1\omega_2|=|\omega_1|+|\omega_2|$
for arbitrary two paths $\omega_1$ and $\omega_2$ on the bound quiver $(\Q,\I)$.
Thus for each path $\varepsilon_v$ of length zero corresponding to the vertex $v$,
we have $\varepsilon_v^2=\varepsilon_v$ and so $|\varepsilon_v|=0$.
In particular, any $\kk$-algebra $A$ is a graded algebra $(A, |\cdot|)$ with
$|\omega|=0$ for all path $\omega$ on the bound quiver $(\Q,\I)$.
Therefore, when no ambiguity is possible, $(A, |\cdot|)$ is abbreviated as $A$.

\begin{definition} \label{def-MRSgreen} \rm
A {\it marked surface} is a pair $(\Surf, \M)$ defined by a smooth surface $\Surf$ and a finite subset $\M$ of $\bSurf$,
where each element in $\M$, denoted by $\gbullet$, is called a {\it marked point}. The segments obtained by $\M$ divides $\bSurf$ are called {\it boundary segments}.
A {\it marked ribbon surface} is a triple $\SURFgreen = (\Surf, \M, \Dgreen)$ where
$\Dgreen$ is a collection of curves whose endpoints are marked points such that:
\begin{itemize}
  \item for any two curves $c_1, c_2\in \Dgreen$ we have $c_1\cap c_2\cap (\Surf\backslash\bSurf) =\varnothing$;
  \item each {\it elementary $\gbullet$-polygon}, the polygon obtained by $\Dgreen$ cutting $\Surf$,
    has a unique side which is not belong to $\Dgreen$ (thus this side is a subset of $\bSurf$).
    We denote by $\OEP(\SURFgreen)$ the set of all elementary $\gbullet$-polygons.
\end{itemize}
\noindent  $\Dgreen$ is called a {\it full formal $\gbullet$-arc system} (=$\gbullet$-FFAS)
or an {\it admissible $\gbullet$-dissection} of $\Surf$ and its elements are called {\it $\gbullet$-arcs}.

Moreover, for any digon in $\OEP(\SURFgreen)$, we add a point $\rbullet$ on its boundary segment, and called a {\it extra marked point}.
We denote by $\E$ the set of all extra marked points and $\SURFextra$ the marked ribbon surface with extra marked points. A bounded component of $\Surf$ is {\it unmarked} if the bounded component with no marked points in $\M$.
\end{definition}

Next we recall how to construct a gentle algebra from a $\gbullet$-FFAS of a marked ribbon surface.

\begin{construction} \label{const-gentle alg}
Let $\SURFextra = (\Surf, \M, \Dgreen)^{\E}$  be a marked ribbon surface. Then we can construct a finite-dimensional $\kk$-algebra $A(\SURFextra)=\kk\Q/\I$ as following:
\begin{itemize}
  \item[\rm Step 1]
    there is a bijection $\mathfrak{v}: \Dgreen \to \Q_0 $, i.e., each $\gbullet$-arc can be viewed as a vertex in $\Q_0$;
  \item[\rm Step 2]
    for each elementary $\gbullet$-polygon $\PP$, there exists an arrow $\alpha: \mathfrak{v}(a_1) \to \mathfrak{v}(a_2)$
    if $a_1, a_2\in \Dgreen$ are two edges of $\PP$ with common endpoints $p\in \Y$
    and $a_1$ precede $a_2$ in counterclockwise order around $p$;
  \item[\rm Step 3]
    the ideal $\I$ is generated by $\alpha\beta$, where $\mathfrak{v}^{-1}(s(\alpha))$, $\mathfrak{v}^{-1}(t(\alpha))=\mathfrak{v}^{-1}(s(\beta))$, $\mathfrak{v}^{-1}(t(\beta))$ are edges of the same elementary $\gbullet$-polygon.
\end{itemize}
\end{construction}

The corresponding $\SURFextra \mapsto A(\SURFextra)$ given by {\rm Construction} \ref{const-gentle alg} induces a bijection
between the set of all homotopy classes of marked ribbon surfaces and that of isoclasses of gentle algebras.
Therefore, up to homotopy equivalence, for any gentle algebra $A$, there exists a unique marked ribbon surface $\SURFextra$
such that $A\cong A(\SURFextra)$. We say $\SURFextra$ is the marked ribbon surface of $A$ and denote it by $\SURFextra(A)$.

In order to describe the indecomposable modules of gentle algebras, Baur and Coelho Sim\~{o}es \cite{BaurCoelho21} defined permissive curves.

A {\it curve} in $\SURFextra=(\Surf,\M,\Dgreen)^{\E}$  is a continuous map $c:[0,1]\to \Surf$ such that
\begin{itemize}
\item[(1)] for any $0<t<1$, $c(t)\in \innerSurf$;
\item[(2)] $c(0),c(1)\in \M\cup\E$ or $c(0)=c(1)\in \innerSurf$.
\end{itemize}
Moreover, we say $c$ is {\it consecutively crossing} $u, v \in \Dgreen$ if
the segment of $c$ between the points $p_1 = c \cap u$ and $p_2 = c \cap v$ does not cross any other arc in $\Dgreen$.
For any curve $c$, we always assume that $c$ has minimal intersection number with $\Dgreen$.

\begin{definition} \label{def-PC-curve} \rm
Let $\SURFextra = (\Surf,\M,\Dgreen)^{\E}$ be a marked ribbon surface.
\begin{enumerate}
\item[(1)] A curve $c$ is called {\it permissible} if the following conditions are satisfied.
\begin{enumerate}
\item The winding number of $c$ around any unmarked boundary component of $\Surf$ is either 0 or 1;
\item If $c$ consecutively crosses two (possibly not distinct) arcs $u$ and $v$ in $\Dgreen$,
then $u$ and $v$ have a common endpoint $p\in \M$, and locally we have a triangle with vertex $p$.
\end{enumerate}
\item[(2)] A {\it permissible closed curve} is a closed curve $c$ satisfying condition (1)(b).
\end{enumerate}
\end{definition}

\begin{definition} \label{def-PC} \rm
Any two permissible curves $c_1$ and $c_2$ in $\SURFextra$ are called {\it equivalent},
if one of the following conditions holds:
\begin{itemize}
  \item[\rm(1)] There is a sequence of consecutive edges $\delta_1,\ldots,\delta_k$
    of elementary $\gbullet$-polygon $\PP$ (which is not of Case I shown in \Pic \ref{Case I and II}) such that:
    \begin{itemize}
      \item $c_1$ is isotopic to the concatenation of $c_2$ and $\delta_1,\ldots,\delta_k$,
      \item $c_1$ starts at an endpoint of $\delta_1$ (resp., $\delta_k$),
        $c_2$ starts at an endpoint of $\delta_k$ (resp., $\delta_1$),
        and their first crossing with $\Dgreen$ is with the same edge of $\PP$.
    \end{itemize}
  \item[\rm(2)] The starting points of $c_1$ and $c_2$ are marked points of a elementary $\gbullet$-polygon $\PP$
    of Case I or II shown in \Pic \ref{Case I and II};
    their first crossing with $\Dgreen$, say $\delta$, is the same as edge of $\PP$
    and the segments of $c_1$ and $c_2$ between $\delta$ and their ending points are isotopic.
\end{itemize}
\end{definition}

\begin{figure}[htbp]
\definecolor{ffqqqq}{rgb}{1,0,0}
\definecolor{qqwuqq}{rgb}{0,0.5,0}
\begin{tikzpicture}
\filldraw[color=black!20] (0,0) circle (0.4);
\draw[line width=1.2pt] (0,0) circle (0.4);
\filldraw[color=black!20] (-2,-1.5) to (2,-1.5) to (2,-1.7) to (-2,-1.7) to (-2,-1.5);
\draw[line width=1.2pt] (-2,-1.5) to (2,-1.5);
\filldraw[qqwuqq] (0, -1.5) circle (0.1);
\draw[qqwuqq][line width=1.2pt] (0, 1.3) to[out=180, in=135] (0, -1.5);
\draw[qqwuqq][line width=1.2pt] (0, 1.3) to[out=  0, in= 45] (0, -1.5);
\draw ( 0.00,-1.75) node[below]{Case I};
\end{tikzpicture}
\ \ \ \ \ \ \ \
\begin{tikzpicture}
\filldraw[color=black!20] (0,0) circle (0.4);
\draw[line width=1.2pt] (0,0) circle (0.4);
\filldraw[color=black!20] (-2,-1.5) to (2,-1.5) to (2,-1.65) to (-2,-1.65) to (-2,-1.5);
\filldraw[color=black!20] (-2, 1.5) to (2, 1.5) to (2, 1.65) to (-2, 1.65) to (-2, 1.5);
\draw[line width=1.2pt] (-2,-1.5) to (2,-1.5); \draw[line width=1.2pt] (-2, 1.5) to (2, 1.5);
\filldraw[qqwuqq] (0, -1.5) circle (0.1); \filldraw[qqwuqq] (0, 1.5) circle (0.1);
\draw[qqwuqq][line width=1.2pt] (0, 1.5) to[out=-135, in=135] (0, -1.5);
\draw[qqwuqq][line width=1.2pt] (0, 1.5) to[out= -45, in= 45] (0, -1.5);
\draw ( 0.00,-1.75) node[below]{Case II};
\end{tikzpicture}
\caption{\textsf{The elementary $\gbullet$-polygon $\PP$ is of the form Case I and Case II}}
\label{Case I and II}
\end{figure}

We say a permissible curve $c$ is {\it trivial} if $\bigcup\limits_{a\in\Dgreen}c\cap a\cap (\innerSurf)=\varnothing$ up to equivalence.

Let $A$ be a gentle algebra. Butler-Ringel and Wald-Waschb\"{u}sch showed that any indecomposable $A$-module is either a string module or a band module. In \cite{BaurCoelho21}, Baur and Coelho Sim\~{o}es established a bijection between the equivalent classes of permissible curves and the indecomposable $A$-module over gentle algebras. Denote by $\PC(\SURFextra)$ the set of all equivalent classes of permissible curves with endpoints lying in $\M\cup\E$ and by $\CC(\SURFextra)$ the set of all homotopy classes of permissible curves without endpoints.

\begin{theorem} Let $\mathscr{J}$ be the set of all Jordan blocks with non-zero eigenvalue. There exists a bijection
    \[ \MM: \PC(\SURFextra) \cup (\CC(\SURFextra)\times\mathscr{J}) \to \ind(\modcat A) \]
between the set $\PC(\SURFextra) \cup (\CC(\SURFextra)\times\mathscr{J})$ of all permissible curves and the set $\ind(\modcat A)$ of all isoclasses of indecomposable modules in $\modcat A$.
\end{theorem}

\begin{remark} \rm
Let $[c]\in \PC(\SURFextra)$ and $([b],\pmb{J})\in \CC(\SURFextra)\times\mathscr{J}$. Then $\MM([c])$ is a string module  and $\MM([b],\pmb{J})$ is a band module.
\end{remark}

Baur and Coelho Sim\~{o}es showed that every irreducible morphism between string modules can be described by pivot elementary move. Next we recall the definition of pivot elementary move (see \cite[Definition 3.14]{BaurCoelho21}).

\begin{definition}\rm \label{def-pivot}
Let $A$ be a gentle algebra and $\SURFextra(A)=(\Surf(A)$, $\M(A)$, $\Dgreen(A))^{\E}$ be the marked ribbon surface of $A$.
Let $c=c_1c_2\cdots c_r: [0,1] \to \Surf$ be a permissible curve and consecutively crossing $a_1,a_2,\ldots, a_{r-1}\in \Dgreen(A)$.
Assume that $\PP \in \OEP(\SURFextra(A))$ such that the segment $c_1$ lies in the inner of $\PP$.
The {\it pivot elementary move} $f_t(c)$ (resp., $f_s(c)$) of $c$ is such a permissible curve obtained from $c$ by fixing the ending point $t=c(1)$ and moving $s=c(0)$ in the following way:
\begin{itemize}
  \item If $\PP$ is neither Case I nor II (See \Pic \ref{Case I and II}), then
  move $s$ as much as possible in the negative direction around $\PP$ to the endpoint $s'$ up to equivalent.

  \item If $\PP$ is of Case I:
      let $c'$ be the permissible curve equivalent to $c$ with the same starting point $s=s'$
      (which is the unique vertex of $\PP$)
      and such that it wraps around the unmarked boundary component in the inner of $\PP$ in the clockwise direction.

  \item If $\PP$ is of Case II:
      let $c'$ be the permissible curve equivalent to $c$ with the starting point $s'$ given by one of the two marked points of $\PP$ such that:
      \begin{itemize}
        \item $s'\ne s$;
        \item the winding number around the unmarked boundary component $b$ in the inner of $\PP$ is zero; and
        \item if $b$ is to the right (resp., left) of $c$, then $b$ is to the left (resp., right) of $c'$.
      \end{itemize}
\end{itemize}
$f_t(c)$ is a permissible curve by moving $s'$ to the next marked point or extra marked point in the negative direction of boundary of $\Surf(A)$.
\end{definition}
\subsection{The geometric models for derived categories of gentle algebras} \label{subsec-geo-der}
Let $T_p \Surf$ be the tangent space of surface $\Surf$ at the point $p\in \Surf$. Then its tangent bundle $T\Surf$ is the union $\bigcup_{p\in \Surf} T_p \Surf$. A grading, $\F$, on $\Surf$ is a foliation, that is a section of the projectivized tangent bundle $\mathbb{P}(T\Surf)$.  We say that a surface $\Surf$ is a {\it graded surface} $\Surf^{\F}$ if it is a surface with a foliation $\F$.

\begin{definition}\rm
A {\it graded curve} in $\Surf^{\F}$ is a triple $(\mathbb{I}, c, \tc)$ where:
\begin{itemize}
  \item $\mathbb{I}$ is a 1-manifold;
  \item $c: \mathbb{I} \to \Surf$ is an immersion, i.e., it can be viewed as a curve in $\Surf$; and
  \item $\widetilde{c}$ is a homotopy class of paths in $\mathbb{P}(T_{c(t)}\Surf)$
    from the subspace given by $\F$ to the tangent space of the curve, varying continuously with $t\in \mathbb{I}$.
\end{itemize}
\end{definition}

For simplicity, we denote by $c$ and $\tc$ the curve $c: \mathbb{I}\to \Surf$ and the graded curve $(\mathbb{I}, c, \tc)$, respectively.  Let $\tc_1$ and $\tc_2$ be two graded curves in $\Surf^{\F}$. Denote by $\Int(c_1, c_2)$ the set of all intersections of curves $c_1$ and $c_2$. If $c_1=c_2$,  $\Int(c_1,c_1)$ is the set of all self-intersections of $c_1$. Note that we always assume that $c_1$ and $c_2$ are representatives in their homotopy classes such that their intersections are minimal.

For any $t_i\in \mathbb{I}$ and $i\in \{1,2\}$, let $p = c_1(t_1) = c_2(t_2)\in \Int(c_1, c_2)$ be an intersection of $c_1$ and $c_2$ and $\dot{c}_1(t_1) \neq \dot{c}_2(t_2) \in \mathbb{P}(T_p\Surf)$, then the {\it intersection index} $\ii_p(\tc_1, \tc_2)$ of $c_1$ and $c_2$ at $p$ is induced by the homotopy classes of paths in $\mathbb{P}(T_p\Surf)$:

\begin{align} \label{formula-index}
\ii_p(\tc_1, \tc_2) := \tc_1(t_1) \cdot \kappa_{12} \cdot \tc_2(t_2)^{-1} \in \pi_1(\mathbb{P}(T_p\Surf)) \cong \mathbb{Z}
\end{align}
\noindent where
\begin{itemize}
  \item $\tc_i(t_i)$ is the homotopy classes of paths from $\F(p)$ to $\dot{c}_i(t_i)$;
  \item $\kappa_{12}$ is the paths from $\dot{c}_1(t_1)$ to $\dot{c}_2(t_2)$ given by clockwise rotation in $T_p\Surf$ by an angle $<\pi$;
  \item $\pi_1(\mathbb{P}(T_p\Surf))$ is the fundamental group defined on the projectivized of the tangent space $T_p\Surf$.
\end{itemize}
Note that this notation is correct only if $c_1,c_2$ pass through $p$ exactly once and a point of transverse intersection of graded curves $\tc_1$ and $\tc_2$ determines an integer. By the definition above, we have
\begin{align}\label{formula-int. ind. in HKK 2.1}
\ii_p(\tc_1, \tc_2) + \ii_p(\tc_2, \tc_1) = 1\ (\text{see \cite[2.1]{HaidenKatzarkovKontsevich17}}).
\end{align}
Moreover, we denote by  $\Int^d(\tc_1,\tc_2):= \sharp\{p\in \Int(c_1, c_2) \mid \ii_p(\tc_1, \tc_2) = d \}$ the set of the number of the intersections of graded curves $\tc_1$ and $\tc_2$ such that the intersection index equals to $d$.

\begin{definition} \rm \label{def-MRSred}
A {\it graded marked surface} is a triple $(\Surf^{\F}, \M, \Y)$ (or equivalently, a quadruple $(\Surf, \M, \Y, \F)$) where:
\begin{itemize}
  \item $\Surf^{\F}$ is a graded surface;
  \item $\M$ is a finite subset of marked points on $\bSurf$ whose elements are represented by symbols $\gbullet$
    and called {\it $\gbullet$-marked points};
  \item $\Y$ is a finite subset of marked points on $\bSurf$ whose elements are represented by symbols $\rbullet$
    and called {\it $\rbullet$-marked points};
  \item Marked points in $\M$ and $\Y$ are alternative in every boundary component.
\end{itemize}
A {\it graded marked ribbon surface}, say $\SURFred{\F}$, is a quadruple $(\Surf^{\F}, \M,$ $\Y, \tDred)$
(or equivalently, a quintuple $(\Surf, \M, \Y, \F, \tDred)$) where:
\begin{itemize}
  \item $\Dred$ is a collection of curves whose endpoints are $\rbullet$-marked points such that:
    \begin{itemize}
      \item for any two curves $c_1, c_2\in \Dred$, we have $c_1\cap c_2\cap (\Surf\backslash\bSurf) =\varnothing$;
      \item each {\it elementary $\rbullet$-polygon}, the polygon obtained by $\Dred$ cutting $\Surf$, has a unique edge which is not belong to $\Dred$.
    \end{itemize}
  \item Any curve $c$ in $\Dred$ is a graded curve, that is, $\tDred:=\{\tc\mid c\in \Dred\}$.
\end{itemize}
$\Dred$ is called a {\it full formal closed arc system} (=$\rbullet$-FFAS) of $\Surf$
and its elements are called  $\rbullet$-arc.
Similarly, we can also define the graded full formal closed arc system (=$\rbullet$-grFFAS) $\tDred$ and the graded $\rbullet$-arc.
\end{definition}

\begin{remark} \rm
\begin{itemize}
  \item[(1)] Recall that $\Dgreen$ is a full formal ${\gbullet}$-arc system. In \cite{AmiotPlamondonSchroll23}, $\Dred$ is called the {\it dual} of $\Dgreen$ and written as $\dualDgreen$. Thus we have $\dualDgreen= \Dred$ and $\dualDred = \Dgreen$.
    Moreover, $\SURFextra$ and $\SURFred{\F}$ can decide for each other. Similar to $\tDred$, we call $\tDgreen$ a graded full formal open arc system
    (=$\gbullet$-grFFAS). The elements of $\tDgreen$ are called graded ${\gbullet}$-arcs.
  \item[(2)]
    In \cite{OpperPlamondonSchroll18}, $\Dgreen$ is called a {\it ribbon graph} and there is a unique $\rbullet$-FFAS $\Dred$ such that:
    \begin{itemize}
      \item for any arc $\eta$ in $\Dred$, there exists a unique arc $\eta^{\star}$ such that $\Int(\eta, \eta^{\star})=1$;
      \item for any arc $\gamma$ in $\Dgreen$, there exists a unique arc $\gamma^{\star}$ such that $\Int(\gamma, \gamma^{\star})=1$.
    \end{itemize}
\end{itemize}
\end{remark}


\begin{construction} \label{const-graded gentle alg}
Let $\SURFred{\F} = (\Surf^{\F}, \M, \Y, \tDred)$  be a graded marked ribbon surface, then we can construct a finite-dimensional graded algebra $A = (\kk\Q/\I, |\cdot|)$ as following:
\begin{itemize}
  \item[\rm Step 1]
    $A=A(\SURFgreen)$, where $\SURFgreen = (\Surf, \M, \dualDred)$;
  \item[\rm Step 2]
    For each angle whose sides $a_1$ and $a_2$ are graded $\rbullet$-arcs lying in $\tDred$ such that
    $a_1$ precede $a_2$ in clockwise order around the vertex $p$ of this angle,
    $|\alpha| = \ii_p(\tilde{a}_1, \tilde{a}_2)$ where $\alpha$ is the arrow $\alpha: \mathfrak{v} (a_1^{\star}) \to \mathfrak{v} (a_2^{\star})$;
  \item[\rm Step 3]
    $|\omega_1\omega_1|=|\omega_1|+|\omega_2|$ for arbitrary paths $\omega_1$ and $\omega_2$.
\end{itemize}
\end{construction}

We denote by $A(\SURFred{\F})$ the graded gentle algebra of graded marked ribbon surface $\SURFred{\F}$. Note that the graded gentle algebra $A(\SURFred{\F})$ arising from surface $\SURFred{\F}$ is homologically smooth. Recall that a or a graded $\gbullet$-curve on $\SURFred{\F}$ is a graded curve $([0,1], c, \tc)$, where $c$ is a curve having minimal intersection number with $\Dred$ and $\tc$ is the grading of $c$.

\begin{definition}\rm \label{def-AC}
We say two graded $\gbullet$-curves $([0,1], c_1, \tc_1)$ and $([0,1], c_2, \tc_2)$
consecutively crossing $\tarc_1$, $\ldots$, $\tarc_t$ $\in\tDred$ are {\it equivalent} if:
\begin{itemize}
  \item for any $1\le i<t$, the segments $(c_1)_{i,i+1}$ and $(c_2)_{i,i+1}$ are homotopic
    where $(c_j)_{i,i+1}$ is the segment of $c_j$ given by $\tarc_{i}$ and $\tarc_{i+1}$ cutting $c_j$ for any $j\in\{1,2\}$; 
  \item $\ii_{p_{1i}}(\tc_1,\tarc_i)=\ii_{p_{2i}}(\tc_2,\tarc_i)$ holds for all $i$, where ${p_{ji}}\in c_j\cap a_i$ for any $j\in\{1,2\}$.
\end{itemize}
\end{definition}


Denote by $\AC_{\m}(\SURFred{\F})$ the set of all equivalent classes of graded $\gbullet$-curves with endpoints lying in $\M$
and $\AC_{\oslash}(\SURFred{\F})$  the set of all homotopy classes of graded $\gbullet$-curves satisfying $c(0)=c(1)\in\innerSurf$.

\begin{theorem} \label{thm. perA=AC} 
Let $A$ be a gentle algebra and $\SURFred{\F_A}(A)$ be its graded marked ribbon surface. 
\begin{itemize}
\item[{\rm(1)}] {\rm \cite[Theorem 3.10]{QiuZhangZhou22}}
There exists a bijection
\[X: \big(\AC_{\m}(\SURFred{\F_A}(A))\times\ind(\modcat \kk)\big) \cup \big(\AC_{\oslash}(\SURFred{\F_A}(A)) \times \ind(\modcat\kk[x,x^{-1}])\big) \mathop{\longrightarrow}\limits^{\text{1-1}} \ind(\per A),\]
between the set of graded $\gbullet$-curves with local system $(\tgamma, L)$ and the set of all isoclasses of indecomposable objects in $\per A$,
where $L$ isomorphic to an indecomposable $A_{\tgamma}$-module and
\begin{align}\label{ind. obj.}
A_{\tgamma}=\begin{cases}
\kk
& \text{if the endpoints of $\gamma$ are $\gbullet$-marked points}, \\
\kk[x, x^{-1}]
& \text{if $\gamma(0)=\gamma(1) \in \Surf\backslash\bSurf$}.
\end{cases}
\end{align}

\item[{\rm(2)}] {\rm \cite[Theorem 6.1]{QiuZhangZhou22}}
For arbitrary two indecomposable complexes $X(\tgamma_1)$, $X(\tgamma_2)\in \per A$, where $\tgamma_1$ and $\tgamma_2$ are two graded $\gbullet$-curves, we have the following dimension formula
\begin{align}\label{dim. formula}
{\dim_{\kk}\Hom_{\per A} (X(\tgamma_1), X(\tgamma_2)[d]) = \Int^d (\tgamma_1, \tgamma_2).}
\end{align}
\end{itemize}
\end{theorem}

\begin{remark} \rm
\begin{itemize}
\item[(1)] 
For each graded $\gbullet$-curve $\tc$ crossing graded $\rbullet$-arcs $\tarc_1, \tarc_2, \ldots \in \tDred$,
the complex $C=X(\tc) = (C^d, \partial^d)$ is defined as following:
\begin{itemize}
  \item $C^d = \bigoplus\limits_{\ii_{c\cap a_i}(\tc, \tarc_i)=-d} P(\mathfrak{v}(a_i))$;
  \item the differential $\partial^d$ is induced by graded arrow in $\Q_1$.
\end{itemize}

\item[(2)] The indecomposable objects in $\per A$ corresponding to the first case and the second case shown in (\ref{ind. obj.}) are called {\it string objects} and {\it band objects}, respectively.
\end{itemize}
\end{remark}

\section{A geometric characterization of tilted algebras of type $A_n$} \label{sect-tilted}

In this section, we give a geometric characterization of tilted algebras of type $A_n$.

\subsection{Generalized dissection}
In order to study the tilted algebras of type $A_n$, we need the following definition (see \cite[Definition 5.5]{HeZhouZhu23}).

\begin{definition} \rm \label{def:GD}
Let $\SURFextra = (\Surf, \M, \Dgreen)^{\E}$ be a marked ribbon surface and $\Gamma$ be a set of some permissible curves.
\begin{itemize}
  \item We say $\Gamma$ is a {\it partial generalized dissection } (=PGD), if:
    \begin{itemize}
      \item every curve in $\Gamma$ can not be homotopic to any boundary segment; 
      \item $(c_1\cap c_2)\cap(\innerSurf) = \varnothing$ for arbitrary two curves $c_1,c_2\in\Gamma$; and
      \item $\Gamma$ contains no closed curve.
    \end{itemize}
  \item We say $\Gamma$ is a {\it generalized dissection} (=GD), if $\Gamma$ is a maximal PGD,
    i.e., for any permissible curve $c'\notin\Gamma$, $(c\cap c') \cap(\innerSurf)\neq\varnothing$ for some $c\in\Gamma$.
\end{itemize}
We denote by $\PGD(\SURFextra)$ the set of all partial generalized dissections and $\GD(\SURFextra)$ the set of all generalized dissections of marked ribbon surface $\SURFextra$, respectively.
\end{definition}

Let $A$ be a skew-gentle algebra. He, Zhou and Zhu established a bijection between the set of generalized dissections
and the set of isoclasses of basic support $\tau$-tilting $A$-modules \cite{HeZhouZhu23} .
In particular, for a gentle algebra, we have the following result.

\begin{theorem} \label{thm-HZZ}
Let $A$ be a gentle algebra and $\SURFextra(A)$ the marked ribbon surface of $A$.
Then there exists a bijection
\[
\begin{matrix}
\MM: & \GD(\SURFextra(A)) &\to& \stautilt(A) \\
& \Gamma &\mapsto& \bigoplus\limits_{c\in\Gamma} \MM([c])
\end{matrix}
 \]
from the set $\GD(\SURFextra(A))$ of all generalized dissections of $\SURFextra(A)$ to the set $\stautilt(A)$ of isoclasses of support $\tau$-tilting $A$-modules.
\end{theorem}

\begin{proof}
It is given by \cite[Theorem B]{HeZhouZhu23}. 
\end{proof}

\subsection{Tiles}

\begin{definition} \label{def-tiles} \rm
Let $\SURFextra = (\Surf, \M, \Dgreen)^{\E}$ be a marked ribbon surface and $\Gamma \in \PGD(\SURFextra)$.
We call every polygon, which obtained by $\Gamma$ dividing $\SURFextra$, a {\it tile}.
In particular, we say a tile is {\it complete} if its edges are all lying in $\Gamma$. Denote by $\Tile_{\Gamma}(\SURFextra)$ the set of all tiles with respect to $\Gamma$.
\end{definition}

Let $\SURFextra(A_n)=(\Surf(A_n), \M(A_n), \Dgreen(A_n))^{\E}$ be the marked ribbon surface of $A_n$ and $\Gamma \in \GD(\SURFextra(A_n))$. The following result shows that each $\GD$ of $\SURFextra(A_n)$ is a triangulation.

\begin{lemma} \label{lemm-triangulation}
Let $\Gamma \in \GD(\SURFextra(A_n))$. Then each tile $\T\in\Tile_{\Gamma}(\SURFextra(A_n))$ is a triangle.
\end{lemma}

\begin{proof}
Let $\T\in\Tile_{\Gamma}(\SURFextra(A_n))$ be a tile and the number of edges $>3$ (see \Pic \ref{fig-edges}).
Now we take a curve $c'$ with endpoints $m_n$ and $m_2$, it is easy to see that $c'\in \PC(\SURFextra(A_n))$.
Thus, $\Gamma$ is not maximal.
\end{proof}

\begin{figure}[H]
\begin{center}
\definecolor{ffqqqq}{rgb}{1,0,0}
\definecolor{qqwuqq}{rgb}{0,0.5,0}
\begin{tikzpicture}[scale=0.9]
\draw[line width=1pt,dotted]  (2,0) arc(0:360:2);
\draw[line width=1.2pt]  (2,0) arc(0:70:2);
\draw[line width=1.2pt]  (2,0) arc(0:-70:2);
\draw[line width=1.2pt]  (-2,0) arc(180:110:2);
\draw[line width=1.2pt]  (-2,0) arc(180:250:2);
\filldraw[qqwuqq] ( 2.00, 0.00) circle (0.1);
\filldraw[qqwuqq] ( 1.41, 1.41) circle (0.1); \draw[qqwuqq] ( 1.41, 1.41) node[right]{$m_1$};
\filldraw[qqwuqq] (-1.41, 1.41) circle (0.1); \draw[qqwuqq] ( 1.41,-1.41) node[right]{$m_2$};
\filldraw[qqwuqq] (-2.00, 0.00) circle (0.1);
\filldraw[qqwuqq] (-1.41,-1.41) circle (0.1); \draw[qqwuqq] (-1.41,-1.41) node[left]{$m_{n-1}$};
\filldraw[qqwuqq] ( 1.41,-1.41) circle (0.1); \draw[qqwuqq] (-1.41, 1.41) node[left]{$m_n$};
\draw[blue][line width=1.2pt]  (-1.41,-1.41)--(-1.41, 1.41)--( 1.41, 1.41)--( 1.41,-1.41);
\draw[blue][line width=1.2pt,dotted]  ( 1.41,-1.41)--(-1.41,-1.41);
\draw[blue](0,0) node{$\mathcal{T}$};
\end{tikzpicture}
\caption{\textsf{The tile with $n(>3)$ edges and some edges can be bounded segments.}}
\label{fig-edges}
\end{center}
\end{figure}

By Lemma \ref{lemm-triangulation}, we obtain that each tile $\T\in\Tile_{\Gamma}(\SURFextra(A_n))$
is one of the following three Types shown in \Pic \ref{fig-tile in triangulation}.
Note that in Type 1, $\T$ is a digon, which can be viewed as a triangle with the point $p$ a vertex of $\T$.
\begin{figure}[htbp]
\begin{center}
\definecolor{ffqqqq}{rgb}{1,0,0}
\definecolor{qqwuqq}{rgb}{0,0.5,0}
\begin{tikzpicture}[scale=0.8]
\draw[blue][line width=1.2pt] (-1.73, 1.00)--( 1.73, 1.00);
\draw[blue] (0,1.5) node{$\T$};
\draw[line width=1.2pt]  (2,0) arc (0:180:2);
\draw[line width=1.2pt,dotted]  (2,0) arc (0:-180:2);
\filldraw ( 0.00, 2.00) circle (0.1); \draw ( 0.00, 2.00) node[above]{$p$};
\filldraw ( 1.73, 1.00) circle (0.1);
\filldraw ( 1.73,-1.00) circle (0.1);
\filldraw ( 0.00,-2.00) circle (0.1);
\filldraw (-1.73,-1.00) circle (0.1);
\filldraw (-1.73, 1.00) circle (0.1);
\draw (0,-2.5) node{Type 1: $\T$ is a digon};
\draw[white] (0,-2.9) node{x};
\draw[white] (0,-3.3) node{x};
\draw[white] (0,-2.9) node{p};
\end{tikzpicture}
\ \ \ \
\begin{tikzpicture}[scale=0.8]
\draw[blue][line width=1.2pt] ( 0.00, 2.00)--(-1.73,-1.00)to[out=-60,in=180]( 0.00,-2.00)--( 0.00, 2.00);
\draw[blue] (-0.5, 0) node{$\T$};
\draw[line width=1.2pt] (2,0) arc (0:360:2);
\draw[line width=2pt][white] (-1.28,1.53) arc (130:170:2); \draw[line width=1.2pt][dotted]  (-1.28,1.53) arc (130:170:2);
\draw[line width=2pt][white] (1.97,0.34) arc (10:50:2); \draw[line width=1.2pt][dotted]  (1.97,0.34) arc (10:50:2);
\draw[line width=2pt][white] (1.28,-1.53) arc (-50:50:2); \draw[line width=1.2pt][dotted]  (1.28,-1.53) arc (-50:50:2);
\filldraw ( 0.00, 2.00) circle (0.1); \draw ( 0.00, 2.00) node[above]{$p$};
\filldraw ( 1.73, 1.00) circle (0.1);
\filldraw ( 1.73,-1.00) circle (0.1);
\filldraw ( 0.00,-2.00) circle (0.1);
\filldraw (-1.73,-1.00) circle (0.1);
\filldraw (-1.73, 1.00) circle (0.1);
\draw (0,-2.5) node{Type 2: $\T$ is a triangle };
\draw (0,-2.9) node{with 2 edges lying in $\Gamma$};
\draw (0,-3.3) node{ };
\end{tikzpicture}
\ \ \ \
\begin{tikzpicture}[scale=0.8]
\draw[blue][line width=1.2pt] ( 0.00, 2.00)--(-1.73,-1.00)--( 1.73,-1.00)--( 0.00, 2.00);
\draw[blue] (0,0) node{$\T$};
\draw[line width=1.2pt] (2,0) arc (0:360:2);
\draw[line width=2pt][white] (1.97,0.34) arc (10:50:2); \draw[line width=1.2pt][dotted]  (1.97,0.34) arc (10:50:2);
\draw[line width=2pt][white] (-1.28,1.53) arc (130:170:2); \draw[line width=1.2pt][dotted]  (-1.28,1.53) arc (130:170:2);
\draw[line width=2pt][white] (-0.68,-1.88) arc (250:290:2); \draw[line width=1.2pt][dotted] (-0.68,-1.88) arc (250:290:2);
\filldraw ( 0.00, 2.00) circle (0.1); \draw ( 0.00, 2.00) node[above]{$p$};
\filldraw ( 1.73, 1.00) circle (0.1);
\filldraw ( 1.73,-1.00) circle (0.1);
\filldraw ( 0.00,-2.00) circle (0.1);
\filldraw (-1.73,-1.00) circle (0.1);
\filldraw (-1.73, 1.00) circle (0.1);
\draw (0,-2.5) node{Type 3: $\T$ is a triangle };
\draw (0,-2.9) node{with all edges lying in $\Gamma$};
\draw (0,-3.3) node{ };
\end{tikzpicture}
\caption{\textsf{All black points $\bullet$ in this figure are either marked points $\gbullet$ or extra marked points $\rbullet$.}}
\label{fig-tile in triangulation}
\end{center}
\end{figure}

By Theorem \ref{thm-HZZ}, we have the following result.

\begin{corollary} \label{coro-triangulation}
There is a bijection
\[
\begin{matrix}
\MM: & \Tri(\SURFextra(A_n)) &\to& \stautilt(A_n) \\
& \Gamma &\mapsto& \bigoplus\limits_{c\in\Gamma} \MM([c])
\end{matrix}
 \]
from the set $\Tri(\SURFextra(A_n))$ of all triangulations of $\SURFextra(A_n)$ to the set $\stautilt(A_n)$ of isoclasses of support $\tau$-tilting modules in $\modcat(A_n)$.
\end{corollary}

For each $\Gamma\in\Tri(\SURFextra(A_n))$, we denote by $\Tri_{\Gamma}(\SURFextra(A_n))$ the set of all triangles obtained by $\Gamma$ dividing $\SURFextra(A_n)$.

\subsection{The algebras induced by generalized dissections}
Note that each $\Gamma\in\GD(\SURFextra(A_n))$ can induce an algebra as following:

\begin{definition} \label{def-alg of GD} \rm
Let $\Gamma\in\GD(\SURFextra(A_n))$. Then $\Gamma$ induce an algebra $A^{\Gamma}=\kk\Q/\I$ as following:
\begin{itemize}
  \item $\Q_0=\Gamma$;
  \item for any triangle $\T\in\Tri_{\Gamma}(\SURFextra(A_n))$,
    each angle of $\T$ provides an arrow $\alpha\in\Q_1$
    such that $s(\alpha)=c_1$ and $t(\alpha)=c_2$ where:
    \begin{itemize}
      \item $c_1$ and $c_2$ are two sides of this angle (without loss of generality, we suppose that $c_1(0)=c_2(0)=p$);
      \item $c_2$ is left to $c_1$ at the $p$; and
      \item $[f_p^{i}(c_2)]$ is non-trivial for any $1\le i\le \theta$,
        where $\theta$ is the positive integer such that $[c_1]=[f_p^{\theta}(c_2)]$,
    \end{itemize}
  \item if there are arrows $\alpha$ and $\beta$ such that
  $s(\alpha)$, $t(\alpha)=s(\beta)$, $t(\beta)$ are edges of the same triangle,
  then $\alpha\beta\in \I$.
\end{itemize}
\end{definition}

The following result shows that the induced algebra $A^{\Gamma}$ by $\Gamma$ is the endomorphism algebra of $\MM(\Gamma)$.

\begin{theorem} \label{thm-endo. alg.}
Let $\Gamma\in \GD(\SURFextra(A_n))$. Then we have $\End(\MM(\Gamma)) \cong A^{\Gamma}$.
\end{theorem}

\begin{proof}
Let $\End(\MM(\Gamma)) \cong \kk\widehat{\Q}/\widehat{\I}$. It is easy to see that $|\MM(\Gamma)|$ equals to $\sharp\widehat{\Q}_0$.
Then there is a bijection
\[\varphi_0: \widehat{\Q}_0 \to \Q_0=\Gamma, \MM[c]\mapsto [c].\]
Assume $c_1\not\simeq c_2\in \Gamma$, we have a bijection between the set of all paths from $\MM([c_1])$ to $\MM([c_2])$ and the basis of $\Hom_{A_n}(\MM([c_2]), \MM([c_1]))$.
Then for each $\varphi$ lying in the basis of $\Hom_{A_n}(\MM([c_2]), \MM([c_1]))$,
we have
\[\varphi = \phi_{m}\phi_{m-1}\cdots\phi_1 \  \text{($\phi_i$ is irreducible for all $1\le i\le m$)}.\]
Thus, there is a sequence of pivot elementary moves $\{f_i \mid 1\le i\le m\}$ such that
\begin{center}
$[f_m\cdots f_2f_1(c_1)]=[c_2]$.
\end{center}
Furthermore, this sequence provides a path from $[c_1]=\varphi_0(\MM([c_1]))$ to $[c_2]=\varphi_0(\MM([c_2]))$.
Thus there exists a bijection $\varphi_1: \widehat{\Q}_1\to \Q_1$ such that
$s(\varphi_1(\alpha))=\varphi_0(s(\alpha))$ and $t(\varphi_1(\alpha))=\varphi_0(t(\alpha))$ hold for all $\alpha\in \widehat{\Q}_1$.

Let $\alpha\neq\beta\in\widehat{\Q}_1$ and $t(\alpha)=s(\beta)$. Then there is a triangle $\T$ whose edges are $s(\alpha)$, $t(\alpha)=s(\beta)$ and $t(\beta)$. According to the position of the marked point $M$, where $M$ is the same endpoint of all $\gbullet$-arcs,
there are six cases in \Pic \ref{fig-triangulation}.
\begin{figure}[htbp]
\begin{center}
\definecolor{ffqqqq}{rgb}{1,0,0}
\definecolor{qqwuqq}{rgb}{0,0.5,0}
\begin{tikzpicture}[scale=0.9]
\draw[line width=1.2pt]  (0,0) circle (2);
\draw[line width=2pt][white] ( 1.97, 0.35) arc (10:50:2);
\draw[line width=1.2pt][dotted] ( 1.97, 0.35) arc (10:50:2);
\draw[line width=2pt][white] (-1.28, 1.53) arc (130:170:2);
\draw[line width=1.2pt][dotted] (-1.28, 1.53) arc (130:170:2);
\draw[line width=2pt][white] (-0.68,-1.88) arc (250:290:2);
\draw[line width=1.2pt][dotted] (-0.68,-1.88) arc (250:290:2);
\draw[qqwuqq!50][line width=1.2pt] ( 1.73, 1.00) -- ( 1.28, 1.53);
\draw[qqwuqq!50][line width=1.2pt] ( 1.73, 1.00) -- ( 1.73, 1.00);
\draw[qqwuqq!50][line width=1.2pt] ( 1.73, 1.00) -- ( 0.00, 2.00);
\draw[qqwuqq!50][line width=1.2pt] ( 1.73, 1.00) -- (-1.28, 1.53);
\draw[qqwuqq!50][line width=1.2pt] ( 1.73, 1.00) -- (-1.97, 0.35);
\draw[qqwuqq!50][line width=1.2pt] ( 1.73, 1.00) -- (-1.73,-1.00);
\draw[qqwuqq!50][line width=1.2pt] ( 1.73, 1.00) -- (-0.68,-1.88);
\draw[qqwuqq!50][line width=1.2pt] ( 1.73, 1.00) -- ( 0.68,-1.88);
\draw[qqwuqq!50][line width=1.2pt] ( 1.73, 1.00) -- ( 1.73,-1.00);
\draw[qqwuqq!50][line width=1.2pt] ( 1.73, 1.00) -- ( 1.97, 0.35);
\filldraw[qqwuqq] ( 1.97, 0.35) circle (0.1);
\filldraw ( 1.97, 0.35) circle (0.1);
\filldraw ( 1.28, 1.53) circle (0.1);
\filldraw[qqwuqq] ( 0.00, 2.00) circle (0.1);
\filldraw (-1.28, 1.53) circle (0.1);
\filldraw (-1.97, 0.35) circle (0.1);
\filldraw[qqwuqq] ( 1.73,-1.00) circle (0.1);
\filldraw (-0.68,-1.88) circle (0.1);
\filldraw ( 0.68,-1.88) circle (0.1);
\filldraw[qqwuqq] (-1.73,-1.00) circle (0.1);
\draw[blue][line width=1.2pt] ( 0.00, 2.00) -- (-1.73,-1.00); \draw[blue] (-0.86, 0.50) node[left]{$c_1$};
\draw[blue][line width=1.2pt] ( 0.00, 2.00) -- ( 1.73,-1.00); \draw[blue] ( 0.86, 0.50) node[right]{$c_2$};
\draw[blue][line width=1.2pt] (-1.73,-1.00) -- ( 1.73,-1.00); \draw[blue] ( 0.00,-1.00) node[below]{$c_3$};
\draw[blue][->] (-0.40, 0.69)--( 0.40, 0.69); \draw[blue] (0, 0.69) node[above]{$\alpha$};
\draw[blue][->] ( 0.80, 0.00)--( 0.40,-0.69); \draw[blue] (0.60,-0.34) node[right]{$\beta$};
\draw[blue][->] (-0.40,-0.69)--(-0.80, 0.00); \draw[blue] (-0.60,-0.34) node[left]{$\gamma$};
\draw (0,2) node[above]{$p$};
\draw (1.73,-1.00) node[right]{$q$};
\draw (-1.73,-1.00) node[left]{$r$};
\draw (0,-2.5) node{Case 1};
\end{tikzpicture}
\ \ \ \ \ \
\begin{tikzpicture}[scale=0.9]
\draw[line width=1.2pt]  (0,0) circle (2);
\draw[line width=2pt][white] ( 1.97, 0.35) arc (10:50:2);
\draw[line width=1.2pt][dotted] ( 1.97, 0.35) arc (10:50:2);
\draw[line width=2pt][white] (-1.28, 1.53) arc (130:170:2);
\draw[line width=1.2pt][dotted] (-1.28, 1.53) arc (130:170:2);
\draw[line width=2pt][white] (-0.68,-1.88) arc (250:290:2);
\draw[line width=1.2pt][dotted] (-0.68,-1.88) arc (250:290:2);
\draw[qqwuqq!50][line width=1.2pt] ( 0.00, 2.00) -- ( 1.28, 1.53);
\draw[qqwuqq!50][line width=1.2pt] ( 0.00, 2.00) -- ( 0.00, 2.00);
\draw[qqwuqq!50][line width=1.2pt] ( 0.00, 2.00) -- (-1.28, 1.53);
\draw[qqwuqq!50][line width=1.2pt] ( 0.00, 2.00) -- (-1.97, 0.35);
\draw[qqwuqq!50][line width=1.2pt] ( 0.00, 2.00) -- (-1.73,-1.00);
\draw[qqwuqq!50][line width=1.2pt] ( 0.00, 2.00) -- (-0.68,-1.88);
\draw[qqwuqq!50][line width=1.2pt] ( 0.00, 2.00) -- ( 0.68,-1.88);
\draw[qqwuqq!50][line width=1.2pt] ( 0.00, 2.00) -- ( 1.73,-1.00);
\draw[qqwuqq!50][line width=1.2pt] ( 0.00, 2.00) -- ( 1.97, 0.35);
\filldraw[qqwuqq] ( 1.97, 0.35) circle (0.1);
\filldraw ( 1.97, 0.35) circle (0.1);
\filldraw ( 1.28, 1.53) circle (0.1);
\filldraw[qqwuqq] ( 0.00, 2.00) circle (0.1);
\filldraw (-1.28, 1.53) circle (0.1);
\filldraw (-1.97, 0.35) circle (0.1);
\filldraw[qqwuqq] ( 1.73,-1.00) circle (0.1);
\filldraw (-0.68,-1.88) circle (0.1);
\filldraw ( 0.68,-1.88) circle (0.1);
\filldraw[qqwuqq] (-1.73,-1.00) circle (0.1);
\draw[blue][line width=1.2pt] ( 0.00, 2.00) -- (-1.73,-1.00); \draw[blue] (-0.86, 0.50) node[left]{$c_1$};
\draw[blue][line width=1.2pt] ( 0.00, 2.00) -- ( 1.73,-1.00); \draw[blue] ( 0.86, 0.50) node[right]{$c_2$};
\draw[blue][line width=1.2pt] (-1.73,-1.00) -- ( 1.73,-1.00); \draw[blue] ( 0.00,-1.00) node[below]{$c_3$};
\draw[blue][->] (-0.40, 0.69)--( 0.40, 0.69); \draw[blue] (0, 0.69) node[above]{$\alpha$};
\draw[blue][->] ( 0.80, 0.00)--( 0.40,-0.69); \draw[blue] (0.60,-0.34) node[right]{$\beta$};
\draw[blue][->] (-0.40,-0.69)--(-0.80, 0.00); \draw[blue] (-0.60,-0.34) node[left]{$\gamma$};
\draw (0,2) node[above]{$p$};
\draw (1.73,-1.00) node[right]{$q$};
\draw (-1.73,-1.00) node[left]{$r$};
\draw (0,-2.5) node{Case 2};
\end{tikzpicture}
\ \ \ \ \ \
\begin{tikzpicture}[scale=0.9]
\draw[line width=1.2pt]  (0,0) circle (2);
\draw[line width=2pt][white] ( 1.97, 0.35) arc (10:50:2);
\draw[line width=1.2pt][dotted] ( 1.97, 0.35) arc (10:50:2);
\draw[line width=2pt][white] (-1.28, 1.53) arc (130:170:2);
\draw[line width=1.2pt][dotted] (-1.28, 1.53) arc (130:170:2);
\draw[line width=2pt][white] (-0.68,-1.88) arc (250:290:2);
\draw[line width=1.2pt][dotted] (-0.68,-1.88) arc (250:290:2);
\draw[qqwuqq!50][line width=1.2pt] (-1.73, 1.00) -- ( 1.28, 1.53);
\draw[qqwuqq!50][line width=1.2pt] (-1.73, 1.00) -- ( 0.00, 2.00);
\draw[qqwuqq!50][line width=1.2pt] (-1.73, 1.00) -- (-1.28, 1.53);
\draw[qqwuqq!50][line width=1.2pt] (-1.73, 1.00) -- (-1.97, 0.35);
\draw[qqwuqq!50][line width=1.2pt] (-1.73, 1.00) -- (-1.73,-1.00);
\draw[qqwuqq!50][line width=1.2pt] (-1.73, 1.00) -- (-0.68,-1.88);
\draw[qqwuqq!50][line width=1.2pt] (-1.73, 1.00) -- ( 0.68,-1.88);
\draw[qqwuqq!50][line width=1.2pt] (-1.73, 1.00) -- ( 1.73,-1.00);
\draw[qqwuqq!50][line width=1.2pt] (-1.73, 1.00) -- ( 1.97, 0.35);
\filldraw[qqwuqq] ( 1.97, 0.35) circle (0.1);
\filldraw ( 1.97, 0.35) circle (0.1);
\filldraw ( 1.28, 1.53) circle (0.1);
\filldraw[qqwuqq] ( 0.00, 2.00) circle (0.1);
\filldraw (-1.28, 1.53) circle (0.1);
\filldraw (-1.97, 0.35) circle (0.1);
\filldraw[qqwuqq] ( 1.73,-1.00) circle (0.1);
\filldraw (-0.68,-1.88) circle (0.1);
\filldraw ( 0.68,-1.88) circle (0.1);
\filldraw[qqwuqq] (-1.73,-1.00) circle (0.1);
\draw[blue][line width=1.2pt] ( 0.00, 2.00) -- (-1.73,-1.00); \draw[blue] (-0.86, 0.50) node[left]{$c_1$};
\draw[blue][line width=1.2pt] ( 0.00, 2.00) -- ( 1.73,-1.00); \draw[blue] ( 0.86, 0.50) node[right]{$c_2$};
\draw[blue][line width=1.2pt] (-1.73,-1.00) -- ( 1.73,-1.00); \draw[blue] ( 0.00,-1.00) node[below]{$c_3$};
\draw[blue][->] (-0.40, 0.69)--( 0.40, 0.69); \draw[blue] (0, 0.69) node[above]{$\alpha$};
\draw[blue][->] ( 0.80, 0.00)--( 0.40,-0.69); \draw[blue] (0.60,-0.34) node[right]{$\beta$};
\draw[blue][->] (-0.40,-0.69)--(-0.80, 0.00); \draw[blue] (-0.60,-0.34) node[left]{$\gamma$};
\draw (0,2) node[above]{$p$};
\draw (1.73,-1.00) node[right]{$q$};
\draw (-1.73,-1.00) node[left]{$r$};
\draw (0,-2.5) node{Case 3};
\end{tikzpicture}
\\
\begin{tikzpicture}[scale=0.9]
\draw[line width=1.2pt]  (0,0) circle (2);
\draw[line width=2pt][white] ( 1.97, 0.35) arc (10:50:2);
\draw[line width=1.2pt][dotted] ( 1.97, 0.35) arc (10:50:2);
\draw[line width=2pt][white] (-1.28, 1.53) arc (130:170:2);
\draw[line width=1.2pt][dotted] (-1.28, 1.53) arc (130:170:2);
\draw[line width=2pt][white] (-0.68,-1.88) arc (250:290:2);
\draw[line width=1.2pt][dotted] (-0.68,-1.88) arc (250:290:2);
\draw[qqwuqq!50][line width=1.2pt] (-1.73,-1.00) -- ( 1.28, 1.53);
\draw[qqwuqq!50][line width=1.2pt] (-1.73,-1.00) -- ( 0.00, 2.00);
\draw[qqwuqq!50][line width=1.2pt] (-1.73,-1.00) -- (-1.28, 1.53);
\draw[qqwuqq!50][line width=1.2pt] (-1.73,-1.00) -- (-1.97, 0.35);
\draw[qqwuqq!50][line width=1.2pt] (-1.73,-1.00) -- (-1.73,-1.00);
\draw[qqwuqq!50][line width=1.2pt] (-1.73,-1.00) -- (-0.68,-1.88);
\draw[qqwuqq!50][line width=1.2pt] (-1.73,-1.00) -- ( 0.68,-1.88);
\draw[qqwuqq!50][line width=1.2pt] (-1.73,-1.00) -- ( 1.73,-1.00);
\draw[qqwuqq!50][line width=1.2pt] (-1.73,-1.00) -- ( 1.97, 0.35);
\filldraw[qqwuqq] ( 1.97, 0.35) circle (0.1);
\filldraw ( 1.97, 0.35) circle (0.1);
\filldraw ( 1.28, 1.53) circle (0.1);
\filldraw[qqwuqq] ( 0.00, 2.00) circle (0.1);
\filldraw (-1.28, 1.53) circle (0.1);
\filldraw (-1.97, 0.35) circle (0.1);
\filldraw[qqwuqq] ( 1.73,-1.00) circle (0.1);
\filldraw (-0.68,-1.88) circle (0.1);
\filldraw ( 0.68,-1.88) circle (0.1);
\filldraw[qqwuqq] (-1.73,-1.00) circle (0.1);
\draw[blue][line width=1.2pt] ( 0.00, 2.00) -- (-1.73,-1.00); \draw[blue] (-0.86, 0.50) node[left]{$c_1$};
\draw[blue][line width=1.2pt] ( 0.00, 2.00) -- ( 1.73,-1.00); \draw[blue] ( 0.86, 0.50) node[right]{$c_2$};
\draw[blue][line width=1.2pt] (-1.73,-1.00) -- ( 1.73,-1.00); \draw[blue] ( 0.00,-1.00) node[below]{$c_3$};
\draw[blue][->] (-0.40, 0.69)--( 0.40, 0.69); \draw[blue] (0, 0.69) node[above]{$\alpha$};
\draw[blue][->] ( 0.80, 0.00)--( 0.40,-0.69); \draw[blue] (0.60,-0.34) node[right]{$\beta$};
\draw[blue][->] (-0.40,-0.69)--(-0.80, 0.00); \draw[blue] (-0.60,-0.34) node[left]{$\gamma$};
\draw (0,2) node[above]{$p$};
\draw (1.73,-1.00) node[right]{$q$};
\draw (-1.73,-1.00) node[left]{$r$};
\draw (0,-2.5) node{Case 4};
\end{tikzpicture}
\ \ \ \ \ \
\begin{tikzpicture}[scale=0.9]
\draw[line width=1.2pt]  (0,0) circle (2);
\draw[line width=2pt][white] ( 1.97, 0.35) arc (10:50:2);
\draw[line width=1.2pt][dotted] ( 1.97, 0.35) arc (10:50:2);
\draw[line width=2pt][white] (-1.28, 1.53) arc (130:170:2);
\draw[line width=1.2pt][dotted] (-1.28, 1.53) arc (130:170:2);
\draw[line width=2pt][white] (-0.68,-1.88) arc (250:290:2);
\draw[line width=1.2pt][dotted] (-0.68,-1.88) arc (250:290:2);
\draw[qqwuqq!50][line width=1.2pt] ( 0.00,-2.00) -- ( 1.28, 1.53);
\draw[qqwuqq!50][line width=1.2pt] ( 0.00,-2.00) -- ( 0.00, 2.00);
\draw[qqwuqq!50][line width=1.2pt] ( 0.00,-2.00) -- (-1.28, 1.53);
\draw[qqwuqq!50][line width=1.2pt] ( 0.00,-2.00) -- (-1.97, 0.35);
\draw[qqwuqq!50][line width=1.2pt] ( 0.00,-2.00) -- (-1.73,-1.00);
\draw[qqwuqq!50][line width=1.2pt] ( 0.00,-2.00) -- (-0.68,-1.88);
\draw[qqwuqq!50][line width=1.2pt] ( 0.00,-2.00) -- ( 0.68,-1.88);
\draw[qqwuqq!50][line width=1.2pt] ( 0.00,-2.00) -- ( 1.73,-1.00);
\draw[qqwuqq!50][line width=1.2pt] ( 0.00,-2.00) -- ( 1.97, 0.35);
\filldraw[qqwuqq] ( 1.97, 0.35) circle (0.1);
\filldraw ( 1.97, 0.35) circle (0.1);
\filldraw ( 1.28, 1.53) circle (0.1);
\filldraw[qqwuqq] ( 0.00, 2.00) circle (0.1);
\filldraw (-1.28, 1.53) circle (0.1);
\filldraw (-1.97, 0.35) circle (0.1);
\filldraw[qqwuqq] ( 1.73,-1.00) circle (0.1);
\filldraw (-0.68,-1.88) circle (0.1);
\filldraw ( 0.68,-1.88) circle (0.1);
\filldraw[qqwuqq] (-1.73,-1.00) circle (0.1);
\draw[blue][line width=1.2pt] ( 0.00, 2.00) -- (-1.73,-1.00); \draw[blue] (-0.86, 0.50) node[left]{$c_1$};
\draw[blue][line width=1.2pt] ( 0.00, 2.00) -- ( 1.73,-1.00); \draw[blue] ( 0.86, 0.50) node[right]{$c_2$};
\draw[blue][line width=1.2pt] (-1.73,-1.00) -- ( 1.73,-1.00); \draw[blue] ( 0.00,-1.00) node[below]{$c_3$};
\draw[blue][->] (-0.40, 0.69)--( 0.40, 0.69); \draw[blue] (0, 0.69) node[above]{$\alpha$};
\draw[blue][->] ( 0.80, 0.00)--( 0.40,-0.69); \draw[blue] (0.60,-0.34) node[right]{$\beta$};
\draw[blue][->] (-0.40,-0.69)--(-0.80, 0.00); \draw[blue] (-0.60,-0.34) node[left]{$\gamma$};
\draw (0,2) node[above]{$p$};
\draw (1.73,-1.00) node[right]{$q$};
\draw (-1.73,-1.00) node[left]{$r$};
\draw (0,-2.5) node{Case 5};
\end{tikzpicture}
\ \ \ \ \ \
\begin{tikzpicture}[scale=0.9]
\draw[line width=1.2pt]  (0,0) circle (2);
\draw[line width=2pt][white] ( 1.97, 0.35) arc (10:50:2);
\draw[line width=1.2pt][dotted] ( 1.97, 0.35) arc (10:50:2);
\draw[line width=2pt][white] (-1.28, 1.53) arc (130:170:2);
\draw[line width=1.2pt][dotted] (-1.28, 1.53) arc (130:170:2);
\draw[line width=2pt][white] (-0.68,-1.88) arc (250:290:2);
\draw[line width=1.2pt][dotted] (-0.68,-1.88) arc (250:290:2);
\draw[qqwuqq!50][line width=1.2pt] ( 1.73,-1.00) -- ( 1.28, 1.53);
\draw[qqwuqq!50][line width=1.2pt] ( 1.73,-1.00) -- ( 0.00, 2.00);
\draw[qqwuqq!50][line width=1.2pt] ( 1.73,-1.00) -- (-1.28, 1.53);
\draw[qqwuqq!50][line width=1.2pt] ( 1.73,-1.00) -- (-1.97, 0.35);
\draw[qqwuqq!50][line width=1.2pt] ( 1.73,-1.00) -- (-1.73,-1.00);
\draw[qqwuqq!50][line width=1.2pt] ( 1.73,-1.00) -- (-0.68,-1.88);
\draw[qqwuqq!50][line width=1.2pt] ( 1.73,-1.00) -- ( 0.68,-1.88);
\draw[qqwuqq!50][line width=1.2pt] ( 1.73,-1.00) -- ( 1.73,-1.00);
\draw[qqwuqq!50][line width=1.2pt] ( 1.73,-1.00) -- ( 1.97, 0.35);
\filldraw[qqwuqq] ( 1.97, 0.35) circle (0.1);
\filldraw ( 1.97, 0.35) circle (0.1);
\filldraw ( 1.28, 1.53) circle (0.1);
\filldraw[qqwuqq] ( 0.00, 2.00) circle (0.1);
\filldraw (-1.28, 1.53) circle (0.1);
\filldraw (-1.97, 0.35) circle (0.1);
\filldraw[qqwuqq] ( 1.73,-1.00) circle (0.1);
\filldraw (-0.68,-1.88) circle (0.1);
\filldraw ( 0.68,-1.88) circle (0.1);
\filldraw[qqwuqq] (-1.73,-1.00) circle (0.1);
\draw[blue][line width=1.2pt] ( 0.00, 2.00) -- (-1.73,-1.00); \draw[blue] (-0.86, 0.50) node[left]{$c_1$};
\draw[blue][line width=1.2pt] ( 0.00, 2.00) -- ( 1.73,-1.00); \draw[blue] ( 0.86, 0.50) node[right]{$c_2$};
\draw[blue][line width=1.2pt] (-1.73,-1.00) -- ( 1.73,-1.00); \draw[blue] ( 0.00,-1.00) node[below]{$c_3$};
\draw[blue][->] (-0.40, 0.69)--( 0.40, 0.69); \draw[blue] (0, 0.69) node[above]{$\alpha$};
\draw[blue][->] ( 0.80, 0.00)--( 0.40,-0.69); \draw[blue] (0.60,-0.34) node[right]{$\beta$};
\draw[blue][->] (-0.40,-0.69)--(-0.80, 0.00); \draw[blue] (-0.60,-0.34) node[left]{$\gamma$};
\draw (0,2) node[above]{$p$};
\draw (1.73,-1.00) node[right]{$q$};
\draw (-1.73,-1.00) node[left]{$r$};
\draw (0,-2.5) node{Case 6};
\end{tikzpicture}
\caption{\textsf{The black points in figures are elements lying in $\M(A_n)\cup\E(A_n)$.}}
\label{fig-triangulation}
\end{center}
\end{figure}

Assume $s(\alpha)=\MM([c_1])$, $t(\alpha)=s(\beta)=\MM([c_2])$ and $t(\beta)=\MM([c_3])$. Let $p$ (resp., $q$, $r$) be the intersection of $c_1$ and $c_2$ (resp., $c_2$ and $c_3$, $c_3$ and $c_1$). We claim that $\alpha\beta\in\I$ and there is no arrow from $\MM([c_3])$ to $\MM([c_1])$.
\begin{itemize}
  \item[Case] 1: In this case, $[c_3] = [f_r^{t}(c_1)]$ for some $t\in\NN^+$.
    Notice that there exists a unique arc $a\in\Dgreen$ with endpoint $r$ such that $a$ crosses $c_2$,
    thus there exists a positive integer $2\le i\le t-1$ such that $f_r^{i}(c_1)$ is a trivial permissible curve. It follows that
    there is no arrow from $\MM([c_3])$ to $\MM([c_1])$ by Definition \ref{def-alg of GD}.
    On the other hand, the simple module $S(a-1)$ corresponding to the vertex $a-1$ is the socle of $\MM([c_1])$
    and the simple module $S(a+1)$ corresponding to the vertex $a+1$ is the top of $\MM([c_3])$.
    Thus $\Hom_{A_n}(\MM(c_3), \MM(c_1))=0$, and so $\alpha\beta\in\I$.
\end{itemize}
We can prove the Cases 2-6 by the same way. Indeed, Cases 2-6 are all not satisfy our hypothesis.
\end{proof}

\begin{proposition} \label{prop-hereditary}
Let $\Gamma \in \GD(\SURFextra(A_n))$. Then $A^{\Gamma}$ is hereditary if and only if $\Tri_{\Gamma}(\SURFextra(A_n))$ contains no complete tiles which are of the form Cases 1, 3 and 5 shown in \Pic \ref{fig-triangulation}.
\end{proposition}

\begin{proof}
Since $\Tri_{\Gamma}(\SURFextra(A))$ contains a complete tile which is of the form Cases 1 (3 or 5) shown in \Pic \ref{fig-triangulation},
there are $\alpha: c_1 \to c_2$ and $\beta: c_2 \to c_3$ such that $\alpha\beta\in\I$.
Thus the projective dimension of simple module $S(c_1)$ is greater than or equal to $2$. This shows that $A^{\Gamma}$ is non-hereditary.
\end{proof}
\subsection{The classification of tilted algebras of type $A_n$} \label{subsec-classification}
In this subsection, we give a classification of the tilted algebras of type $A_n$. First, we have a simple observation.

\begin{lemma} \label{lemm-invariant of tilt}
Let $T$ be a $\tau$-tilting module over $A_n$. Then $P(1)$ is a direct summand of $T$.
\end{lemma}

\begin{proof}
Since $T$ is a $\tau$-tilting module, there is a triangulation $\Gamma$ such that $\MM(\Gamma)\cong T$.
Assume that $c_{P(1)}$, the permissible curve corresponding $P(1)$, is not belong to $\Gamma$.
Then, any triangle $\T\in\Tri_{\Gamma}(\SURFextra(A_n))$ with vertex $p$ is one of the forms (1) -- (3) shown in
\Pic \ref{fig-tile with vertex p}.
\begin{figure}[htbp]
\begin{center}
\definecolor{ffqqqq}{rgb}{1,0,0}
\definecolor{qqwuqq}{rgb}{0,0.5,0}
\begin{tikzpicture}[scale=0.6]
\draw[line width=1.2pt]  (2,0) arc (0:360:2);
\draw[line width=2pt][white]  (1.414,-1.414) arc (-45:-80:2);
\draw[line width=1.2pt][dotted]  (1.414,-1.414) arc (-45:-80:2);
\filldraw[qqwuqq] ( 0.00, 2.00) circle (0.1);
\draw[ffqqqq][line width=1.2pt] (-1.56, 1.25) circle (0.1);
\filldraw[qqwuqq] (-1.95, 0.45) circle (0.1);
\filldraw[qqwuqq] (-1.95,-0.45) circle (0.1);
\filldraw[qqwuqq] (-1.61,-1.19) circle (0.1);
\filldraw[qqwuqq] (-0.87,-1.80) circle (0.1);
\filldraw[qqwuqq] (-0.00,-2.00) circle (0.1);
\filldraw[qqwuqq] ( 0.87,-1.80) circle (0.1);
\filldraw[qqwuqq] ( 1.61,-1.19) circle (0.1);
\filldraw[qqwuqq] ( 1.94,-0.45) circle (0.1);
\filldraw[qqwuqq] ( 1.95, 0.45) circle (0.1);
\draw[ffqqqq][line width=1.2pt] ( 1.56, 1.25) circle (0.1);
\draw[qqwuqq!25][line width=1.2pt] ( 0.00, 2.00) -- (-1.95, 0.45);
\draw[qqwuqq!25][line width=1.2pt] ( 0.00, 2.00) -- (-1.95,-0.45);
\draw[qqwuqq!25][line width=1.2pt] ( 0.00, 2.00) -- (-1.61,-1.19);
\draw[qqwuqq!25][line width=1.2pt] ( 0.00, 2.00) -- (-0.87,-1.80);
\draw[qqwuqq!25][line width=1.2pt] ( 0.00, 2.00) -- (-0.00,-2.00);
\draw[qqwuqq!25][line width=1.2pt][dotted] ( 0.00, 2.00) -- ( 0.87,-1.80);
\draw[qqwuqq!25][line width=1.2pt] ( 0.00, 2.00) -- ( 1.61,-1.19);
\draw[qqwuqq!25][line width=1.2pt] ( 0.00, 2.00) -- ( 1.94,-0.45);
\draw[qqwuqq!25][line width=1.2pt] ( 0.00, 2.00) -- ( 1.95, 0.45);
\draw[qqwuqq] (0,2) node[above]{$p$};
\draw[blue][line width=1.2pt] (0,2) -- (-1.61,-1.19) to[out=-51.4285,in=150] (-0.87,-1.80) -- (0,2);
\draw[blue](-0.7,0) node{$\T$};
\draw (0,-2.5) node{(1)};
\end{tikzpicture}
\ \ \ \
\begin{tikzpicture}[scale=0.6]
\draw[line width=1.2pt]  (2,0) arc (0:360:2);
\draw[line width=2pt][white]  (1.414,-1.414) arc (-45:-135:2);
\draw[line width=1.2pt][dotted]  (1.414,-1.414) arc (-45:-135:2);
\filldraw[qqwuqq] ( 0.00, 2.00) circle (0.1);
\draw[ffqqqq][line width=1.2pt] (-1.56, 1.25) circle (0.1);
\filldraw[qqwuqq] (-1.95, 0.45) circle (0.1);
\filldraw[qqwuqq] (-1.95,-0.45) circle (0.1);
\filldraw[qqwuqq] (-1.61,-1.19) circle (0.1);
\filldraw[qqwuqq] (-0.87,-1.80) circle (0.1);
\filldraw[qqwuqq] (-0.00,-2.00) circle (0.1);
\filldraw[qqwuqq] ( 0.87,-1.80) circle (0.1);
\filldraw[qqwuqq] ( 1.61,-1.19) circle (0.1);
\filldraw[qqwuqq] ( 1.94,-0.45) circle (0.1);
\filldraw[qqwuqq] ( 1.95, 0.45) circle (0.1);
\draw[ffqqqq][line width=1.2pt] ( 1.56, 1.25) circle (0.1);
\draw[qqwuqq!25][line width=1.2pt] ( 0.00, 2.00) -- (-1.95, 0.45);
\draw[qqwuqq!25][line width=1.2pt] ( 0.00, 2.00) -- (-1.95,-0.45);
\draw[qqwuqq!25][line width=1.2pt] ( 0.00, 2.00) -- (-1.61,-1.19);
\draw[qqwuqq!25][line width=1.2pt][dotted] ( 0.00, 2.00) -- (-0.87,-1.80);
\draw[qqwuqq!25][line width=1.2pt][dotted] ( 0.00, 2.00) -- (-0.00,-2.00);
\draw[qqwuqq!25][line width=1.2pt][dotted] ( 0.00, 2.00) -- ( 0.87,-1.80);
\draw[qqwuqq!25][line width=1.2pt] ( 0.00, 2.00) -- ( 1.61,-1.19);
\draw[qqwuqq!25][line width=1.2pt] ( 0.00, 2.00) -- ( 1.94,-0.45);
\draw[qqwuqq!25][line width=1.2pt] ( 0.00, 2.00) -- ( 1.95, 0.45);
\draw[qqwuqq] (0,2) node[above]{$p$};
\draw[blue][line width=1.2pt] (0,2) -- (-1.61,-1.19) -- (1.61,-1.19) -- (0,2);
\draw[blue](0,0) node{$\T$};
\draw (0,-2.5) node{(2)};
\end{tikzpicture}
\ \ \ \
\begin{tikzpicture}[scale=0.6]
\draw[line width=1.2pt]  (2,0) arc (0:360:2);
\draw[line width=2pt][white]  (1.414,-1.414) arc (-45:-135:2);
\draw[line width=1.2pt][dotted]  (1.414,-1.414) arc (-45:-135:2);
\filldraw[qqwuqq] ( 0.00, 2.00) circle (0.1);
\draw[ffqqqq][line width=1.2pt] (-1.56, 1.25) circle (0.1);
\filldraw[qqwuqq] (-1.95, 0.45) circle (0.1);
\filldraw[qqwuqq] (-1.95,-0.45) circle (0.1);
\filldraw[qqwuqq] (-1.61,-1.19) circle (0.1);
\filldraw[qqwuqq] (-0.87,-1.80) circle (0.1);
\filldraw[qqwuqq] (-0.00,-2.00) circle (0.1);
\filldraw[qqwuqq] ( 0.87,-1.80) circle (0.1);
\filldraw[qqwuqq] ( 1.61,-1.19) circle (0.1);
\filldraw[qqwuqq] ( 1.94,-0.45) circle (0.1);
\filldraw[qqwuqq] ( 1.95,-0.45) circle (0.1);
\filldraw[qqwuqq] ( 1.95, 0.45) circle (0.1);
\draw[ffqqqq][line width=1.2pt] ( 1.56, 1.25) circle (0.1);
\draw[qqwuqq!50][line width=1.2pt] ( 0.00, 2.00) -- (-1.95, 0.45);
\draw[qqwuqq!50][line width=1.2pt] ( 0.00, 2.00) -- (-1.95,-0.45);
\draw[qqwuqq!50][line width=1.2pt] ( 0.00, 2.00) -- (-1.61,-1.19);
\draw[qqwuqq!50][line width=1.2pt][dotted] ( 0.00, 2.00) -- (-0.87,-1.80);
\draw[qqwuqq!50][line width=1.2pt][dotted] ( 0.00, 2.00) -- (-0.00,-2.00);
\draw[qqwuqq!50][line width=1.2pt][dotted] ( 0.00, 2.00) -- ( 0.87,-1.80);
\draw[qqwuqq!50][line width=1.2pt] ( 0.00, 2.00) -- ( 1.61,-1.19);
\draw[qqwuqq!50][line width=1.2pt] ( 0.00, 2.00) -- ( 1.94,-0.45);
\draw[qqwuqq!50][line width=1.2pt] ( 0.00, 2.00) -- ( 1.95, 0.45);
\draw[qqwuqq] (0,2) node[above]{$p$};
\draw[blue][line width=1.2pt] ( 0.00, 2.00) -- (-1.95, 0.45) to[out=78,in=180] ( 0.00, 2.00);
\draw[blue] (-1.16, 1.35) node{$\T$};
\draw (0,-2.5) node{(3)};
\end{tikzpicture}
\caption{\textsf{The tile $\T$ with vertex $p$.}}
\label{fig-tile with vertex p}
\end{center}
\end{figure}
\noindent Thus, there exists at least one element in $\Gamma$ which belongs to $\gbullet$-FFAS $\Dgreen(A_n)$.
In this case, $|T|\le \sharp\Gamma-1 = |A|-1$, this is a contradiction.
\end{proof}

\begin{corollary} \label{coro-tilted}
Let $\Gamma\in \Tri(\SURFextra(A_n))$.
Then $A^{\Gamma}$ is a tilted algebra of type $A_{n}$ if and only if $\Gamma$ is a {\rm GD} containing $c_{P(1)}$.
\end{corollary}

\begin{proof}
By Theorem \ref{thm-endo. alg.}, we have $\End(\MM(\Gamma))$ is a tilted algbera.
By Lemma \ref{lemm-invariant of tilt}, $P(1)=\MM(c_{P(1)})$ is a direct summand of $\MM(\Gamma)$.
Thus $c_{P(1)}\in\Gamma$.

If $c_{P(1)}\in \Gamma$, then, by Lemma \ref{lemm-invariant of tilt}, any permissible curve in $\Gamma$ is non-trivial.
By Corollary \ref{coro-triangulation}, $\MM(\Gamma)$ is a $\tau$-tilting module, and then $A^{\Gamma}$ is tilted.
\end{proof}

\begin{proposition} \label{prop-non-hereditary}
Let $\Gamma$ be a $\mathrm{GD}$ and $c_{P(1)}\in\Gamma$.

{\rm(1)} If $\Tri_{\Gamma}(\SURFextra(A_n))$ contains no complete triangle, then $A^{\Gamma}$ is a hereditary tilted algebra of type $A_{n}$.

{\rm(2)} If $\Tri_{\Gamma}(\SURFextra(A_n))$ contains at least one complete triangle, then $A^{\Gamma}$ is a non-hereditary tilted algebra of type $A_{n}$.
\end{proposition}

\begin{proof}
By Corollary \ref{coro-tilted}, we have $A^{\Gamma}$ is tilted because $c_{P(1)}\in\Gamma$.
Therefore, if $\T\in\Tri_{\Gamma}(\SURFextra(A_n))$ is a complete triangle,
then it is one of the forms (1), (2), (3) and (4) shown in \Pic \ref{fig-special tile}.
\begin{figure}[H]
\begin{center}
\definecolor{ffqqqq}{rgb}{1,0,0}
\definecolor{qqwuqq}{rgb}{0,0.5,0}
\begin{tikzpicture}[scale=0.6]
\draw[line width=1.2pt] (2,0) arc (0:360:2);
\draw[line width=2pt][white] (1.414,-1.414) arc (-45:-80:2);
\draw[line width=1.2pt][dotted] (1.414,-1.414) arc (-45:-80:2);
\draw[line width=2pt][white] (-1.414,-1.414) arc (-135:-110:2);
\draw[line width=1.2pt][dotted] (-1.414,-1.414) arc (-135:-110:2);
\filldraw[qqwuqq] ( 0.00, 2.00) circle (0.1);
\draw[ffqqqq][line width=1.2pt] (-1.56, 1.25) circle (0.1);
\filldraw[qqwuqq] (-1.95, 0.45) circle (0.1);
\filldraw[qqwuqq] (-1.95,-0.45) circle (0.1);
\filldraw[qqwuqq] (-1.61,-1.19) circle (0.1);
\filldraw[qqwuqq] (-0.87,-1.80) circle (0.1);
\filldraw[qqwuqq] (-0.00,-2.00) circle (0.1);
\filldraw[qqwuqq] ( 0.87,-1.80) circle (0.1);
\filldraw[qqwuqq] ( 1.61,-1.19) circle (0.1);
\filldraw[qqwuqq] ( 1.94,-0.45) circle (0.1);
\filldraw[qqwuqq] ( 1.95, 0.45) circle (0.1);
\draw[ffqqqq][line width=1.2pt] ( 1.56, 1.25) circle (0.1);
\draw[qqwuqq!25][line width=1.2pt] ( 0.00, 2.00) -- (-1.95, 0.45);
\draw[qqwuqq!25][line width=1.2pt] ( 0.00, 2.00) -- (-1.95,-0.45);
\draw[qqwuqq!25][line width=1.2pt][dotted] ( 0.00, 2.00) -- (-1.61,-1.19);
\draw[qqwuqq!25][line width=1.2pt][dotted] ( 0.00, 2.00) -- (-0.87,-1.80);
\draw[qqwuqq!25][line width=1.2pt] ( 0.00, 2.00) -- (-0.00,-2.00);
\draw[qqwuqq!25][line width=1.2pt][dotted] ( 0.00, 2.00) -- ( 0.87,-1.80);
\draw[qqwuqq!25][line width=1.2pt][dotted] ( 0.00, 2.00) -- ( 1.61,-1.19);
\draw[qqwuqq!25][line width=1.2pt] ( 0.00, 2.00) -- ( 1.94,-0.45);
\draw[qqwuqq!25][line width=1.2pt] ( 0.00, 2.00) -- ( 1.95, 0.45);
\draw[qqwuqq] (0,2) node[above]{$p$};
\draw[blue][line width=1.2pt] (-1.56, 1.25) -- (0,-2) -- (1.56, 1.25) -- (-1.56, 1.25);
\draw[blue](0,0) node{$\T$};
\draw[orange][line width=1.2pt] (-1.56, 1.25)--(1.56, 1.25);
\draw (0,-2.5) node{(1) $\T$ has two};
\draw (0,-3) node{vertices lying in $\E$};
\end{tikzpicture}
\begin{tikzpicture}[scale=0.6]
\draw[line width=1.2pt] (2,0) arc (0:360:2);
\draw[line width=2pt][white] (1.414,-1.414) arc (-45:-80:2);
\draw[line width=1.2pt][dotted] (1.414,-1.414) arc (-45:-80:2);
\draw[line width=2pt][white] (-1.414,-1.414) arc (-135:-100:2);
\draw[line width=1.2pt][dotted] (-1.414,-1.414) arc (-135:-100:2);
\draw[line width=2pt][white] (-2,0) arc (-180:-155:2);
\draw[line width=1.2pt][dotted] (-2,0) arc (-180:-155:2);
\filldraw[qqwuqq] ( 0.00, 2.00) circle (0.1);
\draw[ffqqqq][line width=1.2pt] (-1.56, 1.25) circle (0.1);
\filldraw[qqwuqq] (-1.95, 0.45) circle (0.1);
\filldraw[qqwuqq] (-1.95,-0.45) circle (0.1);
\filldraw[qqwuqq] (-1.61,-1.19) circle (0.1);
\filldraw[qqwuqq] (-0.87,-1.80) circle (0.1);
\filldraw[qqwuqq] (-0.00,-2.00) circle (0.1);
\filldraw[qqwuqq] ( 0.87,-1.80) circle (0.1);
\filldraw[qqwuqq] ( 1.61,-1.19) circle (0.1);
\filldraw[qqwuqq] ( 1.94,-0.45) circle (0.1);
\filldraw[qqwuqq] ( 1.95, 0.45) circle (0.1);
\draw[ffqqqq][line width=1.2pt] ( 1.56, 1.25) circle (0.1);
\draw[qqwuqq!25][line width=1.2pt] ( 0.00, 2.00) -- (-1.95, 0.45);
\draw[qqwuqq!25][line width=1.2pt][dotted] ( 0.00, 2.00) -- (-1.95,-0.45);
\draw[qqwuqq!25][line width=1.2pt] ( 0.00, 2.00) -- (-1.61,-1.19);
\draw[qqwuqq!25][line width=1.2pt][dotted] ( 0.00, 2.00) -- (-0.87,-1.80);
\draw[qqwuqq!25][line width=1.2pt] ( 0.00, 2.00) -- (-0.00,-2.00);
\draw[qqwuqq!25][line width=1.2pt][dotted] ( 0.00, 2.00) -- ( 0.87,-1.80);
\draw[qqwuqq!25][line width=1.2pt][dotted] ( 0.00, 2.00) -- ( 1.61,-1.19);
\draw[qqwuqq!25][line width=1.2pt] ( 0.00, 2.00) -- ( 1.94,-0.45);
\draw[qqwuqq!25][line width=1.2pt] ( 0.00, 2.00) -- ( 1.95, 0.45);
\draw[qqwuqq] (0,2) node[above]{$p$};
\draw[blue][line width=1.2pt] (-1.56,-1.25) -- (0,-2) -- (1.56, 1.25) -- (-1.56,-1.25);
\draw[blue](0,-0.8) node{$\T$};
\draw[orange][line width=1.2pt] (-1.56, 1.25)--(1.56, 1.25);
\draw (0,-2.5) node{(2) $\T$ has only one};
\draw (0,-3) node{vertex lying in $\E$};
\end{tikzpicture}
\begin{tikzpicture}[scale=0.6]
\draw[line width=1.2pt] (2,0) arc (0:360:2);
\draw[line width=2pt][white] (1.414,-1.414) arc (-45:-80:2);
\draw[line width=1.2pt][dotted] (1.414,-1.414) arc (-45:-80:2);
\draw[line width=2pt][white] (-1.414,-1.414) arc (-135:-100:2);
\draw[line width=1.2pt][dotted] (-1.414,-1.414) arc (-135:-100:2);
\draw[line width=2pt][white] (2,0) arc (0:-25:2);
\draw[line width=1.2pt][dotted] (2,0) arc (0:-25:2);
\filldraw[qqwuqq] ( 0.00, 2.00) circle (0.1);
\draw[ffqqqq][line width=1.2pt] (-1.56, 1.25) circle (0.1);
\filldraw[qqwuqq] (-1.95, 0.45) circle (0.1);
\filldraw[qqwuqq] (-1.95,-0.45) circle (0.1);
\filldraw[qqwuqq] (-1.61,-1.19) circle (0.1);
\filldraw[qqwuqq] (-0.87,-1.80) circle (0.1);
\filldraw[qqwuqq] (-0.00,-2.00) circle (0.1);
\filldraw[qqwuqq] ( 0.87,-1.80) circle (0.1);
\filldraw[qqwuqq] ( 1.61,-1.19) circle (0.1);
\filldraw[qqwuqq] ( 1.94,-0.45) circle (0.1);
\filldraw[qqwuqq] ( 1.95, 0.45) circle (0.1);
\draw[ffqqqq][line width=1.2pt] ( 1.56, 1.25) circle (0.1);
\draw[qqwuqq!25][line width=1.2pt] ( 0.00, 2.00) -- (-1.95, 0.45);
\draw[qqwuqq!25][line width=1.2pt] ( 0.00, 2.00) -- (-1.95,-0.45);
\draw[qqwuqq!25][line width=1.2pt] ( 0.00, 2.00) -- (-1.61,-1.19);
\draw[qqwuqq!25][line width=1.2pt][dotted] ( 0.00, 2.00) -- (-0.87,-1.80);
\draw[qqwuqq!25][line width=1.2pt] ( 0.00, 2.00) -- (-0.00,-2.00);
\draw[qqwuqq!25][line width=1.2pt][dotted] ( 0.00, 2.00) -- ( 0.87,-1.80);
\draw[qqwuqq!25][line width=1.2pt] ( 0.00, 2.00) -- ( 1.61,-1.19);
\draw[qqwuqq!25][line width=1.2pt][dotted] ( 0.00, 2.00) -- ( 1.94,-0.45);
\draw[qqwuqq!25][line width=1.2pt] ( 0.00, 2.00) -- ( 1.95, 0.45);
\draw[qqwuqq] (0,2) node[above]{$p$};
\draw[blue][line width=1.2pt] (1.56,-1.25) -- (0,-2) -- (-1.56, 1.25) -- (1.56,-1.25);
\draw[blue](0,-0.8) node{$\T$};
\draw[orange][line width=1.2pt] (-1.56, 1.25)--(1.56, 1.25);
\draw (0,-2.5) node{(3) $\T$ has only one};
\draw (0,-3) node{vertex lying in $\E$};
\end{tikzpicture}
\begin{tikzpicture}[scale=0.6]
\draw[line width=1.2pt] (2,0) arc (0:360:2);
\draw[line width=2pt][white] (1.414,-1.414) arc (-45:-80:2);
\draw[line width=1.2pt][dotted] (1.414,-1.414) arc (-45:-80:2);
\draw[line width=2pt][white] (-1.414,-1.414) arc (-135:-100:2);
\draw[line width=1.2pt][dotted] (-1.414,-1.414) arc (-135:-100:2);
\draw[line width=2pt][white] (-2,0) arc (-180:-155:2);
\draw[line width=1.2pt][dotted] (-2,0) arc (-180:-155:2);
\draw[line width=2pt][white] (2,0) arc (0:-25:2);
\draw[line width=1.2pt][dotted] (2,0) arc (0:-25:2);
\filldraw[qqwuqq] ( 0.00, 2.00) circle (0.1);
\draw[ffqqqq][line width=1.2pt] (-1.56, 1.25) circle (0.1);
\filldraw[qqwuqq] (-1.95, 0.45) circle (0.1);
\filldraw[qqwuqq] (-1.95,-0.45) circle (0.1);
\filldraw[qqwuqq] (-1.61,-1.19) circle (0.1);
\filldraw[qqwuqq] (-0.87,-1.80) circle (0.1);
\filldraw[qqwuqq] (-0.00,-2.00) circle (0.1);
\filldraw[qqwuqq] ( 0.87,-1.80) circle (0.1);
\filldraw[qqwuqq] ( 1.61,-1.19) circle (0.1);
\filldraw[qqwuqq] ( 1.94,-0.45) circle (0.1);
\filldraw[qqwuqq] ( 1.95, 0.45) circle (0.1);
\draw[ffqqqq][line width=1.2pt] ( 1.56, 1.25) circle (0.1);
\draw[qqwuqq!25][line width=1.2pt] ( 0.00, 2.00) -- (-1.95, 0.45);
\draw[qqwuqq!25][line width=1.2pt][dotted] ( 0.00, 2.00) -- (-1.95,-0.45);
\draw[qqwuqq!25][line width=1.2pt] ( 0.00, 2.00) -- (-1.61,-1.19);
\draw[qqwuqq!25][line width=1.2pt][dotted] ( 0.00, 2.00) -- (-0.87,-1.80);
\draw[qqwuqq!25][line width=1.2pt] ( 0.00, 2.00) -- (-0.00,-2.00);
\draw[qqwuqq!25][line width=1.2pt][dotted] ( 0.00, 2.00) -- ( 0.87,-1.80);
\draw[qqwuqq!25][line width=1.2pt] ( 0.00, 2.00) -- ( 1.61,-1.19);
\draw[qqwuqq!25][line width=1.2pt][dotted] ( 0.00, 2.00) -- ( 1.94,-0.45);
\draw[qqwuqq!25][line width=1.2pt] ( 0.00, 2.00) -- ( 1.95, 0.45);
\draw[qqwuqq] (0,2) node[above]{$p$};
\draw[blue][line width=1.2pt] (-1.56,-1.25) -- (0,-2) -- (1.56,-1.25) -- (-1.56,-1.25);
\draw[blue](0,-1.7) node{$\T$};
\draw[orange][line width=1.2pt] (-1.56, 1.25)--(1.56, 1.25);
\draw (0,-2.5) node{(4) All vertices of $\T$};
\draw (0,-3) node{lying in $\M$};
\end{tikzpicture}
\caption{\textsf{The complete triangle $\T\in\Tri_{\Gamma}(\SURFextra(A_n))$,
where $\Gamma$ is a GD containing $c_{P(1)}$.
In (1), the permissible curve $c_{P(1)}$ coincides with one edge of $\T$}. }
\label{fig-special tile}
\end{center}
\end{figure}
\noindent By Proposition \ref{prop-hereditary},
$A^{\Gamma}$ is non-hereditary if and only if $\Tri_{\Gamma}(\SURFextra(A))$ contains complete triangle.
\end{proof}

Now we can give a classification of the non-hereditary tilted algebras of type $A_n$.

\begin{theorem} \label{prop-tri. contains sp. tile}
There is a bijection
\[ \Tri_{\rmc}(\SURFextra(A_n)) \to \sfT_{\nh}(A_n), \\
\Gamma\mapsto A^{\Gamma}\]
between the set $\Tri_{\rmc}(\SURFextra(A_n))$ of all triangulations in $\Tri(\SURFextra(A_n))$ which contains $c_{P(1)}$ and has at least one complete triangle and the set $\sfT_{\nh}(A_n)$ of all non-hereditary tilted algebras of type $A_n$.

\end{theorem}

\begin{proof}
Assume that $\Gamma_1, \Gamma_2 \in \Tri_{\rmc}(\SURFextra(A_n))$ and $A^{\Gamma_1}\cong A^{\Gamma_2}$.
By Proposition \ref{prop-non-hereditary}, $A^{\Gamma_1}$ and $A^{\Gamma_2}$ are non-hereditary tilted algebras of type $A_n$.
Now we show that $\Gamma_1=\Gamma_2$.

For $n=3$, it is easy to check that the number of non-hereditary tilted algebra is $1$.
For $n=4$, by Example \ref{exm-A4}, this statement also holds.
\begin{figure}[H]
\centering
\definecolor{ffqqqq}{rgb}{1,0,0}
\definecolor{qqwuqq}{rgb}{0,0.5,0}
\begin{tikzpicture}[scale=0.9]
\draw[line width=1.2pt] (0,0) circle (2);
\draw[line width=2pt][white] ( 0.35,-1.97) arc(-80:-40:2);
\draw[line width=1.2pt][dotted] ( 0.35,-1.97) arc(-80:-40:2);
\draw[line width=2pt][white] (-0.35,-1.97) arc(-100:-140:2);
\draw[line width=1.2pt][dotted]  (-0.35,-1.97) arc(-100:-140:2);
\draw[line width=2pt][white] ( 1.88, 0.68) arc(20:40:2);
\draw[line width=1.2pt][dotted] ( 1.88, 0.68) arc(20:40:2);
\draw[line width=2pt][white] (-1.88, 0.68) arc(160:140:2);
\draw[line width=1.2pt][dotted] (-1.88, 0.68) arc(160:140:2);
%
\draw[qqwuqq!50][line width=1.2pt][dotted] ( 0.00, 2.00) -- ( 1.73, 1.00);
\draw[qqwuqq!50][line width=1.2pt] ( 0.00, 2.00) -- ( 1.28, 1.53);
\draw[qqwuqq!50][line width=1.2pt] ( 0.00, 2.00) -- ( 0.00, 2.00);
\draw[qqwuqq!50][line width=1.2pt] ( 0.00, 2.00) -- (-1.28, 1.53);
\draw[qqwuqq!50][line width=1.2pt][dotted] ( 0.00, 2.00) -- (-1.73, 1.00);
\draw[qqwuqq!50][line width=1.2pt] ( 0.00, 2.00) -- (-1.97, 0.35);
\draw[qqwuqq!50][line width=1.2pt] ( 0.00, 2.00) -- (-1.73,-1.00);
\draw[qqwuqq!50][line width=1.2pt][dotted] ( 0.00, 2.00) -- (-1.00,-1.73);
\draw[qqwuqq!50][line width=1.2pt] ( 0.00, 2.00) -- ( 0.00,-2.00);
\draw[qqwuqq!50][line width=1.2pt][dotted] ( 0.00, 2.00) -- ( 1.00,-1.73);
\draw[qqwuqq!50][line width=1.2pt] ( 0.00, 2.00) -- ( 1.73,-1.00);
\draw[qqwuqq!50][line width=1.2pt] ( 0.00, 2.00) -- ( 1.97, 0.35);
\filldraw[qqwuqq] ( 1.97, 0.35) circle (0.1);
\filldraw[qqwuqq] ( 1.73, 1.00) circle (0.1);
\filldraw[qqwuqq] ( 1.28, 1.53) circle (0.1);
\filldraw[qqwuqq] ( 0.00, 2.00) circle (0.1);
\filldraw[qqwuqq] (-1.28, 1.53) circle (0.1);
\filldraw[qqwuqq] (-1.73, 1.00) circle (0.1);
\filldraw[qqwuqq] (-1.97, 0.35) circle (0.1);
\filldraw[qqwuqq] ( 1.73,-1.00) circle (0.1);
\filldraw[qqwuqq] (-1.00,-1.73) circle (0.1);
\filldraw[qqwuqq] ( 0.00,-2.00) circle (0.1);
\filldraw[qqwuqq] ( 1.00,-1.73) circle (0.1);
\filldraw[qqwuqq] (-1.73,-1.00) circle (0.1);
\draw[red][line width=1.2pt] (-0.68, 1.88) circle (0.1);
\draw[red][line width=1.2pt] ( 0.68, 1.88) circle (0.1);
\draw[blue] (0,-0.4) node{$\T$};
\draw[blue][line width=1.2pt] (-1.97, 0.35) -- ( 1.97, 0.35); \draw[blue] ( 0.00, 0.35) node[above]{$c_1$};
\draw[blue][line width=1.2pt] ( 0.00,-2.00) -- (-1.97, 0.35); \draw[blue] (-0.86,-1.00) node[left]{$c_2$};
\draw[blue][line width=1.2pt] ( 0.00,-2.00) -- ( 1.97, 0.35); \draw[blue] ( 0.86,-1.00) node[right]{$c_3$};
\draw[blue][<-] (-0.40, 0.19)--(-0.80,-0.50); \draw[blue] (-0.60,0.34-0.5) node[left]{$\alpha$};
\draw[blue][<-] ( 0.80,-0.50)--( 0.40, 0.19); \draw[blue] ( 0.60,0.34-0.5) node[right]{$\beta$};
\draw[orange][line width=1.2pt] (-0.68, 1.88)--( 0.68, 1.88);
\end{tikzpicture}
\caption{\textsf{The complete triangle $\T$ belongs to $\Tri_{\Gamma}(\SURFextra(A_n))$,
where $c_{P(1)}\in\Gamma$.}}
\label{fig-complete tile 2}
\end{figure}
Assume that this statement holds for all $n\le k-1$ ($k\ge 4$), we need to prove that it holds for $n=k$.
Let $A^{\Gamma_1}=\kk\Q^1/\I^1$ and $A^{\Gamma_2}=\kk\Q^2/\I^2$.
Since $A^{\Gamma_1}\cong A^{\Gamma_2}$, we can suppose that $\Q^1=\Q^2$ and $\I^1=\I^2$.
Thus $\Q^1=\Q^2$ contains a subquiver
\begin{center}
$\Q^{\T}:=\xymatrix@C=0.5cm{c_2 \ar[r]_{\alpha} \ar@{.}@/^0.6pc/[rr] & c_1 \ar[r]_{\beta} & c_3}$ with $\alpha\beta\in\I^1=\I^2$,
\end{center}
which corresponding to the complete triangle $\T\in \Tri_{\Gamma_1}(\SURFextra(A_n))\cap \Tri_{\Gamma_2}(\SURFextra(A_n))$
(see the blue triangle shown in \Pic \ref{fig-complete tile 2}).
Then $\Q^1$ and $\Q^2$ are of the form shown in \Pic \ref{fig-quiver}.
\begin{figure}[H]
\centering
\definecolor{ffqqqq}{rgb}{1,0,0}
\definecolor{qqwuqq}{rgb}{0,0.5,0}
\begin{tikzpicture}
\draw (0,0) node{$
\xymatrix@C=0.25cm@R=0.3cm{ \cdots \ar@{-}[rd] & &\cdots\ar@{-}[rd]& &\cdots\ar@{-}[ld]& &  \ar@{-}[ld]\cdots \\
\cdots &  c_2 \ar[rr]_{\alpha} \ar@{-}[l]\ar@{.}@/_0.6pc/[rrrr] && c_1 \ar[rr]_{\beta} && c_3 \ar@{-}[r] & \cdots \\
\cdots \ar@{-}[ru] & && && & \cdots\ar@{-}[lu]}
$};
\shade[top color=blue, bottom color=cyan][opacity=0.5]
    (-2,0) to[out=90,in=180] (0,0.3) to[out=0,in=90] (2,0)  to[out=-90,in=0] (0,-0.3) to[out=180,in=-90] (-2,0);
\draw (-0.2,-0.2) node[below]{$\Q^{\T}$};
\draw[yellow][line width=1.5pt][opacity=0.5]
    (-3,0) to[out=90,in=180] (-2.5,1) to[out=0,in=90] (-1.1,0) to[out=-90,in=0] (-2.5,-1) to[out=180,in=-90] (-3,0);
\draw (-3,0) node[left]{$\Q^{i,L}$};
\draw[cyan][line width=1.5pt][opacity=0.5]
    (3,0) to[out=90,in=0] (2.5,1) to[out=180,in=90] (1.1,0) to[out=-90,in=180] (2.5,-1) to[out=0,in=-90] (3,0);
\draw (3,0) node[right]{$\Q^{i,R}$};
\draw[purple][line width=1.5pt][rounded corners = 0.3cm][opacity=0.5]
    (0,-0.2) -- (-1.3, 0.8) -- (1.3, 0.8) -- (0,-0.2);
\draw (0,0.8) node[above]{$\Q^{i,U}$};
\end{tikzpicture}
\caption{\textsf{The quiver $(\Q^i,\I^i)$ for $i\in\{1,2\}$.}}
\label{fig-quiver}
\end{figure}

Let $\SURFextra(A_n)$ be the marked ribbon surface of $A_n$ as shown in \Pic \ref{fig-complete tile 3}.
\begin{figure}[htbp]
\centering
\definecolor{ffqqqq}{rgb}{1,0,0}
\definecolor{qqwuqq}{rgb}{0,0.5,0}
\begin{tikzpicture}[scale=0.9]
\draw[line width=1.2pt] ( 1.97, 0.35) arc(10:170:2);
\draw[line width=2pt][white] ( 1.88, 0.68) arc(20:40:2);
\draw[line width=1.2pt][dotted] ( 1.88, 0.68) arc(20:40:2);
\draw[line width=2pt][white] (-1.88, 0.68) arc(160:140:2);
\draw[line width=1.2pt][dotted] (-1.88, 0.68) arc(160:140:2);
\draw[line width=1.2pt] (-1.97, 0.35) -- ( 1.97, 0.35);
\filldraw[qqwuqq] ( 1.97, 0.35) circle (0.1); \draw ( 1.97, 0.35) node[right]{$q_2$};
\filldraw[qqwuqq] ( 1.73, 1.00) circle (0.1);
\filldraw[qqwuqq] ( 1.28, 1.53) circle (0.1);
\filldraw[qqwuqq] ( 0.00, 2.00) circle (0.1); \draw ( 0.00, 2.00) node[above]{$p$};
\filldraw[qqwuqq] (-1.28, 1.53) circle (0.1);
\filldraw[qqwuqq] (-1.73, 1.00) circle (0.1);
\filldraw[qqwuqq] (-1.97, 0.35) circle (0.1); \draw (-1.97, 0.35) node[left]{$q_3$};
\draw[red][line width=1.2pt] (-0.68, 1.88) circle (0.1);
\draw[red][line width=1.2pt] ( 0.68, 1.88) circle (0.1);
\draw[qqwuqq][line width=1.2pt] ( 0.00, 2.00) -- ( 1.97, 0.35);
\draw[qqwuqq][line width=1.2pt] ( 0.00, 2.00) -- ( 1.73, 1.00);
\draw[qqwuqq][line width=1.2pt] ( 0.00, 2.00) -- ( 1.28, 1.53);
\draw[qqwuqq][line width=1.2pt] ( 0.00, 2.00) -- (-1.97, 0.35);
\draw[qqwuqq][line width=1.2pt] ( 0.00, 2.00) -- (-1.73, 1.00);
\draw[qqwuqq][line width=1.2pt] ( 0.00, 2.00) -- (-1.28, 1.53);
(-1.97, 0.35) to[out=77,in=180] (0,2) to[out=0,in=103] ( 1.97, 0.35) -- (-1.97, 0.35);
\draw (0,0) node{${_U\SURFextra}$};
\draw[orange][line width=1.2pt] (-0.68, 1.88)--( 0.68, 1.88);
\end{tikzpicture}
\\
\begin{tikzpicture}[scale=0.9]
\draw[line width=1.2pt]  ( 0.00,-2.00) -- ( 0.00, 2.00) -- (-1.97, 0.35);
\draw[line width=1.2pt]  ( 0.00,-2.00) arc(-90:-190:2);
\draw[line width=2pt][white] (-0.35,-1.97) arc(-100:-140:2);
\draw[line width=1.2pt][dotted]  (-0.35,-1.97) arc(-100:-140:2);
\filldraw[qqwuqq] ( 0.00, 2.00) circle (0.1); \draw ( 0.00, 2.00) node[above]{$p$};
\draw[red][line width=1pt] (-1.97, 0.35) circle (0.1); \draw (-1.97, 0.35) node[left]{$q_3$};
\filldraw[qqwuqq] (-1.73,-1.00) circle (0.1);
\filldraw[qqwuqq] (-1.00,-1.73) circle (0.1);
\draw[red][line width=1pt] ( 0.00,-2.00) circle (0.1); \draw ( 0.00,-2.00) node[below]{$q_1$};
\draw[qqwuqq][line width=1pt] ( 0.00, 2.00) -- (-1.73,-1.00);
\draw[qqwuqq][line width=1pt][dotted] ( 0.00, 2.00) -- (-1.00,-1.73);
\draw[blue][line width=1.2pt] (-1.97, 0.35)--(0,-2);
\draw (0,-2.8) node{${_L\SURFextra}$};
\end{tikzpicture}
\ \ \ \
\begin{tikzpicture}[scale=0.9]
\draw[line width=1.2pt]  (0,0) circle (2);
\draw[line width=2pt][white] ( 1.88, 0.68) arc(20:40:2);
\draw[line width=1.2pt][dotted] ( 1.88, 0.68) arc(20:40:2);
\draw[line width=2pt][white] (-1.88, 0.68) arc(160:140:2);
\draw[line width=1.2pt][dotted] (-1.88, 0.68) arc(160:140:2);
\draw[line width=2pt][white] ( 0.35,-1.97) arc(-80:-40:2);
\draw[line width=1.2pt][dotted] ( 0.35,-1.97) arc(-80:-40:2);
\draw[line width=2pt][white] (-0.35,-1.97) arc(-100:-140:2);
\draw[line width=1.2pt][dotted]  (-0.35,-1.97) arc(-100:-140:2);
\draw[qqwuqq!50][line width=1.2pt][dotted] ( 0.00, 2.00) -- ( 1.73, 1.00);
\draw[qqwuqq!50][line width=1.2pt] ( 0.00, 2.00) -- ( 1.28, 1.53);
\draw[qqwuqq!50][line width=1.2pt] ( 0.00, 2.00) -- ( 0.00, 2.00);
\draw[qqwuqq!50][line width=1.2pt] ( 0.00, 2.00) -- (-1.28, 1.53);
\draw[qqwuqq!50][line width=1.2pt][dotted] ( 0.00, 2.00) -- (-1.73, 1.00);
\draw[qqwuqq!50][line width=1.2pt] ( 0.00, 2.00) -- (-1.97, 0.35);
\draw[qqwuqq!50][line width=1.2pt] ( 0.00, 2.00) -- (-1.73,-1.00);
\draw[qqwuqq!50][line width=1.2pt][dotted] ( 0.00, 2.00) -- (-1.00,-1.73);
\draw[qqwuqq!50][line width=1.2pt] ( 0.00, 2.00) -- ( 0.00,-2.00);
\draw[qqwuqq!50][line width=1.2pt][dotted] ( 0.00, 2.00) -- ( 1.00,-1.73);
\draw[qqwuqq!50][line width=1.2pt] ( 0.00, 2.00) -- ( 1.73,-1.00);
\draw[qqwuqq!50][line width=1.2pt] ( 0.00, 2.00) -- ( 1.97, 0.35);
\filldraw[qqwuqq] ( 1.97, 0.35) circle (0.1); \draw ( 1.97, 0.35) node[right]{$q_2$};
\filldraw[qqwuqq] ( 1.73, 1.00) circle (0.1);
\filldraw[qqwuqq] ( 1.28, 1.53) circle (0.1);
\filldraw[qqwuqq] ( 0.00, 2.00) circle (0.1); \draw ( 0.00, 2.00) node[above]{$p$};
\filldraw[qqwuqq] (-1.28, 1.53) circle (0.1);
\filldraw[qqwuqq] (-1.73, 1.00) circle (0.1);
\filldraw[qqwuqq] (-1.97, 0.35) circle (0.1); \draw (-1.97, 0.35) node[left]{$q_3$};
\filldraw[qqwuqq] ( 1.73,-1.00) circle (0.1);
\filldraw[qqwuqq] (-1.00,-1.73) circle (0.1);
\filldraw[qqwuqq] ( 0.00,-2.00) circle (0.1); \draw ( 0.00,-2.00) node[below]{$q_1$};
\filldraw[qqwuqq] ( 1.00,-1.73) circle (0.1);
\filldraw[qqwuqq] (-1.73,-1.00) circle (0.1);
\draw[red][line width=1.2pt] (-0.68, 1.88) circle (0.1);
\draw[red][line width=1.2pt] ( 0.68, 1.88) circle (0.1);
\draw[blue] (0,-0.4) node{$\T$};
(-1.97, 0.35) to[out=77,in=180] (0,2) to[out=0,in=103] ( 1.97, 0.35) -- (-1.97, 0.35);
\draw[blue][line width=1.2pt] (-1.97, 0.35) -- ( 1.97, 0.35); \draw[blue] ( 0.00, 0.35) node[above]{$c_1$};
\draw[blue][line width=1.2pt] ( 0.00,-2.00) -- (-1.97, 0.35); \draw[blue] (-0.86,-1.00) node[left]{$c_2$};
\draw[blue][line width=1.2pt] ( 0.00,-2.00) -- ( 1.97, 0.35); \draw[blue] ( 0.86,-1.00) node[right]{$c_3$};
\draw[blue][<-] (-0.40, 0.19)--(-0.80,-0.50); \draw[blue] (-0.60,0.34-0.5) node[left]{$\alpha$};
\draw[blue][<-] ( 0.80,-0.50)--( 0.40, 0.19); \draw[blue] ( 0.60,0.34-0.5) node[right]{$\beta$};
\draw[orange][line width=1.2pt] (-0.68, 1.88)--( 0.68, 1.88);
\draw (0,-2.8) node{$\SURFextra(A_n)$ contains three parts};
\draw (0,-3.2) node{${_L\SURFextra}$, ${_R\SURFextra}$ and ${_U\SURFextra}$.};
\end{tikzpicture}
\ \ \ \
\begin{tikzpicture}[scale=0.9]
\draw[line width=1.2pt]  ( 0.00,-2.00) -- ( 0.00, 2.00) -- ( 1.97, 0.35);
\draw[line width=1.2pt]  ( 0.00,-2.00) arc(-90:10:2);
\draw[line width=2pt][white] ( 0.35,-1.97) arc(-80:-40:2);
\draw[line width=1.2pt][dotted] ( 0.35,-1.97) arc(-80:-40:2);
\filldraw[qqwuqq] ( 0.00, 2.00) circle (0.1); \draw ( 0.00, 2.00) node[above]{$p$};
\draw[red][line width=1pt] (1.97, 0.35) circle (0.1); \draw (1.97, 0.35) node[right]{$q_2$};
\filldraw[qqwuqq] (1.73,-1.00) circle (0.1);
\filldraw[qqwuqq] (1.00,-1.73) circle (0.1);
\draw[red][line width=1pt] ( 0.00,-2.00) circle (0.1); \draw ( 0.00,-2.00) node[below]{$q_1$};
\draw[qqwuqq][line width=1pt] ( 0.00, 2.00) -- (1.73,-1.00);
\draw[qqwuqq][line width=1pt][dotted] ( 0.00, 2.00) -- (1.00,-1.73);
\draw[blue][line width=1.2pt] (1.97, 0.35)--(0,-2);
\draw (0,-2.8) node{${_R\SURFextra}$};
\end{tikzpicture}
\caption{\textsf{The complete tile $\T$ belongs to $\Tri_{\Gamma}(\SURFextra(A_n))$,
where $\Gamma$ is a GD containing $c_{P(1)}$.}}
\label{fig-complete tile 3}
\end{figure}
By Definition \ref{def-alg of GD}, we have:
\begin{itemize}

  \item $\Q^{1,L}$ and $\Q^{2,L}$ are decided by the triangulation of the marked ribbon surface ${_L\SURFextra} = ({_L\Surf(A_n)}, {_L\M(A_n)}, {_L\Dgreen(A_n)})^{{_L\E(A_n)}}$ where:
    \begin{itemize}
      \item  ${_L\SURFextra}$ obtained by arcs $pq_1$ and $pq_3$ cutting $\SURFextra(A_n)$
        (see ${_L\SURFextra}$ in \Pic \ref{fig-complete tile 3});
      \item ${_L\M(A_n)}=\M(A_n)|_{{_L\Surf(A_n)}}$ and ${_L\Dgreen(A_n)}=\Dgreen(A_n)|_{{_L\Surf(A_n)}}$
        are restrictions of $\M(A_n)$ and $\Dgreen(A_n)$ over ${_L\Surf(A_n)}$ of $\Surf(A_n)$, respectively;
    \end{itemize}
  \item $\Q^{1,R}$ and $\Q^{2,R}$ are decided by ${_R\SURFextra} = ({_R\Surf(A_n)}, {_R\M(A_n)}, {_R\Dgreen(A_n)})^{{_R\E(A_n)}}$ where:
    \begin{itemize}
      \item  ${_R\SURFextra}$ obtained by arcs $pq_1$ and $pq_2$ cutting $\SURFextra(A_n)$
        (see ${_R\SURFextra}$ in \Pic \ref{fig-complete tile 3});
      \item ${_R\M(A_n)}=\M(A_n)|_{{_R\Surf(A_n)}}$ and ${_R\Dgreen(A_n)}=\Dgreen(A_n)|_{{_R\Surf(A_n)}}$;
    \end{itemize}
  \item $\Q^{1,U}$ and $\Q^{2,U}$ are decided by ${_U\SURFextra} = ({_U\Surf(A_n)}, {_U\M(A_n)}, {_U\Dgreen(A_n)})^{{_U\E(A_n)}}$ where:
    \begin{itemize}
      \item  ${_U\SURFextra}$ obtained by arcs $q_2q_3$ cutting $\SURFextra(A_n)$ such that $p$ is a marked point on its boundary
        (see ${_U\SURFextra}$ in \Pic \ref{fig-complete tile 3});
      \item ${_U\M(A_n)}=\M(A_n)|_{{_U\Surf(A_n)}}$ and ${_U\Dgreen(A_n)} = \Dgreen(A_n)|_{{_U\Surf(A_n)}}$.
    \end{itemize}
\end{itemize}
\noindent The proof of $\Gamma_1 = \Gamma_2$ is divided to three parts:
\begin{center}
  (a)\ $\Gamma_1|_{{_L\Surf(A_n)}} = \Gamma_2|_{{_L\Surf(A_n)}}$;
  (b)\ $\Gamma_1|_{{_R\Surf(A_n)}} = \Gamma_2|_{{_R\Surf(A_n)}}$;
  (c)\ $\Gamma_1|_{{_U\Surf(A_n)}} = \Gamma_2|_{{_U\Surf(A_n)}}$.
\end{center}
We only prove (a), the proofs of (b) and (c) are similar to (a).

(a) Note that ${_L\SURFextra}$ can be seen as a marked ribbon surface of $A_L=\kk\A_{k_L}$ for some $1\le k_L\le k-1$,
i.e., ${_L\SURFextra}\simeq \SURF_{\gbullet}^{\E(A_L)}(A_L)$. The subquiver $\Q^{1,L}$ and $\Q^{2,L}$ can be viewed as the quivers induced by some triangulations
$\Gamma_1|_{{_L\Surf(A_n)}}$ and $\Gamma_2|_{{_L\Surf(A_n)}}$ of ${_L\SURFextra}$, respectively.
It is easy to see that the curve $c_2$ (as a permissible curve in ${_L\SURFextra}$)
corresponds to the indecomposable projective-injective $A_L$-module.

On the other hand, let
\[\widehat{{_L\SURFextra}} = ({_{L\cup R}\Surf(A_n)}, {_L\M(A_n)}\cup\{q_2'\}, {_L\Dgreen(A_n)}\cup\{a\})^{{_L\E}},\]
where $_L\E$ is decided by the definition of marked ribbon surface (see \Pic \ref{fig-complete tile 4}).
Then we have a homotopy equivalence $\widehat{{_L\SURFextra}}\simeq\SURF_{\gbullet}^{\E(\widehat{A_L})}(\widehat{A_L})$,
where $\widehat{A_L} \cong \kk\A_{k_L+1}$.
It induces a bijection
\[\varphi: \Tri(\widehat{{_L\SURFextra}}) \to \Tri(\SURF_{\gbullet}^{\E(\widehat{A_L})}(\widehat{A_L})).\]
Assume $k_L+1\le k-1$, then $A^{\Gamma_1|_{{_{L\cup R}\Surf(A_n)}}} \cong A^{\Gamma_2|_{{_{L\cup R}\Surf(A_n)}}}$,
where $\Gamma_i|_{{_{L\cup R}\Surf(A_n)}}$ $=$ $\{c\in\Gamma_i\mid c \in \PC(\widehat{{_L\SURFextra}})\}$ ($i\in\{1,2\}$).
By inductive hypotheses, we have
\begin{align}
 A^{\Gamma_1|_{{_{L\cup R}\Surf(A_n)}}} & \cong A^{\varphi(\Gamma_1|_{{_{L\cup R}\Surf(A_n)}})}
 \cong A^{\varphi(\Gamma_2|_{{_{L\cup R}\Surf(A_n)}})} \cong A^{\Gamma_2|_{{_{L\cup R}\Surf(A_n)}}} \nonumber \\
& \Leftrightarrow \ \varphi(\Gamma_1|_{{_{L\cup R}\Surf(A_n)}}) =  \varphi(\Gamma_2|_{{_{L\cup R}\Surf(A_n)}}) \nonumber \\
&  \Leftrightarrow \ \Gamma_1|_{{_{L\cup R}\Surf(A_n)}} = \Gamma_2|_{{_{L\cup R}\Surf(A_n)}}. \nonumber
\end{align}
Thus, $\Gamma_1|_{{_L\Surf(A_n)}} = (\Gamma_1|_{{_{L\cup R}\Surf(A_n)}})\backslash\{c_1, c_3\}  = (\Gamma_2|_{{_{L\cup R}\Surf(A_n)}})\backslash\{c_1, c_3\} = \Gamma_2|_{{_L\Surf(A_n)}}$.

\begin{multicols}{2} 

\begin{figure}[H]
\centering
\definecolor{ffqqqq}{rgb}{1,0,0}
\definecolor{qqwuqq}{rgb}{0,0.5,0}
\begin{tikzpicture}[scale=0.9]
\draw[line width=1.2pt]  ( 0.00, 2.00) -- (-1.97, 0.35);
\draw[line width=1.2pt]  ( 0.00,-2.00) arc(-90:-190:2);
\draw[line width=2pt][white] (-0.35,-1.97) arc(-100:-140:2);
\draw[line width=1.2pt][dotted]  (-0.35,-1.97) arc(-100:-140:2);
\filldraw[qqwuqq] ( 0.00, 2.00) circle (0.1); \draw ( 0.00, 2.00) node[above]{$p$};
\draw[red][line width=1pt] (-1.97, 0.35) circle (0.1); \draw (-1.97, 0.35) node[left]{$q_3$};
\filldraw[qqwuqq] (-1.73,-1.00) circle (0.1);
\filldraw[qqwuqq] (-1.00,-1.73) circle (0.1);
\draw[red][line width=1pt] ( 0.00,-2.00) circle (0.1); \draw ( 0.00,-2.00) node[below]{$q_1$};
\draw[qqwuqq][line width=1pt] ( 0.00, 2.00) -- (-1.73,-1.00);
\draw[qqwuqq][line width=1pt][dotted] ( 0.00, 2.00) -- (-1.00,-1.73);
\draw[blue][line width=1.2pt] (-1.97, 0.35)--(0,-2);
\filldraw[yellow][opacity=0.25] ( 0.00, 2.00) -- (-1.97, 0.35) to[out=-97, in=180] ( 0.00,-2.00) -- ( 0.00, 2.00);
\draw[line width=1.2pt]  ( 0.00, 2.00) -- ( 1.97, 0.35);
\draw[line width=1.2pt]  ( 0.00,-2.00) arc(-90:10:2);
\filldraw[qqwuqq] ( 0.00, 2.00) circle (0.1); \draw ( 0.00, 2.00) node[above]{$p$};
\draw[red][line width=1pt] (1.97, 0.35) circle (0.1); \draw (1.97, 0.35) node[right]{$q_2$};
\filldraw[qqwuqq] (1.73,-1.00) circle (0.1); \draw (1.73,-1.00) node[right]{$q_2'$};
\filldraw[qqwuqq][line width=1pt] ( 0.00,-2.00) circle (0.1); \draw ( 0.00,-2.00) node[below]{$q_1$};
\draw[qqwuqq][line width=1pt] ( 0.00, 2.00) -- (1.73,-1.00); \draw [qqwuqq] (0.56,0.8) node[right]{$a$};
\draw[blue][line width=1.2pt] (1.97, 0.35)--(0,-2);
\filldraw[cyan][opacity=0.25] ( 0.00, 2.00) -- ( 1.97, 0.35) to[out=-80, in=0] ( 0.00,-2.00) -- ( 0.00, 2.00);
\draw[qqwuqq][line width=1.2pt]  ( 0.00, 2.00) -- ( 0.00,-2.00);
\draw[blue][line width=1.2pt] (-1.97, 0.35) -- ( 1.97, 0.35); \draw[blue] ( 0.00, 0.35) node[above]{$c_1$};
\draw[blue][line width=1.2pt] ( 0.00,-2.00) -- (-1.97, 0.35); \draw[blue] (-0.86,-1.00) node[left]{$c_2$};
\draw[blue][line width=1.2pt] ( 0.00,-2.00) -- ( 1.97, 0.35); \draw[blue] ( 0.86,-1.00) node[right]{$c_3$};
\draw[blue][<-] (-0.40, 0.19)--(-0.80,-0.50); \draw[blue] (-0.60,0.34-0.5) node[left]{$\alpha$};
\draw[blue][<-] ( 0.80,-0.50)--( 0.40, 0.19); \draw[blue] ( 0.60,0.34-0.5) node[right]{$\beta$};
\end{tikzpicture}
\caption{\textsf{ }}
\label{fig-complete tile 4}
\end{figure}

\begin{figure}[H]
\centering
\definecolor{ffqqqq}{rgb}{1,0,0}
\definecolor{qqwuqq}{rgb}{0,0.5,0}
\begin{tikzpicture}[scale=0.9]
\draw[line width=1.2pt] (2,0) arc (0:360:2);
\draw[line width=2pt][white] (1.414,-1.414) arc (-45:-80:2);
\draw[line width=1.2pt][dotted] (1.414,-1.414) arc (-45:-80:2);
\draw[line width=2pt][white] (-1.414,-1.414) arc (-135:-110:2);
\draw[line width=1.2pt][dotted] (-1.414,-1.414) arc (-135:-110:2);
\filldraw[qqwuqq] ( 0.00, 2.00) circle (0.1);
\draw[ffqqqq][line width=1.2pt] (-1.56, 1.25) circle (0.1); \draw (-1.56, 1.25) node[left]{$q_3$};
\filldraw[qqwuqq] (-1.95, 0.45) circle (0.1);
\filldraw[qqwuqq] (-1.95,-0.45) circle (0.1);
\filldraw[qqwuqq] (-1.61,-1.19) circle (0.1);
\filldraw[qqwuqq] (-0.87,-1.80) circle (0.1);
\filldraw[qqwuqq] (-0.00,-2.00) circle (0.1);
\filldraw[qqwuqq] ( 0.87,-1.80) circle (0.1);
\filldraw[qqwuqq] ( 1.61,-1.19) circle (0.1);
\filldraw[qqwuqq] ( 1.94,-0.45) circle (0.1); \draw ( 1.94,-0.45) node[right]{$q_1$};
\filldraw[qqwuqq] ( 1.95, 0.45) circle (0.1); \draw ( 1.94, 0.45) node[right]{$q_2'$};
\draw[ffqqqq][line width=1.2pt] ( 1.56, 1.25) circle (0.1); \draw ( 1.56, 1.25) node[right]{$q_2$};
\draw[qqwuqq][line width=1.2pt] ( 0.00, 2.00) -- (-1.95, 0.45);
\draw[qqwuqq][line width=1.2pt] ( 0.00, 2.00) -- (-1.95,-0.45);
\draw[qqwuqq][line width=1.2pt][dotted] ( 0.00, 2.00) -- (-1.61,-1.19);
\draw[qqwuqq][line width=1.2pt][dotted] ( 0.00, 2.00) -- (-0.87,-1.80);
\draw[qqwuqq][line width=1.2pt] ( 0.00, 2.00) -- (-0.00,-2.00);
\draw[qqwuqq][line width=1.2pt][dotted] ( 0.00, 2.00) -- ( 0.87,-1.80);
\draw[qqwuqq][line width=1.2pt][dotted] ( 0.00, 2.00) -- ( 1.61,-1.19);
\draw[qqwuqq][line width=1.2pt] ( 0.00, 2.00) -- ( 1.94,-0.45);
\draw[qqwuqq][line width=1.2pt] ( 0.00, 2.00) -- ( 1.95, 0.45);
\draw[qqwuqq] (0,2) node[above]{$p$};
\draw[blue][line width=1.2pt] (-1.56, 1.25) -- ( 1.94,-0.45) -- (1.56, 1.25) -- (-1.56, 1.25);
\draw[blue](0.8,0.8) node{$\T$};
\filldraw[yellow][opacity=0.25][line width=1.2pt]
(0, 2) to[out=180, in=90] (-2, 0) to[out=-90,in=180] (0,-2) to[out=0,in=-103.0593] ( 1.94,-0.45) -- (0,2);
\draw[orange][line width=1.2pt] (-1.56, 1.25)--(1.56, 1.25);
\draw[white] ( 0.00,-2.00) node[below]{$q_1$};
\end{tikzpicture}
\caption{\textsf{ }}
\label{fig-complete tile 5}
\end{figure}
\end{multicols}

Assume $k_L+1= k$ (see \Pic \ref{fig-complete tile 5}). If $\Gamma_1|_{{_L\Surf(A_n)}}$ provides a complete triangle $\T' \ne \T$, we can do it as above. Otherwise, $\Q^{1,L}=\Q^{2,L}$ is the quiver of some hereditary algebra of type $A_{k-2}$.
In this case, we have $\Gamma_1|_{{_L\Surf(A_n)}} = \Gamma_2|_{{_L\Surf(A_n)}}$ because they have a same curve $q_1q_3$.
\end{proof}

\section{A geometric characterization of the silted algebra of type $A_n$} \label{sec:end-silt}
In this section, we give a geometric characterization of the endomorphism algebras of silting complexes for gentle algebras. Then by using the geometric characterization, we give a classification of silted algebras of type $A_n$.

\subsection{A geometric characterization of the endomorphism algebras of silting complexes} Recall that in \cite[Theorem 3.2]{AmiotPlamondonSchroll23}, Amiot, Plamondon and Schroll gave a characterization of the silting complexes for gentle algebras. Let $A$ be a graded gentle algebra and $\SURFred{\F_A}(A)=(\Surf_A, \M_A, \Y_A, \F_A, \tDred(A))$ the graded marked ribbon surface of $A$. Let $S\cong \bigoplus_{1\le i\le n} X(\tgamma_i, L_i)$. Then by (\ref{formula-int. ind. in HKK 2.1}), \cite[Proposition 5.16]{ArnesenLakingPauksztello16},
 \cite[Proposition 1.11]{AmiotPlamondonSchroll23} and dimension formula (\ref{dim. formula}), we have the following lemma.

\begin{lemma}\label{lem:silting-complexes}
$S$ is silting if and only if it satisfies the following conditions.
\begin{itemize}
  \item[\rm (S1)] $\tgamma_i\cap\tgamma_j\cap(\Surf_A\backslash\bSurf_A)=\varnothing$ for any $ 1\le i,j\le n$;
  \item[\rm (S2)] $\tgamma_i\in\AC_{\m}(\SURFred{\F_A}(A))$, because each band complex has non-trivial positive self-extension;
  \item[\rm (S3)] 
    $n  = \sharp\M_A + b + 2g - 2\ (= \sharp \Q_0)$,
    where $b$ is the number of boundary components of $\Surf_A$ and $g$ is the genus of $\Surf_A$;
  \item[\rm (S4)] If $\tgamma_i$ is to the left $($resp., right$)$ of $\tgamma_j$ at the point $p\in\M_A$,
    then $\ii_p(\tgamma_i, \tgamma_j)\le 0$ $($resp., $\ii_p(\tgamma_j, \tgamma_i)\le 0)$.
\end{itemize}
\end{lemma}

A {\it non-positive {\rm $\gbullet$-grFFAS}} is a $\gbullet$-grFFAS satisfying condition ($\mathrm{S}$4). 
The following theorem provides a corresponding between non-positive {\rm $\gbullet$-grFFAS} and silting complex. 

\begin{theorem}  \label{thm. silt}
Let $A$ be a gentle algebra. Then there is a bijection
\[X: \gbullet\textrm{-}\mathrm{grFFAS}^{\leq 0}(A) \to \silt(A) \]
between the set $\gbullet\textrm{-}\mathrm{grFFAS}^{\leq 0}(A)$ of all non-positive {\rm $\gbullet$-grFFASs} of $\SURFred{\F_A}(A)$
and the set $\silt(A)$ of all isoclasses of silting complexes in $\per(A)$.
\end{theorem}

\begin{proof} 
Let $S=\bigoplus_{i\in I}S_i$ be a silting complex. Then $X^{-1}(S)=\{X^{-1}(S_i)\mid i\in I\}$ is a non-positive $\gbullet$-grFFAS by Lemma \ref{lem:silting-complexes}. 

Conversely, let $\tilde{\Gamma}=\{\tgamma_i\mid 1\le i\le n\}$ be a non-positive $\gbullet$-grFFAS. 
Then $\tilde{\Gamma}$ satisfies (S1), (S2) and (S4) given in Lemma \ref{lem:silting-complexes}. 
Thus, $\Hom_{\per(A)}(X(\tgamma_i), X(\tgamma_j)[>0])=0$ because $\ii_p(\tgamma_i,\tgamma_j)\le 0$, where 
$p\in\M_A$ is the intersection of $\gamma_i$ and $\gamma_j$. 
In particular, (S3) yields that $X(\tilde{\Gamma}) = \bigoplus_{i=1}^n X(\tgamma_i)$ generates $\per(A)$, 
because $\Gamma$ is an admissible dissection (see \cite[Definition 3.9 and Corollary 5.8]{AmiotPlamondonSchroll23}). 
Therefore, $X(\tilde{\Gamma})$ is a silting complex. 
\end{proof}

Recall that any graded $\gbullet$-curve $\tc$ is divided to some segments by $\rbullet$-FFAS cutting $\SURFred{\F}$.
Therefore, assume $\tc$ consecutively crosses graded $\rbullet$-arcs $\tarc_1,\ldots, \tarc_{n}$,
the segment $\tc_{i,i+1}$ obtained by $\tarc_i$ and $\tarc_{i+1}$ cutting $\tc$ can be viewed as a subset
of the elementary $\rbullet$-polygon $\PP$ such that $\tarc_i$ and $\tarc_{i+1}$ are edges of $\PP$.
Notice that $\PP$ has a unique edge $b_{\PP}$ which is a subset of $\bSurf$
and there exists a unique $\gbullet$-marked point $\m_{\PP}\in\M$ on the edge $b_{\PP}$,
thus, for simplicity, we say {\it $\m_{\PP}$ is to the left $($resp., right$)$ of $\tc$}
if $\m_{\PP}$ is on the left side of the forward direction of the segment $\tc_{i, i+1}$
obtained by $\tarc_{i}$ and $\tarc_{i+1}$ cutting the curve $c$.

\begin{lemma} \label{lemm-int:+1 and -1}
Let $A$ be a gentle algebra. For any graded $\gbullet$-curve $\tc \in \AC_{\m}(\SURFred{\F_A}(A))$
consecutively crossing $\tarc_1, \tarc_2 \in \tDred(A)$ {\rm(}the intersections are denoted by $u$ and $v$, respectively{\rm)},
let $\m_{\PP}$ be the $\gbullet$-marked point on the elementary $\rbullet$-polygon $\PP$
such that $\tarc_1$ and $\tarc_2$ are two edges of $\PP$.
\begin{itemize}
  \item[\rm(1)] If $\m_{\PP}$ is to the left of $\tc$, then $\ii_v(\tc, \tarc_2) = \ii_u(\tc, \tarc_1)+1$;
  \item[\rm(2)] If $\m_{\PP}$ is to the right of $\tc$, then $\ii_v(\tc, \tarc_2) = \ii_u(\tc, \tarc_1)-1$.
\end{itemize}
\end{lemma}

\begin{proof} We only prove $(1)$, the proof of $(2)$ is similar.
If $\m_{\PP}$ is to the left of $\tc$, then $\tc$ divides $\PP$ into two parts:
the left parts $\PP_L$ containing $\m_{\PP}$ and the right parts $\PP_R$ containing no $\gbullet$-marked point.
In counterclockwise order, we denote by $\tilde{\sigma}_2,\ldots,\tilde{\sigma}_{n-1}$
the graded $\rbullet$-arcs on the right of the segment $\tc_{1,2}$ obtained by
$\tarc_1$ and $\tarc_2$ cutting $\tc$, and set $\tarc_1=\tilde{\sigma}_1$ and $\tarc_2=\tilde{\sigma}_n$
(see \Pic \ref{curve divides poly}).
By hypothesis, we have the grading of each arrow equals to zero and so $\ii_{\sigma_j \cap \sigma_{j+1}}(\tilde{\sigma}_j, \tilde{\sigma}_{j+1}) = 0 \in \ZZ$ for all $1\le j< n$. Therefore,
\begin{align}
    \ii_v(\tc, \tarc_2)
& = \ii_v(\tc, \tilde{\sigma}_n) = \tc\cdot\kappa_{\tc\to\tilde{\sigma}_n} \cdot \tilde{\sigma}_n^{-1}
  = (\tc\cdot\kappa_{\tc\to\tilde{\sigma}_{n-1}} \cdot \tilde{\sigma}_{n-1}^{-1} )
    \cdot (\tilde{\sigma}_{n-1} \cdot\kappa_{\tilde{\sigma}_{n-1}\to\tilde{\sigma}_n} \cdot \tilde{\sigma}_n^{-1}) \nonumber \\
& = (\tc\cdot\kappa_{\tc\to\tilde{\sigma}_1} \cdot \tilde{\sigma}_{1}^{-1} )
    \cdot \prod_{j=1}^{n-1} \ii_{\sigma_{j}\cap\sigma_{j+1}}(\tilde{\sigma}_{j}, \tilde{\sigma}_{j+1})
  = \ii_u(\tc, \tilde{\sigma}_{1}) \cdot \prod_{j=1}^{n-1} \ii_{\sigma_{j}\cap\sigma_{j+1}}(\tilde{\sigma}_{j}, \tilde{\sigma}_{j+1}) \nonumber \\
& \mathop{\mapsto}\limits^{\text{1-1}}_{\text{(\ref{formula-index})}} \ii_u(\tc, \tilde{\sigma}_{1})+1\  \in \ZZ, \nonumber
\end{align}
where $\kappa_{\tc_1\to\tc_2}$ is the paths from $\dot{c}_1(t_1)$ to $\dot{c}_2(t_2)$
given by clockwise rotation in $T_{p}\Surf$ by an angle $<\pi$ ($p=c_1(t_1)=c_2(t_2)\in\tc_1\cap \tc_2$).
\begin{figure}[H]
\definecolor{ffqqqq}{rgb}{1,0,0}
\definecolor{qqwuqq}{rgb}{0,0.5,0}
\centering
\begin{tikzpicture}
\draw[red][line width=1.2pt] (2,0) -- (1,1.73);
\draw[black][line width=1.2pt] (1,1.73) -- (-1, 1.73);
\draw[red][line width=1.2pt, dotted] (-1, 1.73) -- (-2,0);
\draw[red][line width=1.2pt] (-2,0) -- (-1,-1.73);
\draw[red][line width=1.2pt] (-1,-1.73) -- (1,-1.73);
\draw[red][line width=1.2pt, dotted] (1,-1.73) -- (2,0);
\fill [qqwuqq] (0,1.73) circle (2.8pt);
\fill [red] (2,0) circle (2.8pt);       \fill [white] (2,0) circle (2.2pt);
\fill [red] (1,1.73) circle (2.8pt);    \fill [white] (1,1.73) circle (2.2pt);
\fill [red] (-1, 1.73) circle (2.8pt);  \fill [white] (-1, 1.73) circle (2.2pt);
\fill [red] (-2,0) circle (2.8pt);      \fill [white] (-2,0) circle (2.2pt);
\fill [red] (-1,-1.73) circle (2.8pt);  \fill [white] (-1,-1.73) circle (2.2pt);
\fill [red] (1,-1.73) circle (2.8pt);   \fill [white] (1,-1.73) circle (2.2pt);
\draw[qqwuqq][line width=1.2pt][->] (-1.73,-1) -- (1.73, 1);
\draw[qqwuqq] (-1.73, -1) node[below]{$\tc$};
\draw[red] (-1.8, -0.5) node[left]{$\tarc_1 = \tilde{\sigma}_1$};
\draw[red] (0, -1.73) node[below]{$\tilde{\sigma}_2$};
\draw[red] ( 1.8,  0.5) node[right]{$\tarc_2 = \tilde{\sigma}_n$};
\end{tikzpicture}
\caption{\textsf{$\tc$ divides $\PP$ into two parts}}
\label{curve divides poly}
\end{figure}
\end{proof}

\begin{corollary} \label{coro-silting and FFAS}
Let $A$ be a gentle algebra. Then there exists a bijection
\[X: \gbullet\textrm{-}\mathrm{grFFAS}^{\leq 0}_{\{0,1\}}(A) \to 2\textrm{-}\silt(A) \]
between the set $\gbullet\textrm{-}\mathrm{grFFAS}^{\leq 0}_{\{0,1\}}(A)$ of all non-positive $\gbullet$-$\mathrm{grFFAS}$ of $\SURFred{\F_A}(A)$ such that for all graded $\rbullet$-arc $\tarc\in\tDred(A)$ crossed by $\tgamma\in\gtD$ the intersection index $\ii_p(\tgamma, \tarc)\in\{0,1\}$
and the set $2\textrm{-}\silt(A)$ of all isoclasses of 2-term silting complexes in $\per(A)$.
\end{corollary}

\begin{proof}
It is obvious by Theorem \ref{thm. silt} and Lemma \ref{lemm-int:+1 and -1}.
\end{proof}

Let $A$ be a graded gentle algebra.
Next we consider the endomorphism algebras of complexes in $\per(A)$. Assume that $C = \bigoplus_{i\in I} X(\tgamma_i)$ is a complex in $\per(A)$. We say that $C$ is a {\it complex without intersection} if $\gtD(C):=\{\tgamma_i\mid i\in I\}$ is a dissection of $\SURFred{\F_A}(A)$ such that $\tgamma_i\cap\tgamma_j\cap(\innerSurf)=\varnothing$. In this case, the complex $C$ can induce a graded algebra $A^C=\kk\Q^C/\I^C$ by graded marked ribbon surface as following:

Step 1: The quiver $\Q^C$ of $A^C$ is obtained by:
  \begin{itemize}
    \item The vertex set of $\Q^C$, say $\Q^C_0$, is the set $\gtD(C):=\{\tgamma_i \mid i\in I\}$.

    \item  If $\gamma_i$ and $\gamma_j$ have an intersection $p\in \M_A$ such that $\gamma_i$ is to the left of $\gamma_j$ at $p$
      and for all $\tgamma_l\in \gtD$ with endpoint $p$, $\tgamma_l$ is not between $\tgamma_i$ and $\tgamma_j$, then there exists an arrow $\alpha$ from $\tgamma_j$ to $\tgamma_i$ in $\Q^C_1$ of $\Q^C$. In this case,
      we define $|\alpha| = \ii_p(\tgamma_i, \tgamma_j)$.
  \end{itemize}

Step 2: For two arrows $\alpha: \gamma_i\to \gamma_j$ and $\beta: \gamma_j \to \gamma_k$,
    $\alpha\beta\in \I^C$ if and only if $\gamma_i, \gamma_j, \gamma_k$ are edges of the same polygon obtained by $\gtD(C)$ cutting $\SURFred{\F_A}$.

In particular, let $S$ be a silting complex in $\per(A)$. Then the collection of arcs
corresponding to $S$ forms a full formal arc system of the surface by Lemma \ref{lem:silting-complexes} and the silting complex $S$ can induce a homologically smooth graded gentle algebra $A^S$ associated to the surface model of $A$. Recall that in \cite{HaidenKatzarkovKontsevich17}, Haiden, Katzarkov and Kontsevich proved that the perfect derived category of a homologically smooth graded gentle algebra is triangle equivalent to the partially wrapped Fukaya category of a graded oriented smooth surface. Thus, $\per(A)$ and $\per(A^S)$
are both triangle equivalent to the homologically Fukaya category of the surface. Then we obtain a triangle equivalence
from $\per(A)$ to $\per(A^S)$ which sending $S$ to $A^S$. So we have the following corollary.

\begin{corollary}\label{coro:endo-algebras}
Let $A$ be a graded gentle algebra and $S\in \per(A)$ a silting complex. Then $\H^0(A^S)=\End_{\per(A^S)}(A^S)=\End_{\per(A)}(S)$.
\end{corollary}

By Corollaries \ref{coro-silting and FFAS} and \ref{coro:endo-algebras}, we can obtain a geometric characterization of the silted algebras for gentle algebras.

\begin{corollary} \label{coro-silted}
Let $A$ be a hereditary gentle algebra and $S$ be a 2-term silting complex over $A$. Then $\H^0(A^{S})$ is a silted algebra of type $A$. Moreover, for any silted algebra $B$ of type $A$, there exists a $\gtD\in\gbullet\textrm{-}\mathrm{grFFAS}^{\leq 0}_{\{0,1\}}(A)$ such that $B\cong A^{X(\gtD)}$.
\end{corollary}

\subsection{The classification of silted algebras of type $A_n$} \label{sec:silt-cls}
In this subsection, we give a classification of silted algebras of type $A_n$ in terms of Corollaries \ref{coro-silting and FFAS} and \ref{coro-silted}. According to subsection \ref{subsec-geo-der}, we have the marked ribbon surface $\SURFred{\F_{A_n}}(A_n)$ of $\per(A_n)$ (see \Pic \ref{grmarked ribbon surface of An 2}).
Let $\tarc_t$ be the graded $\rbullet$-arc in $\tDred(A_n)$ for any $1\le t\le n$ and $\gamma_{s,t}$ be the $\gbullet$-curve with endpoints $\m_s$ and $\m_t$, where $0\le s\ne t\le n$. It is easy to see that any $\gbullet$-FFAS $\gD$ of $\SURFred{\F_{A_n}}(A_n)$ contains at least one $\gbullet$-curve which is of the form $\gamma_{0,l}$ for any $1\le l\le n$.
For simplicity, we need the following notations:
$\gDI=\{\gamma_{0,l}\mid 1\le l\le n\}\cap \gD$, $\gDII=\gD\backslash \gDI$,
$\gtDI = \{\tgamma_{0,l}\mid \gamma_{0,l}\in\gDI\}$ and $\gtDII = \{\tgamma_{s,t}\mid \gamma_{s,t}\in\gDII\}$.
\begin{figure}[H]
\definecolor{ffqqqq}{rgb}{1,0,0}
\definecolor{qqwuqq}{rgb}{0,0.5,0}
\centering
\begin{tikzpicture}
\draw[black][line width=1.2pt] (2,0) arc (0:360:2);
\draw[red][line width=1.2pt] (2,0) -- (1,1.73);  \draw[red] (1.5,0.86) node[left]{$\tarc_n$};
\draw[red][line width=1.2pt] (-1, 1.73) -- (-2,0); \draw[red] (-1.5,0.86) node[right]{$\tarc_1$};
\draw[red][line width=1.2pt] (-2,0) -- (-1,-1.73); \draw[red] (-1.5,-0.86) node[right]{$\tarc_2$};
\draw[red][line width=1.2pt, dotted] (-1,-1.73) -- (1,-1.73);
\draw[red][line width=1.2pt] (1,-1.73) -- (2,0); \draw[red] (1.5,-0.86) node[left]{$\tarc_{n-1}$};
\fill [qqwuqq] (0,2) circle (2.8pt); \draw[qqwuqq] (0,2) node[above]{$\m_0$};
\fill [qqwuqq] (-1.73,1) circle (2.8pt); \draw[qqwuqq] (-1.73,1) node[left]{$\m_1$};
\fill [qqwuqq] (-1.73,-1) circle (2.8pt); \draw[qqwuqq] (-1.73,-1) node[left]{$\m_2$};
\fill [qqwuqq][opacity=0.25] (0,-2) circle (2.8pt); \draw[qqwuqq] (0,-2) node[below]{$\cdots$};
\fill [qqwuqq] (1.73,-1) circle (2.8pt); \draw[qqwuqq] (1.73,-1) node[right]{$\m_{n-1}$};
\fill [qqwuqq] (1.73,1) circle (2.8pt); \draw[qqwuqq] (1.73,1) node[right]{$\m_n$};
\fill [red] (2,0) circle (2.8pt); \fill [white] (2,0) circle (2.2pt); \draw[red] (2,0) node[right]{$\c_{n-1,n}$};
\fill [red] (1,1.73) circle (2.8pt); \fill [white] (1,1.73) circle (2.2pt); \draw[red] (1,1.73) node[right]{$\c_{n,0}$};
\fill [red] (-1, 1.73) circle (2.8pt); \fill [white] (-1, 1.73) circle (2.2pt); \draw[red] (-1,1.73) node[left]{$\c_{0,1}$};
\fill [red] (-2,0) circle (2.8pt); \fill [white] (-2,0) circle (2.2pt); \draw[red] (-2,0) node[left]{$\c_{1,2}$};
\fill [red] (-1,-1.73) circle (2.8pt); \fill [white] (-1,-1.73) circle (2.2pt);
\fill [red] (1,-1.73) circle (2.8pt); \fill [white] (1,-1.73) circle (2.2pt);
\draw [blue][line width=1.2pt] (-1.73,-1)--(1.73,-1);
\draw [blue] (0,-1) node[above]{$\gamma_{2,n-1}$};
\end{tikzpicture}
\caption{\textsf{The marked ribbon surface $\SURFred{\F_{A_n}}(A_n)$ of $A_n$, where the foliation $\F_{A_n}$ satisfy that $|\alpha|=0$ for any arrow $\alpha$ of $A_n$.}}
\label{grmarked ribbon surface of An 2}
\end{figure}

Note that tilted algebras are silted. In subsection \ref{subsec-classification}, we give a description of the tilted algebras of type $A_n$. Next we give a description of the silted algebras of type $A_n$ which are not tilted algebras of type $A_n$. First, we need the following lemma.

\begin{lemma} \label{lemm:2-silt-ind}

Let $\tgamma\in \AC_{\m}(\SURFred{\F_{A_n}}(A_n))$. Then $\tgamma$ passes through at most two graded $\rbullet$-arcs.
Moreover, if $\tgamma$ passes through two graded $\rbullet$-arcs $\tarc_i$ and $\tarc_j$ for $i\neq j$, then we have:
\begin{center}
$\ii_{\gamma\cap a_i}(\tgamma, \tarc_i) = 0$ and $\ii_{\gamma\cap a_j}(\tgamma, \tarc_j) = 1$;
\end{center}
or
\begin{center}
$\ii_{\gamma\cap a_i}(\tgamma, \tarc_i) = 1$ and $\ii_{\gamma\cap a_j}(\tgamma, \tarc_j) = 0$.
\end{center}
\end{lemma}

\begin{proof} The first part of the statement is obvious. If $\tgamma$ passes through $\tarc_i$ and $\tarc_j$, then we have
$\{\ii_{\gamma\cap a_i}(\tgamma, \tarc_i), \ii_{\gamma\cap a_j}(\tgamma, \tarc_j)\} = \{0,1\}$
by Corollary \ref{coro-silting and FFAS} and Lemma \ref{lemm-int:+1 and -1}.
\end{proof}

\begin{proposition}\label{prop-gamma(0,l)}
Let $X(\gtD)=\bigoplus_{\tgamma\in\gtD}X(\tgamma)$ be a 2-term silting complex but not tilting.
Then there are at least two graded $\gbullet$-arcs $\tgamma_{0,l}$ and $\tgamma_{0,k}$ in $\gtDI$ such that
\begin{itemize}
  \item[\rm(1)]
    $\ii_{p_l}(\tgamma_{0,l}, \tarc_l)\ne \ii_{p_k}(\tgamma_{0,k}, \tarc_k) \in \{0,1\}$ for any $1\le l<k\le n$,
    where $p_l=\gamma_{0,l}\cap a_l$ and $p_{k}=\gamma_{0,{k}}\cap a_{k}$;

  \item[\rm(2)] $\gamma_{0,t}\notin\gDI$ for all $l<t<k$.
\end{itemize}
\end{proposition}

\begin{proof}
Denote by $p_t$ the intersection of $\gamma_{0,t}$ and $a_{t}$ such that $\gamma_{0,t}\in \gtDI$, where $1\leq t\leq n$. Assume that $\ii_{p_t}(\tgamma_{0,t},\tarc_{t}) = 0$ (resp., $=1$) for all $t$. Next we only prove that $X(\gtD)$ is tilting for $\ii_{p_t}(\tgamma_{0,t},\tarc_{t}) = 0$.

Assume that there is only one graded $\gbullet$-arc in $\gtDI$, we claim that $\gtDI=\{\tgamma_{0,1}\}$. Otherwise, we have $\gtDI=\{\tgamma_{0,t}\}$ for some $1< t\ (=i_1)\ \le n$. Since $\gtD$ is a $\gbullet$-grFFAS of $\SURFred{\F_{A_n}}(A_n)$, there is a grade $\gbullet$-arc $\tgamma_{\hat{t},t}\in\gtDII$ for some $1\le \hat{t}<t$ and $\gamma_{0,x}\notin\gDI$ for all $1\le x<t$, see \Pic \ref{fig-prop-gamma(0,l)-1}. Thus by (S4) and (\ref{formula-index}) (see \cite[(2.15)]{HaidenKatzarkovKontsevich17}), we obtain that
\begin{center}
$\ii_{\gamma_{\hat{t},t}\cap a_t}(\tgamma_{\hat{t},t}, \tarc_t)- \ii_{p_t}(\tgamma_{0,t},\tarc_t)=\ii_{m_t}(\gamma_{\hat{t},t},\tgamma_{0,t}) \leq 0$.
\end{center}
  Since $X(\gtD)$ is 2-term, we have $\ii_{\gamma_{\hat{t},t}\cap a_t}(\tgamma_{\hat{t},t}, \tarc_t) = 0$
  by Lemma \ref{lemm-int:+1 and -1}. Then
  \[\ii_{\gamma_{\hat{t},t}\cap a_{\hat{t}}}(\tgamma_{\hat{t},t}, \tarc_{\hat{t}})
  = \ii_{\gamma_{\hat{t},t}\cap a_t}(\tgamma_{\hat{t},t}, \tarc_t) - 1 = -1. \]
  By Lemma \ref{lemm:2-silt-ind}, this is a contradiction.

\begin{figure}[H]
\definecolor{ffqqqq}{rgb}{1,0,0}
\definecolor{qqwuqq}{rgb}{0,0.5,0}
\centering
\begin{tikzpicture}
\draw[black][line width=1.2pt] (2,0) arc (0:360:2);
\draw[red][line width=1.2pt] (-0.52, 1.93) -- (-1.41, 1.41);
\draw[red][line width=1.2pt][dotted] (-1.41,1.41) -- (-1.93, 0.52);
\draw[red][line width=1.2pt] (-1.93, 0.52) -- (-1.93,-0.52);
\draw[red][line width=1.2pt] (-1.93,-0.52) -- (-1.41,-1.41);
\draw[red][line width=1.2pt] (-1.41,-1.41) -- (-0.52,-1.93); \draw[red] (-0.97,-1.7) node[above]{$\tarc_{\hat{t}}$};
\draw[red][line width=1.2pt] (-0.52,-1.93) -- ( 0.52,-1.93);
\draw[red][line width=1.2pt][dotted] ( 0.52,-1.93) -- ( 1.41,-1.41);
\draw[red][line width=1.2pt] ( 1.41,-1.41) -- ( 1.93,-0.52);
\draw[red][line width=1.2pt] ( 1.93,-0.52) -- ( 1.93, 0.52); \draw[red] ( 1.93,0) node[left]{$\tarc_{t}$};
\draw[red][line width=1.2pt] ( 1.93, 0.52) -- ( 1.41, 1.41);
\draw[red][line width=1.2pt] ( 1.41, 1.41) -- ( 0.52, 1.93);
\fill [qqwuqq] (0,2) circle (2.8pt); \draw[qqwuqq] (0,2) node[above]{$\m_0$};
\fill [qqwuqq] (-1,1.732) circle (2.8pt); \draw[qqwuqq] (-1,1.732) node[above]{$\m_1$};
\fill [qqwuqq][opacity=0.25] (-1.732,1) circle (2.8pt);
\fill [qqwuqq] (-2,0) circle (2.8pt);
\fill [qqwuqq] (-1.732,-1) circle (2.8pt);
\fill [qqwuqq] (-1,-1.732) circle (2.8pt); \draw[qqwuqq] (-1,-1.732) node[left]{$\m_{\hat{t}}$};
\fill [qqwuqq] (0,-2) circle (2.8pt);
\fill [qqwuqq][opacity=0.25] (1,-1.732) circle (2.8pt);
\fill [qqwuqq] (1.732,-1) circle (2.8pt);
\fill [qqwuqq] (2,0) circle (2.8pt); \draw[qqwuqq] (2,0) node[right]{$\m_t$};
\fill [qqwuqq] (1.732,1) circle (2.8pt);
\fill [qqwuqq] (1,1.732) circle (2.8pt);
\fill [red] (-0.52, 1.93) circle (2.8pt); \fill[white] (-0.52, 1.93) circle (2.2pt);
                                          \draw[red]   (-0.52, 1.93) node[above]{$\c_{0,1}$};
\fill [red] (-1.41, 1.41) circle (2.8pt); \fill[white] (-1.41, 1.41) circle (2.2pt);
                                          \draw[red]   (-1.41, 1.41) node[above]{$\c_{1,2}$};
\fill [red] (-1.93, 0.52) circle (2.8pt); \fill[white] (-1.93, 0.52) circle (2.2pt);
\fill [red] (-1.93,-0.52) circle (2.8pt); \fill[white] (-1.93,-0.52) circle (2.2pt);
\fill [red] (-1.41,-1.41) circle (2.8pt); \fill[white] (-1.41,-1.41) circle (2.2pt);
\fill [red] (-0.52,-1.93) circle (2.8pt); \fill[white] (-0.52,-1.93) circle (2.2pt);
\fill [red] ( 0.52,-1.93) circle (2.8pt); \fill[white] ( 0.52,-1.93) circle (2.2pt);
\fill [red] ( 1.41,-1.41) circle (2.8pt); \fill[white] ( 1.41,-1.41) circle (2.2pt);
\fill [red] ( 1.93,-0.52) circle (2.8pt); \fill[white] ( 1.93,-0.52) circle (2.2pt);
\fill [red] ( 1.93, 0.52) circle (2.8pt); \fill[white] ( 1.93, 0.52) circle (2.2pt);
\fill [red] ( 1.41, 1.41) circle (2.8pt); \fill[white] ( 1.41, 1.41) circle (2.2pt);
\fill [red] ( 0.52, 1.93) circle (2.8pt); \fill[white] ( 0.52, 1.93) circle (2.2pt);
\draw[blue][opacity=0.7][line width=1.2pt] (0,2)--(2,0); \draw[blue] (1,1) node[right]{$\tgamma_{0,t}$};
\draw[blue][opacity=0.7][line width=1.2pt] (2,0)--(-1,-1.732); \draw[blue] (0.5,-0.86) node[above]{$\tgamma_{\hat{t},t}$};
\draw[cyan][opacity=0.7][line width=0.8pt][->] (-0.967, -1.6) to[out=90,in=0] (-2.2, -0.5);
\draw[cyan] (-2.2, -0.5) node[left]{$\ii_{\gamma_{\hat{t},t}\cap a_{\hat{t}}}(\tgamma_{\hat{t},t}, \tarc_{\hat{t}})=-1$};
\draw[cyan][opacity=0.7][line width=0.8pt][->] (1.93,-0.15)  to[out=-90,in=180] (2.5,-1);
\draw[cyan] (2.5,-1) node[right]{$\ii_{\gamma_{\hat{t},t}\cap a_t}(\tgamma_{\hat{t},t}, \tarc_t)=0$};
\draw[cyan][opacity=0.7][line width=0.8pt][->] (1.93, 0.15)  to[out= 90,in=180] (2.5, 1);
\draw[cyan] (2.5, 1) node[right]{$\ii_{p_t}(\tgamma_{0,t}, \tarc_t)=0$};
\end{tikzpicture}
\caption{ }
\label{fig-prop-gamma(0,l)-1}
\end{figure}

  Next we claim that $\gamma_{1,n} \in \gDII$. Otherwise, since $\gtD$ is a $\gbullet$-grFFAS and $\gDI=\{\gamma_{0,1}\}$, there exists $\gamma_{1,t}\in \gDII$ for some $1<t<n$, see \Pic \ref{fig-prop-gamma(0,l)-2}. Since $X(\gtD)$ is 2-term, by (S4), we obtain that
  \begin{center}
  $\ii_{\gamma_{1,t}\cap a_1}(\tgamma_{1,t},\tarc_1)=0$ and $\ii_{\gamma_{1,t}\cap a_t}(\tgamma_{1,t},\tarc_t)=1$.
  \end{center}

\begin{figure}[htbp]
\definecolor{ffqqqq}{rgb}{1,0,0}
\definecolor{qqwuqq}{rgb}{0,0.5,0}
\centering
\begin{tikzpicture}
\draw[black][line width=1.2pt] (2,0) arc (0:360:2);
\draw[red][line width=1.2pt] (-0.52, 1.93) -- (-1.41, 1.41);
\draw[red][line width=1.2pt][dotted] (-1.41,1.41) -- (-1.93, 0.52);
\draw[red][line width=1.2pt] (-1.93, 0.52) -- (-1.93,-0.52);
\draw[red][line width=1.2pt] (-1.93,-0.52) -- (-1.41,-1.41);
\draw[red][line width=1.2pt] (-1.41,-1.41) -- (-0.52,-1.93); \draw[red] (-0.97,-1.7) node[above]{$\tarc_{t}$};
\draw[red][line width=1.2pt] (-0.52,-1.93) -- ( 0.52,-1.93);
\draw[red][line width=1.2pt][dotted] ( 0.52,-1.93) -- ( 1.41,-1.41);
\draw[red][line width=1.2pt] ( 1.41,-1.41) -- ( 1.93,-0.52);
\draw[red][line width=1.2pt] ( 1.93,-0.52) -- ( 1.93, 0.52); \draw[red] ( 1.93,0) node[left]{$\tarc_{\hat{t}}$};
\draw[red][line width=1.2pt][dotted] ( 1.93, 0.52) -- ( 1.41, 1.41);
\draw[red][line width=1.2pt] ( 1.41, 1.41) -- ( 0.52, 1.93);
\fill [qqwuqq] (0,2) circle (2.8pt); \draw[qqwuqq] (0,2) node[above]{$\m_0$};
\fill [qqwuqq] (-1,1.732) circle (2.8pt); \draw[qqwuqq] (-1,1.732) node[above]{$\m_1$};
\fill [qqwuqq][opacity=0.25] (-1.732,1) circle (2.8pt);
\fill [qqwuqq] (-2,0) circle (2.8pt);
\fill [qqwuqq] (-1.732,-1) circle (2.8pt);
\fill [qqwuqq] (-1,-1.732) circle (2.8pt); \draw[qqwuqq] (-1,-1.732) node[left]{$\m_{t}$};
\fill [qqwuqq] (0,-2) circle (2.8pt);
\fill [qqwuqq][opacity=0.25] (1,-1.732) circle (2.8pt);
\fill [qqwuqq] (1.732,-1) circle (2.8pt);
\fill [qqwuqq] (2,0) circle (2.8pt); \draw[qqwuqq] (2,0) node[right]{$\m_{\hat{t}}$};
\fill [qqwuqq][opacity=0.25] (1.732,1) circle (2.8pt);
\fill [qqwuqq] (1,1.732) circle (2.8pt);
\fill [red] (-0.52, 1.93) circle (2.8pt); \fill[white] (-0.52, 1.93) circle (2.2pt);
                                          \draw[red]   (-0.52, 1.93) node[above]{$\c_{0,1}$};
\fill [red] (-1.41, 1.41) circle (2.8pt); \fill[white] (-1.41, 1.41) circle (2.2pt);
                                          \draw[red]   (-1.41, 1.41) node[above]{$\c_{1,2}$};
\fill [red] (-1.93, 0.52) circle (2.8pt); \fill[white] (-1.93, 0.52) circle (2.2pt);
\fill [red] (-1.93,-0.52) circle (2.8pt); \fill[white] (-1.93,-0.52) circle (2.2pt);
\fill [red] (-1.41,-1.41) circle (2.8pt); \fill[white] (-1.41,-1.41) circle (2.2pt);
\fill [red] (-0.52,-1.93) circle (2.8pt); \fill[white] (-0.52,-1.93) circle (2.2pt);
\fill [red] ( 0.52,-1.93) circle (2.8pt); \fill[white] ( 0.52,-1.93) circle (2.2pt);
\fill [red] ( 1.41,-1.41) circle (2.8pt); \fill[white] ( 1.41,-1.41) circle (2.2pt);
\fill [red] ( 1.93,-0.52) circle (2.8pt); \fill[white] ( 1.93,-0.52) circle (2.2pt);
\fill [red] ( 1.93, 0.52) circle (2.8pt); \fill[white] ( 1.93, 0.52) circle (2.2pt);
\fill [red] ( 1.41, 1.41) circle (2.8pt); \fill[white] ( 1.41, 1.41) circle (2.2pt);
\fill [red] ( 0.52, 1.93) circle (2.8pt); \fill[white] ( 0.52, 1.93) circle (2.2pt);
\draw[blue][opacity=0.7][line width=1.2pt] (0,2)--(-1,1.732); \draw[blue] (-0.5,1.86) node[below]{$\tgamma_{0,1}$};
\draw[blue][opacity=0.7][line width=1.2pt] (-1,1.732)--(-1,-1.732); \draw[blue] (-1,0) node[right]{$\tgamma_{1,t}$};
\draw[blue][opacity=0.7][line width=1.2pt] (2,0)--(-1,-1.732); \draw[blue] (0.5,-0.86) node[above]{$\tgamma_{t,\hat{t}}$};
\draw[cyan][opacity=0.7][line width=0.8pt][->] (-0.7,1.8) to[out=-90,in=180] (2.2, 1.4);
\draw[cyan] (2.2, 1.4) node[right]{$\ii_{p_1}(\tgamma_{0,1},\tarc_1)=0$};
\draw[cyan][opacity=0.7][line width=0.8pt][->] (-1,1.6) to[out=-90,in=0] (-2.2, 1.2);
\draw[cyan] (-2.2, 1.2) node[left]{$\ii_{\gamma_{1,t}\cap a_1}(\tgamma_{1,t},\tarc_1)=0$};
\draw[cyan][opacity=0.7][line width=0.8pt][->] (-1,-1.6) to[out=-90,in=0] (-2.2,-2);
\draw[cyan] (-2.2,-2) node[left]{$\ii_{\gamma_{1,t}\cap a_t}(\tgamma_{1,t},\tarc_t)=1$};
\draw[cyan][opacity=0.7][line width=0.8pt][->] (-0.9,-1.7) to[out=90,in=0] (-2.2,-0.7);
\draw[cyan] (-2.2,-0.7) node[left]{$\ii_{\gamma_{t,\hat{t}}\cap a_t}(\tgamma_{t,\hat{t}},\tarc_t)= 1$};
\draw[cyan][opacity=0.7][line width=0.8pt][->] (1.93,-0.15)  to[out=-90,in=180] (2.5,-1);
\draw[cyan] (2.5,-1) node[right]{$\ii_{\gamma_{t,\hat{t}}\cap a_t}(\tgamma_{t,\hat{t}}, \tarc_t)= 2$};
\draw (0,-2.5);
\end{tikzpicture}
\caption{ }
\label{fig-prop-gamma(0,l)-2}
\end{figure}
  
Furthermore, there exists $\gamma_{t,\hat{t}}\in\gDII$ for some $t<\hat{t}\le n$. Similarly, by Lemma \ref{lemm-int:+1 and -1}, we can compute that
  \begin{center}
  $\ii_{\gamma_{t,\hat{t}}\cap a_t}(\tgamma_{t,\hat{t}},\tarc_t)= 1$ and
  $\ii_{\gamma_{t,\hat{t}}\cap a_{\hat{t}}}(\tgamma_{t,\hat{t}},\tarc_{\hat{t}})= 2$.
 \end{center}
  By Lemma \ref{lemm:2-silt-ind}, this is a contradiction.

Now we delete two marked points $\m_0$ and $\c_{n,1}$ and a $\rbullet$-arc $a_n$ in $\SURFred{\F_{A_n}}(A_n)$. Then the new marked ribbon surface is homotopic to $\SURFred{\F_{A_{n-1}}}(A_{n-1})$, see \Pic \ref{fig-prop-gamma(0,l)-3}. Since $\gamma_{1,n} \in \gDII$, by Lemmas \ref{lemm-int:+1 and -1} and \ref{lemm:2-silt-ind}, it is easy to see that
\begin{center}
$\ii_{\gamma_{1,n}\cap a_1}(\tgamma_{1,n},\tarc_1)=0$ and $\ii_{\gamma_{1,n}\cap a_n}(\tgamma_{1,n},\tarc_n)= 1$
\end{center}
Then $X(\gtD)$ is tilting by induction.

\begin{figure}[htbp]
\definecolor{ffqqqq}{rgb}{1,0,0}
\definecolor{qqwuqq}{rgb}{0,0.5,0}
\centering
\begin{tikzpicture}[scale=0.8]
\draw[black][line width=1.2pt] (2,0) arc (0:360:2);
\draw[red][line width=1.2pt] (-0.52, 1.93) -- (-1.41, 1.41);
\draw[red][line width=1.2pt][dotted] (-1.41,1.41) -- (-1.93, 0.52);
\draw[red][line width=1.2pt] (-1.93, 0.52) -- (-1.93,-0.52);
\draw[red][line width=1.2pt] (-1.93,-0.52) -- (-1.41,-1.41);
\draw[red][line width=1.2pt] (-1.41,-1.41) -- (-0.52,-1.93);
\draw[red][line width=1.2pt][dotted] (-0.52,-1.93) -- ( 0.52,-1.93);
\draw[red][line width=1.2pt] ( 0.52,-1.93) -- ( 1.41,-1.41);
\draw[red][line width=1.2pt] ( 1.41,-1.41) -- ( 1.93,-0.52);
\draw[red][line width=1.2pt] ( 1.93,-0.52) -- ( 1.93, 0.52);
\draw[red][line width=1.2pt][dotted] ( 1.93, 0.52) -- ( 1.41, 1.41);
\draw[red][line width=1.2pt] ( 1.41, 1.41) -- ( 0.52, 1.93);
\fill [qqwuqq] (0,2) circle (2.8pt); \draw[qqwuqq] (0,2) node[above]{$\m_0$};
\fill [qqwuqq] (-1,1.732) circle (2.8pt); \draw[qqwuqq] (-1,1.732) node[above]{$\m_1$};
\fill [qqwuqq][opacity=0.25] (-1.732,1) circle (2.8pt);
\fill [qqwuqq] (-2,0) circle (2.8pt);
\fill [qqwuqq] (-1.732,-1) circle (2.8pt);
\fill [qqwuqq] (-1,-1.732) circle (2.8pt);
\fill [qqwuqq][opacity=0.25] (0,-2) circle (2.8pt);
\fill [qqwuqq] (1,-1.732) circle (2.8pt);
\fill [qqwuqq] (1.732,-1) circle (2.8pt);
\fill [qqwuqq] (2,0) circle (2.8pt);
\fill [qqwuqq][opacity=0.25] (1.732,1) circle (2.8pt);
\fill [qqwuqq] (1,1.732) circle (2.8pt);  \draw[qqwuqq] (1,1.732) node[above]{$\m_n$};
\fill [red] (-0.52, 1.93) circle (2.8pt); \fill[white] (-0.52, 1.93) circle (2.2pt);
                                          \draw[red]   (-0.52, 1.93) node[above]{$\c_{0,1}$};
\fill [red] (-1.41, 1.41) circle (2.8pt); \fill[white] (-1.41, 1.41) circle (2.2pt);
                                          \draw[red]   (-1.41, 1.41) node[left]{$\c_{1,2}$};
\fill [red] (-1.93, 0.52) circle (2.8pt); \fill[white] (-1.93, 0.52) circle (2.2pt);
\fill [red] (-1.93,-0.52) circle (2.8pt); \fill[white] (-1.93,-0.52) circle (2.2pt);
\fill [red] (-1.41,-1.41) circle (2.8pt); \fill[white] (-1.41,-1.41) circle (2.2pt);
\fill [red] (-0.52,-1.93) circle (2.8pt); \fill[white] (-0.52,-1.93) circle (2.2pt);
\fill [red] ( 0.52,-1.93) circle (2.8pt); \fill[white] ( 0.52,-1.93) circle (2.2pt);
\fill [red] ( 1.41,-1.41) circle (2.8pt); \fill[white] ( 1.41,-1.41) circle (2.2pt);
\fill [red] ( 1.93,-0.52) circle (2.8pt); \fill[white] ( 1.93,-0.52) circle (2.2pt);
\fill [red] ( 1.93, 0.52) circle (2.8pt); \fill[white] ( 1.93, 0.52) circle (2.2pt);
\fill [red] ( 1.41, 1.41) circle (2.8pt); \fill[white] ( 1.41, 1.41) circle (2.2pt);
                                          \draw[red]   ( 1.41, 1.41) node[right]{$\c_{n-1,n}$};
\fill [red] ( 0.52, 1.93) circle (2.8pt); \fill[white] ( 0.52, 1.93) circle (2.2pt);
                                          \draw[red]   ( 0.52, 1.93) node[above]{$\c_{n,1}$};
\draw[blue][opacity=0.7][line width=1.2pt] (0,2)--(-1,1.732);
\draw[blue][->] (-0.5,1.86) -- (-1.5,1.86);
\draw[blue] (-1.5,1.86) node[left]{$\tgamma_{0,1}$};
\draw[blue][opacity=0.7][line width=1.2pt] (-1,1.732)to[out=-30,in=210](1,1.732); \draw[blue] (0,1.5) node[below]{$\tgamma_{1,n}$};
\draw (0,-2.5) node{(1)};
\end{tikzpicture}
\ \ \ \ \ \
\begin{tikzpicture}[scale=0.8]
\draw[black][line width=1.2pt] (2,0) arc (0:360:2);
\draw[red][line width=1.2pt] (-0.52, 1.93) -- (-1.41, 1.41);
\draw[red][line width=1.2pt][dotted] (-1.41,1.41) -- (-1.93, 0.52);
\draw[red][line width=1.2pt] (-1.93, 0.52) -- (-1.93,-0.52);
\draw[red][line width=1.2pt] (-1.93,-0.52) -- (-1.41,-1.41);
\draw[red][line width=1.2pt] (-1.41,-1.41) -- (-0.52,-1.93);
\draw[red][line width=1.2pt][dotted] (-0.52,-1.93) -- ( 0.52,-1.93);
\draw[red][line width=1.2pt] ( 0.52,-1.93) -- ( 1.41,-1.41);
\draw[red][line width=1.2pt] ( 1.41,-1.41) -- ( 1.93,-0.52);
\draw[red][line width=1.2pt] ( 1.93,-0.52) -- ( 1.93, 0.52);
\draw[red][line width=1.2pt][dotted] ( 1.93, 0.52) -- ( 1.41, 1.41);
\fill [qqwuqq] (-1,1.732) circle (2.8pt); \draw[qqwuqq] (-1,1.732) node[left]{$\m'_1:=\m_1$};
\fill [qqwuqq][opacity=0.25] (-1.732,1) circle (2.8pt);
\fill [qqwuqq] (-2,0) circle (2.8pt);
\fill [qqwuqq] (-1.732,-1) circle (2.8pt);
\fill [qqwuqq] (-1,-1.732) circle (2.8pt);
\fill [qqwuqq][opacity=0.25] (0,-2) circle (2.8pt);
\fill [qqwuqq] (1,-1.732) circle (2.8pt);
\fill [qqwuqq] (1.732,-1) circle (2.8pt);
\fill [qqwuqq] (2,0) circle (2.8pt);
\fill [qqwuqq][opacity=0.25] (1.732,1) circle (2.8pt); \draw[qqwuqq] (1,1.732) node[right]{$\m'_0:=\m_n$};
\fill [qqwuqq] (1,1.732) circle (2.8pt);
\fill [red] (-0.52, 1.93) circle (2.8pt); \fill[white] (-0.52, 1.93) circle (2.2pt);
                                          \draw[red]   (-0.52, 1.93) node[above]{$\c_{0,1}$};
\fill [red] (-1.41, 1.41) circle (2.8pt); \fill[white] (-1.41, 1.41) circle (2.2pt);
                                          \draw[red]   (-1.41, 1.41) node[left]{$\c_{1,2}$};
\fill [red] (-1.93, 0.52) circle (2.8pt); \fill[white] (-1.93, 0.52) circle (2.2pt);
\fill [red] (-1.93,-0.52) circle (2.8pt); \fill[white] (-1.93,-0.52) circle (2.2pt);
\fill [red] (-1.41,-1.41) circle (2.8pt); \fill[white] (-1.41,-1.41) circle (2.2pt);
\fill [red] (-0.52,-1.93) circle (2.8pt); \fill[white] (-0.52,-1.93) circle (2.2pt);
\fill [red] ( 0.52,-1.93) circle (2.8pt); \fill[white] ( 0.52,-1.93) circle (2.2pt);
\fill [red] ( 1.41,-1.41) circle (2.8pt); \fill[white] ( 1.41,-1.41) circle (2.2pt);
\fill [red] ( 1.93,-0.52) circle (2.8pt); \fill[white] ( 1.93,-0.52) circle (2.2pt);
\fill [red] ( 1.93, 0.52) circle (2.8pt); \fill[white] ( 1.93, 0.52) circle (2.2pt);
\fill [red] ( 1.41, 1.41) circle (2.8pt); \fill[white] ( 1.41, 1.41) circle (2.2pt);
                                          \draw[red]   ( 1.41, 1.41) node[right]{$\c_{n-1,n}$};
\draw[blue][opacity=0.7][line width=1.2pt] (-1,1.732)to[out=-30,in=210](1,1.732);
\draw[blue] (0,1.5) node[below]{$\tgamma_{0,n}':=\tgamma_{1,n}$};
\draw[cyan] (0,0) node{$\ii_{p_1'}(\tgamma_{0,1}',\tarc_1)=0$};
\draw (0,-2.5) node{(2)};
\end{tikzpicture}
\caption{\textsf{$\m'_s=\m_s$ for all $1\le s\le n-1$ and $\m'_0=\m_n$. }}
\label{fig-prop-gamma(0,l)-3}
\end{figure}
Secondly, we assume that there are at least two graded $\gbullet$-arcs in $\gtDI$. Let
\begin{center}
 $\gtDI=\{\tgamma_{0,i_1},\tgamma_{0,i_2},\ldots,\tgamma_{0,i_r}\}$, where $1\leq i_1<\ldots<i_r\leq n$.
\end{center}
If $i_{t-1}+1=i_t$ for all $t$, then $\gtDI$ is of the form shown in \Pic \ref{fig-prop-gamma(0,l)-(4)} (1). Similar to the first case, we obtain that
$\gamma_{i_t,n}\in\gDII$ satisfies that
    \begin{center}
    $\ii_{\gamma_{i_t,n}\cap a_{i_t}}(\tgamma_{i_t,n}, \tarc_{i_t})=1$ and
      $\ii_{\gamma_{i_t,n}\cap a_{n}}(\tgamma_{i_t,n}, \tarc_{n})=1$,
   \end{center} for any $x<i_t$, $\gamma_{0,x}\in\gDI$.
  It follows that $X(\gtD)$ is tilting by induction.

\begin{figure}[htbp]
\definecolor{ffqqqq}{rgb}{1,0,0}
\definecolor{qqwuqq}{rgb}{0,0.5,0}
\centering
\ \ \ \
\begin{tikzpicture}[scale=0.8]
\draw[black][line width=1.2pt] (2,0) arc (0:360:2);
\draw[red][line width=1.2pt] (-0.52, 1.93) -- (-1.41, 1.41);
\draw[red][line width=1.2pt][dotted] (-1.41,1.41) -- (-1.93, 0.52);
\draw[red][line width=1.2pt] (-1.93, 0.52) -- (-1.93,-0.52);
\draw[red][line width=1.2pt] (-1.93,-0.52) -- (-1.41,-1.41);
\draw[red][line width=1.2pt] (-1.41,-1.41) -- (-0.52,-1.93);
\draw[red][line width=1.2pt][dotted] (-0.52,-1.93) -- ( 0.52,-1.93);
\draw[red][line width=1.2pt] ( 0.52,-1.93) -- ( 1.41,-1.41);
\draw[red][line width=1.2pt] ( 1.41,-1.41) -- ( 1.93,-0.52);
\draw[red][line width=1.2pt] ( 1.93,-0.52) -- ( 1.93, 0.52);
\draw[red][line width=1.2pt][dotted] ( 1.93, 0.52) -- ( 1.41, 1.41);
\draw[red][line width=1.2pt] ( 1.41, 1.41) -- ( 0.52, 1.93);
\fill [qqwuqq] (0,2) circle (2.8pt); \draw[qqwuqq] (0,2) node[above]{$\m_0$};
\fill [qqwuqq] (-1,1.732) circle (2.8pt); \draw[qqwuqq] (-1,1.732) node[above]{$\m_1$};
\fill [qqwuqq][opacity=0.25] (-1.732,1) circle (2.8pt);
\fill [qqwuqq] (-2,0) circle (2.8pt);
\fill [qqwuqq] (-1.732,-1) circle (2.8pt);
\fill [qqwuqq] (-1,-1.732) circle (2.8pt);
\fill [qqwuqq][opacity=0.25] (0,-2) circle (2.8pt);
\fill [qqwuqq] (1,-1.732) circle (2.8pt); \draw[qqwuqq] (1,-1.732) node[right]{$\m_{i_{t-1}}$};
\fill [qqwuqq] (1.732,-1) circle (2.8pt); \draw[qqwuqq] (1.732,-1) node[right]{$\m_{i_{t}}$};
\fill [qqwuqq] (2,0) circle (2.8pt);
\fill [qqwuqq][opacity=0.25] (1.732,1) circle (2.8pt);
\fill [qqwuqq] (1,1.732) circle (2.8pt); \draw[qqwuqq] (1,1.732) node[above]{$\m_n$};
\fill [red] (-0.52, 1.93) circle (2.8pt); \fill[white] (-0.52, 1.93) circle (2.2pt);
\fill [red] (-1.41, 1.41) circle (2.8pt); \fill[white] (-1.41, 1.41) circle (2.2pt);
\fill [red] (-1.93, 0.52) circle (2.8pt); \fill[white] (-1.93, 0.52) circle (2.2pt);
\fill [red] (-1.93,-0.52) circle (2.8pt); \fill[white] (-1.93,-0.52) circle (2.2pt);
\fill [red] (-1.41,-1.41) circle (2.8pt); \fill[white] (-1.41,-1.41) circle (2.2pt);
\fill [red] (-0.52,-1.93) circle (2.8pt); \fill[white] (-0.52,-1.93) circle (2.2pt);
\fill [red] ( 0.52,-1.93) circle (2.8pt); \fill[white] ( 0.52,-1.93) circle (2.2pt);
\fill [red] ( 1.41,-1.41) circle (2.8pt); \fill[white] ( 1.41,-1.41) circle (2.2pt);
                                          \draw[red]   ( 1.41,-1.41) node[right]{$\c_{i_{t-1},i_t}$};
\fill [red] ( 1.93,-0.52) circle (2.8pt); \fill[white] ( 1.93,-0.52) circle (2.2pt);
\fill [red] ( 1.93, 0.52) circle (2.8pt); \fill[white] ( 1.93, 0.52) circle (2.2pt);
\fill [red] ( 1.41, 1.41) circle (2.8pt); \fill[white] ( 1.41, 1.41) circle (2.2pt);
\fill [red] ( 0.52, 1.93) circle (2.8pt); \fill[white] ( 0.52, 1.93) circle (2.2pt);
\draw[blue][opacity=0.75][line width=1.2pt] (0,2) -- (-1,1.732);
\draw[blue][opacity=0.75][line width=1.2pt][dotted] (0,2) -- (-1.732,1);
\draw[blue][opacity=0.75][line width=1.2pt] (0,2) -- (-2,0);
\draw[blue][opacity=0.75][line width=1.2pt] (0,2) -- (-1.732,-1);
\draw[blue][opacity=0.75][line width=1.2pt] (0,2) -- (-1,-1.732);
\draw[blue][opacity=0.75][line width=1.2pt][dotted] (0,2) -- (0,-2);
\draw[blue][opacity=0.75][line width=1.2pt] (0,2) -- (1,-1.732);
\draw[blue][opacity=0.75][line width=1.2pt] (0,2) -- (1.732,-1);
\draw[blue][->] (0.86,0.5) -- (0.86+0.5,0.5+0.5) -- (2,0.5+0.5);
\draw[blue] (2,0.5+0.5) node[right]{$\tgamma_{0,i_{t}}$};
\draw[blue][opacity=0.75][line width=1.2pt] (1.732,-1) -- (1,1.732);
\draw[blue][->] (1.3, 0.5) -- (2.4, 0.5);
\draw[blue] (2.4, 0.5) node[right]{$\tgamma_{i_t,n}$};
\draw[pattern color=blue, pattern=north west lines][opacity=0.5] (0,2) -- (1,-1.732) to[out=30,in=225] (1.732,-1) -- (0,2);
\draw (0,-2.5) node{(1) shadow part is a triangle};
\end{tikzpicture}
\ \
\begin{tikzpicture}[scale=0.8]
\draw[black][line width=1.2pt] (2,0) arc (0:360:2);
\draw[red][line width=1.2pt] (-0.52, 1.93) -- (-1.41, 1.41);
\draw[red][line width=1.2pt][dotted] (-1.41,1.41) -- (-1.93, 0.52);
\draw[red][line width=1.2pt] (-1.93, 0.52) -- (-1.93,-0.52);
\draw[red][line width=1.2pt] (-1.93,-0.52) -- (-1.41,-1.41);
\draw[red][line width=1.2pt] (-1.41,-1.41) -- (-0.52,-1.93);
\draw[red][line width=1.2pt][dotted] (-0.52,-1.93) -- ( 0.52,-1.93);
\draw[red][line width=1.2pt] ( 0.52,-1.93) -- ( 1.41,-1.41);
\draw[red][line width=1.2pt] ( 1.41,-1.41) -- ( 1.93,-0.52);
\draw[red][line width=1.2pt] ( 1.93,-0.52) -- ( 1.93, 0.52);
\draw[red][line width=1.2pt][dotted] ( 1.93, 0.52) -- ( 1.41, 1.41);
\draw[red][line width=1.2pt] ( 1.41, 1.41) -- ( 0.52, 1.93);
\fill [qqwuqq] (0,2) circle (2.8pt); \draw[qqwuqq] (0,2.8) node{$(\m'_0=)\m_0$}; \draw[qqwuqq][->] (0,2.2)--(0,2.45);
\fill [qqwuqq] (-1,1.732) circle (2.8pt); \draw[qqwuqq] (-1,1.732) node[above]{$\m_1$};
\fill [qqwuqq][opacity=0.25] (-1.732,1) circle (2.8pt);
\fill [qqwuqq] (-2,0) circle (2.8pt);  \draw[qqwuqq] (-2,0) node[left]{$(\m'_1=)\m_{i_{t-1}}$};
\fill [qqwuqq] (-1.732,-1) circle (2.8pt);
\fill [qqwuqq] (-1,-1.732) circle (2.8pt);
\fill [qqwuqq][opacity=0.25] (0,-2) circle (2.8pt);
\fill [qqwuqq] (1,-1.732) circle (2.8pt);
\fill [qqwuqq] (1.732,-1) circle (2.8pt);
\fill [qqwuqq] (2,0) circle (2.8pt); \draw[qqwuqq] (2,0) node[right]{$\m_{i_t}(=\m'_{n'})$};
\fill [qqwuqq][opacity=0.25] (1.732,1) circle (2.8pt);
\fill [qqwuqq] (1,1.732) circle (2.8pt); \draw[qqwuqq] (1,1.732) node[above]{$\m_n$};
\fill [red] (-0.52, 1.93) circle (2.8pt); \fill[white] (-0.52, 1.93) circle (2.2pt);
                                          \draw[red]   (-0.52, 1.93) node[above]{$\c_{0,1}$};
\fill [red] (-1.41, 1.41) circle (2.8pt); \fill[white] (-1.41, 1.41) circle (2.2pt);
                                          \draw[red]   (-1.41, 1.41) node[above]{$\c_{1,2}$};
\fill [red] (-1.93, 0.52) circle (2.8pt); \fill[white] (-1.93, 0.52) circle (2.2pt);
                                          \draw[red]   (-1.93, 0.52) node[left]{$\c_{i_{t-1}-1,i_{t-1}}$};
\fill [red] (-1.93,-0.52) circle (2.8pt); \fill[white] (-1.93,-0.52) circle (2.2pt);
\fill [red] (-1.41,-1.41) circle (2.8pt); \fill[white] (-1.41,-1.41) circle (2.2pt);
\fill [red] (-0.52,-1.93) circle (2.8pt); \fill[white] (-0.52,-1.93) circle (2.2pt);
\fill [red] ( 0.52,-1.93) circle (2.8pt); \fill[white] ( 0.52,-1.93) circle (2.2pt);
\fill [red] ( 1.41,-1.41) circle (2.8pt); \fill[white] ( 1.41,-1.41) circle (2.2pt);
\fill [red] ( 1.93,-0.52) circle (2.8pt); \fill[white] ( 1.93,-0.52) circle (2.2pt);
\fill [red] ( 1.93, 0.52) circle (2.8pt); \fill[white] ( 1.93, 0.52) circle (2.2pt);
                                          \draw[red]   ( 1.93, 0.52) node[right]{$\c_{i_t,i_t+1}$};
\fill [red] ( 1.41, 1.41) circle (2.8pt); \fill[white] ( 1.41, 1.41) circle (2.2pt);
                                          \draw[red]   ( 1.41-0.2, 1.41+0.4) node[right]{$\c_{n-1,n}$};
\fill [red] ( 0.52, 1.93) circle (2.8pt); \fill[white] ( 0.52, 1.93) circle (2.2pt);
                                          \draw[red]   ( 0.52, 1.93) node[above]{$\c_{n,0}$};
\draw[blue][line width=1.2pt] (0,2) -- (-2,0);
\draw[blue][->] (-1,1) -- (-1.5,1.5) -- (-2,1.5);
\draw[blue] (-2,1.4) node[left]{$\tgamma_{0,i_{t-1}}$};
\draw[blue][line width=1.2pt] (0,2) -- ( 2,0);
\draw[blue][->] (1,1) -- (1.5,1.5) -- (2,1.5);
\draw[blue] (2,1.4) node[right]{$\tgamma_{0,i_{t}}$};
%
\draw[blue] (0,0) node{$\Surf_A^{0,i_{t-1},i_t}$};
\draw (0,-2.5) node{(2) $n'=i_t-i_{t-1}+1$};
\end{tikzpicture}
\caption{ }
\label{fig-prop-gamma(0,l)-(4)}
\end{figure}

If $i_{t-1}+1<i_t$ for some $t$, then we set $\m_0=\m'_0$, $\m_{i_{t-1}}=\m'_1,\ldots,\m_{i_t}=\m'_{n'}$, where $n'=i_t-i_{t-1}+1$, see \Pic \ref{fig-prop-gamma(0,l)-(4)} (2). In this case, $\Surf_A^{0,i_{t-1},i_t}$ can be viewed as a marked ribbon surface $\SURFred{\F_{A_{n}}}(A_{n'})$ of $A_{n'}$ through delete the marked points
  \begin{center}
  $\{\c_{n,0}\} \cup \big\{\c_{r,r+1}|{0\leq r<i_{t-1}-1 \atop i_t< r<n}\big\}
                \cup \big\{\m_s|{0<s<i_{t-1} \atop i_t<s\leq n}\big\}$.
  \end{center}
Denote by $\gtD|_{0,i_{t-1},i_t}$ the restriction of the $\gbullet$-grFFAS $\gtD$ on $\Surf_A^{0,i_{t-1},i_t}$.
Then $\gtD|_{0,i_{t-1},i_t}$ is a $\gbullet$-grFFAS and contains no graded $\gbullet$-arc $\tgamma'_{0,x}$ whose endpoints are $\m'_0$ and $\m'_x$ for any $1<x<n'$.

Let $\tgamma'_{x,y}$ be the graded $\gbullet$-arc with endpoints $\m'_x$ and $\m'_y$ for any $1\le x\ne y\le n'$. Next we claim that $\tgamma'_{1,n'-1}$ is a graded $\gbullet$-curve in $\gtD|_{0,i_{t-1},i_t}$ (see \Pic \ref{fig-prop-gamma(0,l)-(6)} (1)) such that
      \[\ii_{\gamma'_{1,n'-1}\cap a'_1}(\tgamma'_{1,n'-1}, \tarc'_1)=0
      \ \text{and}\ \ii_{\gamma'_{1,n'-1}\cap a'_{n'-1}}(\tgamma'_{1,n'-1}, \tarc'_{n'-1})=1\]
Otherwise, $\gtD|_{0,i_{t-1},i_t}$ contains at least one of $\tgamma'_{1,x}$ and $\tgamma'_{y,n'}$ for some $x$ and $y$, where $2\le x, y\le n'-1$. By Lemmas \ref{lemm-int:+1 and -1} and \ref{lemm:2-silt-ind}, $\gamma'_{y,n'}\notin \gD|_{0,i_{t-1},i_t}$. Furthermore, there are two positive integers $v$ and $w$ such that $\gamma'_{1,v}, \gamma'_{v,w}\in \gD|_{0,i_{t-1},i_t}$, where $2\le v<w\le n'-1$, (see \Pic \ref{fig-prop-gamma(0,l)-(6)} (2)).
By Lemmas \ref{lemm-int:+1 and -1} and \ref{lemm:2-silt-ind}, we have
   \[\ii_{\gamma'_{v,w}\cap a'_{w}}(\tgamma'_{v,w}, \tarc'_{w})
       = \ii_{\gamma'_{v,w}\cap a'_1}(\tgamma'_{v,w}, \tarc'_v) + 1 = \ii_{\gamma'_{1,v}\cap a'_v}(\tgamma'_{1,v}, \tarc'_v) + 1= 2. \]
This is a contradiction.

       \begin{figure}[htbp]
\definecolor{ffqqqq}{rgb}{1,0,0}
\definecolor{qqwuqq}{rgb}{0,0.5,0}
\centering
\begin{tikzpicture}[scale=0.8]
\draw[black][line width=1.2pt] (2,0) arc (0:360:2);
\draw[red][line width=1.2pt] (-1.93, 0.52) -- (-1.93,-0.52);
\draw[red][line width=1.2pt] (-1.93,-0.52) -- (-1.41,-1.41);
\draw[red][line width=1.2pt] (-1.41,-1.41) -- (-0.52,-1.93);
\draw[red][line width=1.2pt][dotted] (-0.52,-1.93) -- ( 0.52,-1.93);
\draw[red][line width=1.2pt] ( 0.52,-1.93) -- ( 1.41,-1.41);
\draw[red][line width=1.2pt] ( 1.41,-1.41) -- ( 1.93,-0.52);
\draw[red][line width=1.2pt] ( 1.93,-0.52) -- ( 1.93, 0.52);
\fill [qqwuqq] (0,2) circle (2.8pt); \draw[qqwuqq] (0,2) node[above]{$\m'_0$};
\fill [qqwuqq] (-2,0) circle (2.8pt);  \draw[qqwuqq] (-2,0) node[left]{$\m'_1$};
\fill [qqwuqq] (-1.732,-1) circle (2.8pt);
\fill [qqwuqq] (-1,-1.732) circle (2.8pt);
\fill [qqwuqq][opacity=0.25] (0,-2) circle (2.8pt);
\fill [qqwuqq] (1,-1.732) circle (2.8pt);
\fill [qqwuqq] (1.732,-1) circle (2.8pt); \draw[qqwuqq] (1.732,-1) node[right]{$\m'_{n'-1}$};
\fill [qqwuqq] (2,0) circle (2.8pt); \draw[qqwuqq] (2,0) node[right]{$\m'_{n'}$};
\fill [red] ( 1.93, 0.52) circle (2.8pt); \fill[white] ( 1.93, 0.52) circle (2.2pt);
                                          \draw[red]   ( 1.93, 0.52) node[right]{$\c_{i_t,i_t+1}$};
\fill [red] (-1.93,-0.52) circle (2.8pt); \fill[white] (-1.93,-0.52) circle (2.2pt);
\fill [red] (-1.41,-1.41) circle (2.8pt); \fill[white] (-1.41,-1.41) circle (2.2pt);
\fill [red] (-0.52,-1.93) circle (2.8pt); \fill[white] (-0.52,-1.93) circle (2.2pt);
\fill [red] ( 0.52,-1.93) circle (2.8pt); \fill[white] ( 0.52,-1.93) circle (2.2pt);
\fill [red] ( 1.41,-1.41) circle (2.8pt); \fill[white] ( 1.41,-1.41) circle (2.2pt);
\fill [red] ( 1.93,-0.52) circle (2.8pt); \fill[white] ( 1.93,-0.52) circle (2.2pt);
\fill [red] (-1.93, 0.52) circle (2.8pt); \fill[white] (-1.93, 0.52) circle (2.2pt);
                                          \draw[red]   (-1.93, 0.52) node[left]{$\c_{i_t-1,i_t}$};
\draw[blue][line width=1.2pt] (0,2) -- (-2,0);
\draw[blue][->] (-1,1) -- (-1.5,1.5) -- (-2,1.5);
\draw[blue] (-2,1.4) node[left]{$\tgamma'_{0,1}$};
\draw[blue][line width=1.2pt] (0,2) -- ( 2,0);
\draw[blue][->] (1,1) -- (1.5,1.5) -- (2,1.5);
\draw[blue] (2,1.4) node[right]{$\tgamma'_{0,n'}$};
\draw[blue][line width=1.2pt] (-2,0) -- (1.732,-1); \draw[blue] (-0.18,-0.5) node[below]{$\tgamma'_{1,n'-1}$};
\draw (0,-2.5) node{(1)};
\end{tikzpicture}
\ \ \ \ \ \
\begin{tikzpicture}[scale=0.8]
\draw[black][line width=1.2pt] (2,0) arc (0:360:2);
\draw[red][line width=1.2pt] (-1.93, 0.52) -- (-1.93,-0.52);
\draw[red][line width=1.2pt][dotted] (-1.93,-0.52) -- (-1.41,-1.41);
\draw[red][line width=1.2pt] (-1.41,-1.41) -- (-0.52,-1.93);
\draw[red][line width=1.2pt][dotted] (-0.52,-1.93) -- ( 0.52,-1.93);
\draw[red][line width=1.2pt] ( 0.52,-1.93) -- ( 1.41,-1.41);
\draw[red][line width=1.2pt][dotted] ( 1.41,-1.41) -- ( 1.93,-0.52);
\draw[red][line width=1.2pt] ( 1.93,-0.52) -- ( 1.93, 0.52);
\fill [qqwuqq] (0,2) circle (2.8pt); \draw[qqwuqq] (0,2) node[above]{$\m'_0$};
\fill [qqwuqq] (-2,0) circle (2.8pt);  \draw[qqwuqq] (-2,0) node[left]{$\m'_1$};
\fill [qqwuqq][opacity=0.25] (-1.732,-1) circle (2.8pt);
\fill [qqwuqq] (-1,-1.732) circle (2.8pt); \fill [qqwuqq] (-1,-1.732) node[below]{$\m_v$};
\fill [qqwuqq][opacity=0.25] (0,-2) circle (2.8pt);
\fill [qqwuqq] (1,-1.732) circle (2.8pt); \fill [qqwuqq] (1,-1.732) node[below]{$\m_w$};
\fill [qqwuqq][opacity=0.25] (1.732,-1) circle (2.8pt);
\fill [qqwuqq] (2,0) circle (2.8pt); \draw[qqwuqq] (2,0) node[right]{$\m'_{n'}$};
\fill [red] ( 1.93, 0.52) circle (2.8pt); \fill[white] ( 1.93, 0.52) circle (2.2pt);
                                          \draw[red]   ( 1.93, 0.52) node[right]{$\c_{i_t,i_t+1}$};
\fill [red] (-1.93,-0.52) circle (2.8pt); \fill[white] (-1.93,-0.52) circle (2.2pt);
\fill [red] (-1.41,-1.41) circle (2.8pt); \fill[white] (-1.41,-1.41) circle (2.2pt);
\fill [red] (-0.52,-1.93) circle (2.8pt); \fill[white] (-0.52,-1.93) circle (2.2pt);
\fill [red] ( 0.52,-1.93) circle (2.8pt); \fill[white] ( 0.52,-1.93) circle (2.2pt);
\fill [red] ( 1.41,-1.41) circle (2.8pt); \fill[white] ( 1.41,-1.41) circle (2.2pt);
\fill [red] ( 1.93,-0.52) circle (2.8pt); \fill[white] ( 1.93,-0.52) circle (2.2pt);
\fill [red] (-1.93, 0.52) circle (2.8pt); \fill[white] (-1.93, 0.52) circle (2.2pt);
                                          \draw[red]   (-1.93, 0.52) node[left]{$\c_{i_t-1,i_t}$};
\draw[blue][line width=1.2pt] (0,2) -- (-2,0);
\draw[blue][->] (-1,1) -- (-1.5,1.5) -- (-2,1.5);
\draw[blue] (-2,1.4) node[left]{$\tgamma'_{0,1}$};
\draw[blue][line width=1.2pt] (0,2) -- ( 2,0);
\draw[blue][->] (1,1) -- (1.5,1.5) -- (2,1.5);
\draw[blue] (2,1.4) node[right]{$\tgamma'_{0,n'}$};
\draw[blue][line width=1.2pt] (-2,0) -- (-1,-1.732); \draw[blue] (-1.5,-0.86) node[right]{$\tgamma'_{1,v}$};
\draw[blue][line width=1.2pt] (-1,-1.732) -- (1,-1.732); \draw[blue] (0,-1.732) node[above]{$\tgamma'_{v,w}$};
\draw (0,-2.5) node{(2)};
\end{tikzpicture}
\caption{\textsf{ }}
\label{fig-prop-gamma(0,l)-(6)}
\end{figure}
At last, similar to the first case, it is easy to show that $X(\gtD|_{0, i_{t-1}, i_t})$ is tilting by induction.
\end{proof}

Now we can give a classification of the silted but not tilted algebras of type $A_n$.

\begin{theorem} \label{coro-gamma(0,l)}
Let $B$ be a finite-dimensional algebra. Then $B$ is a silted algebra of type $A_n$ but not tilted of type $A_n$ if and only if the quiver of $B$ is non-connected. Furthermore, $B$ is a direct product of two tilted algebras  $B'$ and $B''$, where $B'$ and $B''$ are tilted algebras of types $A_{n'}$ and $A_{n''}$ such that $n'+n''=n$, respectively.
\end{theorem}

\begin{proof}
Let $S=X(\gtD)=\bigoplus_{\tgamma\in\gtD} X(\gtD)$ be a 2-term silting complex and $B = \End_{\per A_n} (S)$. Note that all tilted algebras are connected. Assume that $B$ is not tilted, we obtain two graded $\gbullet$-arc $\tgamma_{0,l}$ and $\tgamma_{0,k}$ satisfy the conditions (1) and (2) in Lemma \ref{prop-gamma(0,l)}.
Then the quiver of $A^S$ has an arrow $\alpha: \gamma_{0,l}\to\gamma_{0,k}$ with $|\alpha|=1$, see \Pic \ref{fig-lemm-silt is a prod of two tilt}.
By Corollary \ref{coro-silted}, $\H^0(A^S)\cong B$, this shows that $\alpha$ is not an arrow of the quiver of $B$.
Notice that the surface $\Surf(A_n)$ is a disk, this shows that the quiver of $A^S$ is a tree \cite[Corollary 1.23]{OpperPlamondonSchroll18}.
So it is clear that the quiver of $B$ is not connected.

Let $\PP$ be the elementary $\gbullet$-polygon with edges $\tgamma_{0,l}$ and $\tgamma_{0,k}$, and $c$ the curve with endpoints $\m_0$ and $\c$, see \Pic \ref{fig-lemm-silt is a prod of two tilt}.
 Then $c$ divides $\SURFred{\F_{A_n}}(A_n)$ to two parts which are homotopic to $\SURF'=\SURFred{\F_{A_n}}(A_{n'})$ and $\SURF''=\SURFred{\F_{A_n}}(A_{n''})$, where $n'+n''=n$.

\begin{figure}[htbp]
\definecolor{ffqqqq}{rgb}{1,0,0}
\definecolor{qqwuqq}{rgb}{0,0.5,0}
\centering
\begin{tikzpicture}[scale=0.8]
\draw[black][line width=1.2pt][opacity=0.25][dotted] (2,0) arc (0:360:2);
\draw[black][line width=1.2pt] (0,2) arc (90:345:2);
\draw[red][line width=1.2pt] (-0.52, 1.93) -- (-1.41, 1.41);
\draw[red][line width=1.2pt][dotted] (-1.41,1.41) -- (-1.93, 0.52);
\draw[red][line width=1.2pt] (-1.93, 0.52) -- (-1.93,-0.52);
\draw[red][line width=1.2pt] (-1.93,-0.52) -- (-1.41,-1.41);
\draw[red][line width=1.2pt] (-1.41,-1.41) -- (-0.52,-1.93);
\draw[red][line width=1.2pt][dotted] (-0.52,-1.93) -- ( 0.52,-1.93);
\draw[red][line width=1.2pt] ( 0.52,-1.93) -- ( 1.41,-1.41);
\draw[red][line width=1.2pt] ( 1.41,-1.41) -- ( 1.93,-0.52);
\fill [qqwuqq] (0,2) circle (2.8pt); \draw[qqwuqq] (0,2) node[above]{$\m_0$};
\fill [qqwuqq] (-1,1.732) circle (2.8pt); \draw[qqwuqq] (-1,1.732) node[above]{$\m_1$};
\fill [qqwuqq][opacity=0.25] (-1.732,1) circle (2.8pt);
\fill [qqwuqq] (-2,0) circle (2.8pt);  \draw[qqwuqq] (-2,0) node[left]{$\m_l$};
\fill [qqwuqq] (-1.732,-1) circle (2.8pt);
\fill [qqwuqq] (-1,-1.732) circle (2.8pt);
\fill [qqwuqq][opacity=0.25] (0,-2) circle (2.8pt);
\fill [qqwuqq] (1,-1.732) circle (2.8pt);
\fill [qqwuqq] (1.732,-1) circle (2.8pt);
\fill [red] (-0.52, 1.93) circle (2.8pt); \fill[white] (-0.52, 1.93) circle (2.2pt);
                                          \draw[red]   (-0.52, 1.93) node[above]{$\c_{0,1}$};
\fill [red] (-1.41, 1.41) circle (2.8pt); \fill[white] (-1.41, 1.41) circle (2.2pt);
                                          \draw[red]   (-1.41, 1.41) node[left]{$\c_{1,2}$};
\fill [red] (-1.93, 0.52) circle (2.8pt); \fill[white] (-1.93, 0.52) circle (2.2pt);
                                          \draw[red]   (-1.93, 0.52) node[left]{$\c_{l-1,l}$};
\fill [red] (-1.93,-0.52) circle (2.8pt); \fill[white] (-1.93,-0.52) circle (2.2pt);
\fill [red] (-1.41,-1.41) circle (2.8pt); \fill[white] (-1.41,-1.41) circle (2.2pt);
\fill [red] (-0.52,-1.93) circle (2.8pt); \fill[white] (-0.52,-1.93) circle (2.2pt);
\fill [red] ( 0.52,-1.93) circle (2.8pt); \fill[white] ( 0.52,-1.93) circle (2.2pt);
\fill [red] ( 1.41,-1.41) circle (2.8pt); \fill[white] ( 1.41,-1.41) circle (2.2pt);
\fill [red] ( 1.93,-0.52) circle (2.8pt); \fill[white] ( 1.93,-0.52) circle (2.2pt);
                                          \draw[red]   ( 1.93,-0.52) node[right]{$\c$};
\fill[pink][opacity=0.5] (0,2) arc (90:345:2);
\draw[blue][line width=1.2pt] (0,2) -- (-2,0); 
\draw[blue][line width=1.2pt][dotted] (-2,0) -- (1.73,-1); 
\draw[pattern color=blue, pattern=north west lines][opacity=0.5] (0,2) -- (-2,0) -- (1.73,-1) to[out=60,in=-105] (1.93,-0.52) -- (0,2);
\draw[black] [line width=1.2pt] (0,2) -- ( 1.93,-0.52);
\draw (0,-2.5) node{$\SURF'$};
\end{tikzpicture}
\ \
\begin{tikzpicture}[scale=0.8]
\draw[black][line width=1.2pt] (2,0) arc (0:360:2);
\draw[red][line width=1.2pt] (-0.52, 1.93) -- (-1.41, 1.41);
\draw[red][line width=1.2pt][dotted] (-1.41,1.41) -- (-1.93, 0.52);
\draw[red][line width=1.2pt] (-1.93, 0.52) -- (-1.93,-0.52);
\draw[red][line width=1.2pt] (-1.93,-0.52) -- (-1.41,-1.41);
\draw[red][line width=1.2pt] (-1.41,-1.41) -- (-0.52,-1.93);
\draw[red][line width=1.2pt][dotted] (-0.52,-1.93) -- ( 0.52,-1.93);
\draw[red][line width=1.2pt] ( 0.52,-1.93) -- ( 1.41,-1.41);
\draw[red][line width=1.2pt] ( 1.41,-1.41) -- ( 1.93,-0.52);
\draw[red][line width=1.2pt] ( 1.93,-0.52) -- ( 1.93, 0.52);
\draw[red][line width=1.2pt][dotted] ( 1.93, 0.52) -- ( 1.41, 1.41);
\draw[red][line width=1.2pt] ( 1.41, 1.41) -- ( 0.52, 1.93);
\fill [qqwuqq] (0,2) circle (2.8pt); \draw[qqwuqq] (0,2) node[above]{$\m_0$};
\fill [qqwuqq] (-1,1.732) circle (2.8pt); \draw[qqwuqq] (-1,1.732) node[above]{$\m_1$};
\fill [qqwuqq][opacity=0.25] (-1.732,1) circle (2.8pt);
\fill [qqwuqq] (-2,0) circle (2.8pt);  \draw[qqwuqq] (-2,0) node[left]{$\m_l$};
\fill [qqwuqq] (-1.732,-1) circle (2.8pt);
\fill [qqwuqq] (-1,-1.732) circle (2.8pt);
\fill [qqwuqq][opacity=0.25] (0,-2) circle (2.8pt);
\fill [qqwuqq] (1,-1.732) circle (2.8pt);
\fill [qqwuqq] (1.732,-1) circle (2.8pt);
\fill [qqwuqq] (2,0) circle (2.8pt); \draw[qqwuqq] (2,0) node[right]{$\m_{k}$};
\fill [qqwuqq][opacity=0.25] (1.732,1) circle (2.8pt);
\fill [qqwuqq] (1,1.732) circle (2.8pt); \draw[qqwuqq] (1,1.732) node[above]{$\m_n$};
\fill [red] (-0.52, 1.93) circle (2.8pt); \fill[white] (-0.52, 1.93) circle (2.2pt);
                                          \draw[red]   (-0.52, 1.93) node[above]{$\c_{0,1}$};
\fill [red] (-1.41, 1.41) circle (2.8pt); \fill[white] (-1.41, 1.41) circle (2.2pt);
                                          \draw[red]   (-1.41, 1.41) node[left]{$\c_{1,2}$};
\fill [red] (-1.93, 0.52) circle (2.8pt); \fill[white] (-1.93, 0.52) circle (2.2pt);
                                          \draw[red]   (-1.93, 0.52) node[left]{$\c_{l-1,l}$};
\fill [red] (-1.93,-0.52) circle (2.8pt); \fill[white] (-1.93,-0.52) circle (2.2pt);
\fill [red] (-1.41,-1.41) circle (2.8pt); \fill[white] (-1.41,-1.41) circle (2.2pt);
\fill [red] (-0.52,-1.93) circle (2.8pt); \fill[white] (-0.52,-1.93) circle (2.2pt);
\fill [red] ( 0.52,-1.93) circle (2.8pt); \fill[white] ( 0.52,-1.93) circle (2.2pt);
\fill [red] ( 1.41,-1.41) circle (2.8pt); \fill[white] ( 1.41,-1.41) circle (2.2pt);
\fill [red] ( 1.93,-0.52) circle (2.8pt); \fill[white] ( 1.93,-0.52) circle (2.2pt);
                                          \draw[red]   ( 1.93,-0.52) node[right]{$\c$};
\fill [red] ( 1.93, 0.52) circle (2.8pt); \fill[white] ( 1.93, 0.52) circle (2.2pt);
                                          \draw[red]   ( 1.93, 0.52) node[right]{$\c_{k,k+1}$};
\fill [red] ( 1.41, 1.41) circle (2.8pt); \fill[white] ( 1.41, 1.41) circle (2.2pt);
                                          \draw[red]   ( 1.41, 1.41) node[right]{$\c_{n-1,n}$};
\fill [red] ( 0.52, 1.93) circle (2.8pt); \fill[white] ( 0.52, 1.93) circle (2.2pt);
                                          \draw[red]   ( 0.52, 1.93) node[above]{$\c_{n,0}$};
\fill[pink][opacity=0.5] (0,2) arc (90:345:2);
\shade[top color=orange, bottom color = yellow][opacity=0.5] (0,2) arc (90:-15:2);
\draw[blue][line width=1.2pt] (0,2) -- (-2,0); 
\draw[blue][line width=1.2pt] (0,2) -- ( 2,0); 
\draw[blue][line width=0.8pt][->] (-1,1) -- (1,1); \draw[blue] (0,1) node[above]{$\alpha$};
\draw[blue][line width=1.2pt][dotted] (-2,0) -- (1.73,-1); 
\draw[pattern color=blue, pattern=north west lines][opacity=0.5] (0,2) -- (-2,0)  -- (1.73,-1) to[out=60,in=-90] (2,0) -- (0,2);
\draw[blue] (0,0) node{$\PP$};
\draw[black] [line width=1.2pt] (0,2) -- ( 1.93,-0.52);
\draw (0,-2.5) node{$\SURFred{\F_{A_n}}(A_n)$};
\end{tikzpicture}
\ \
\begin{tikzpicture}[scale=0.8]
\draw[black][line width=1.2pt][opacity=0.25][dotted] (1.414,-1.414) arc (-45:135:2);
\draw[black][line width=1.2pt] (0,2) arc (90:-15:2);
\draw[red][line width=1.2pt] ( 1.93,-0.52) -- ( 1.93, 0.52);
\draw[red][line width=1.2pt][dotted] ( 1.93, 0.52) -- ( 1.41, 1.41);
\draw[red][line width=1.2pt] ( 1.41, 1.41) -- ( 0.52, 1.93);
\fill [qqwuqq] (0,2) circle (2.8pt); \draw[qqwuqq] (0,2) node[above]{$\m_0$};
\fill [qqwuqq] (2,0) circle (2.8pt); \draw[qqwuqq] (2,0) node[right]{$\m_{k}$};
\fill [qqwuqq][opacity=0.25] (1.732,1) circle (2.8pt);
\fill [qqwuqq] (1,1.732) circle (2.8pt); \draw[qqwuqq] (1,1.732) node[above]{$\m_n$};
\fill [red] ( 1.93,-0.52) circle (2.8pt); \fill[white] ( 1.93,-0.52) circle (2.2pt);
                                          \draw[red]   ( 1.93,-0.52) node[right]{$\c$};
\fill [red] ( 1.93, 0.52) circle (2.8pt); \fill[white] ( 1.93, 0.52) circle (2.2pt);
                                          \draw[red]   ( 1.93, 0.52) node[right]{$\c_{k,k+1}$};
\fill [red] ( 1.41, 1.41) circle (2.8pt); \fill[white] ( 1.41, 1.41) circle (2.2pt);
                                          \draw[red]   ( 1.41, 1.41) node[right]{$\c_{n-1,n}$};
\fill [red] ( 0.52, 1.93) circle (2.8pt); \fill[white] ( 0.52, 1.93) circle (2.2pt);
                                          \draw[red]   ( 0.52, 1.93) node[above]{$\c_{n,0}$};
\shade[top color=orange, bottom color = yellow][opacity=0.5] (0,2) arc (90:-15:2);
\draw[blue][line width=1.2pt] (0,2) -- ( 2,0); 
\draw[pattern color=blue, pattern=north west lines][opacity=0.5] (0,2) -- ( 1.93,-0.52) to[out=90,in=-90] (2,0) -- (0,2);
\draw[black] [line width=1.2pt] (0,2) -- ( 1.93,-0.52);
\draw (0,-2.5) node{$\SURF''$};
\end{tikzpicture}
\caption{ }
\label{fig-lemm-silt is a prod of two tilt}
\end{figure}

Moreover, we have $\Q^B = \Q'\times \Q''$, where $\Q'$ is the quiver given by the $\gbullet$-grFFAS $\gtD|_{\SURF'}$ of $\SURF'$
and $\Q''$ is that of the $\gbullet$-grFFAS $\gtD|_{\SURF''}$ of $\SURF''$. On the graded marked ribbon surface $\SURF'$, if $\tgamma_{0,x}\in\gtD|_{\SURF'}$ for any $\tgamma_{0,x}$, where $1\le x<l$,
then we have $\ii_{p_x}(\tgamma_{0,x},\tarc_{x})\ge 1$ by (S4).
Note that $X(\gtD)$ is 2-term, so is $X(\gtD|_{\SURF'})$. Thus, by Lemma \ref{lemm:2-silt-ind}, $\ii_{p_x}(\tgamma_{0,x},\tarc_{x})=1$.
By Proposition \ref{prop-gamma(0,l)}, $\Q'$ is connected. Similarly, we have $\Q''$ is connected.
Therefore, $B \cong \kk\Q'\times \kk\Q''$, where $\kk\Q'$ and $\kk\Q''$ are tilted by Theorem \ref{thm. silt} and Proposition \ref{prop-gamma(0,l)}.
\end{proof}

To summarize, we have the following result.

\begin{corollary} The silted algebras of type $A_n$ forming two families:
\begin{itemize}
  \item[\rm(1)] tilted algebras of type $A_n$;
  \item[\rm(2)] tilted algebras of type $A_m \times A_{n-m}$, where $1\leq m\leq n-1$.
\end{itemize}
\end{corollary}

\begin{corollary}
There are no strictly shod algebras in $A_n$.
\end{corollary}

\begin{remark} \rm
\begin{itemize}
  \item [(1)] The classification of silted algebras of type $A_n$ has also been studied by Xie, Yang and the second author of this paper \cite{XieYangZhang22} by using algebraic method. 
  \item [(2)] In \cite{ZhangLiu22}, we show that there are no strictly shod algebras in Dynkin type $\mathbb{A}_n$ with arbitrary orientation by using a curve embedding for the geometric models of gentle algebras from module categories to derived categories.
\end{itemize}
\end{remark}

\section{The number of silted algebras of type $A_n$} \label{sect-number-silted}
In this section, based on the classification of the silted algebras of type $A_n$, we obtain a formula to compute the number of the silted algebras of type $A_n$.

\subsection{The number of tilted algebras of type $A_n$}
In this subsection, based on the classification of tilted algebras of type $A_n$ in subsection \ref{subsec-classification},
we give a formula to compute the number of tilted algebras of type $A_n$.
Denote by $a_{\ta}(A_n)$ the number of tilted algebras,  $a_{\hta}(A_n)$ the number of hereditary tilted algebras and $a_{\nhta}(A_n)$ the number of
non-hereditary tilted algebras of type $A_n$, respectively. Then we have

\begin{proposition} \label{prop-cls:nhta}
The number of the non-hereditary tilted algebras of type $A_n$ is:
\begin{align}\label{formula-prop-cls:nhta 1}
 a_{\nhta}(A_n)  = 2 a_{\nhta}(A_{n-1}) + \sum_{j=2}^{n-1} C_{j-1}C_{n-j},
\end{align}
where $C_r$ is the $r$-{\rm Catalan number} {\footnote{ The $r$-Catalan number $C_r$ is the number of triangulations of convex $(r+2)$-gon. Thus $C_{r+1} = \sharp\stautilt(\kk\A_r)$. }}
which equals to $\frac{1}{r+1}\tbinom{2r}{r}$.
Furthermore, we have
\[a_{\nhta}(A_n) = C_n - 2^{n-1}.\]
\end{proposition}

\begin{proof}
By Proposition \ref{prop-tri. contains sp. tile},
we only need to compute the number of triangulations lying in $\Tri_{\rmc}(\SURFextra(A_n))$. We do this by the following three cases in \Pic \ref{fig-prop-cls:nhta triangulations}. Note that the number of triangulations of cases (1) and (3) equal to $a_{\nhta}(A_{n-1})$. Now we consider case (2).
Let $\T_j$ be a triangle whose vertices are $r_1$, $r_2$ and $v_j$ for any $2\le j\le n-1$. Then we obtain that the number of tilted algebras in this case is $C_{j-1}C_{n-j}$. Thus, we obtain (\ref{formula-prop-cls:nhta 1}).

\begin{figure}[htbp]
\centering
\definecolor{ffqqqq}{rgb}{1,0,0}
\definecolor{qqwuqq}{rgb}{0,0.5,0}
\begin{tikzpicture}[scale=0.7]
\draw[line width=1.2pt] (2,0) arc (0:360:2);
\draw[line width=2pt][white] (1.414,-1.414) arc (-45:-135:2);
\draw[line width=1.2pt][dotted] (1.414,-1.414) arc (-45:-135:2);
\filldraw[qqwuqq] ( 0.00, 2.00) circle (0.1);
\draw[ffqqqq][line width=1.2pt] (-1.56, 1.25) circle (0.1); \draw[ffqqqq] (-1.56, 1.25) node[left]{$r_1$};
\filldraw[qqwuqq] (-1.95, 0.45) circle (0.1); \draw[qqwuqq] (-1.95, 0.45) node[left]{$v_1$};
\filldraw[qqwuqq] (-1.95,-0.45) circle (0.1); \draw[qqwuqq] (-1.95,-0.45) node[left]{$v_2$};
\filldraw[qqwuqq] (-1.61,-1.19) circle (0.1); \draw[qqwuqq] (-1.61,-1.19) node[left]{$v_3$};
\filldraw[qqwuqq] (-0.87,-1.80) circle (0.1);
\filldraw[qqwuqq] (-0.00,-2.00) circle (0.1);
\filldraw[qqwuqq] ( 0.87,-1.80) circle (0.1);
\filldraw[qqwuqq] ( 1.61,-1.19) circle (0.1); \draw[qqwuqq] ( 1.61,-1.19) node[right]{$v_{n-2}$};
\filldraw[qqwuqq] ( 1.94,-0.45) circle (0.1); \draw[qqwuqq] ( 1.94,-0.45) node[right]{$v_{n-1}$};
\filldraw[qqwuqq] ( 1.95, 0.45) circle (0.1); \draw[qqwuqq] ( 1.94, 0.45) node[right]{$v_n$};
\draw[ffqqqq][line width=1.2pt] ( 1.56, 1.25) circle (0.1); \draw[ffqqqq] ( 1.56, 1.25) node[right]{$r_2$};
\draw[qqwuqq!25][line width=1.2pt] ( 0.00, 2.00) -- (-1.95, 0.45);
\draw[qqwuqq!25][line width=1.2pt] ( 0.00, 2.00) -- (-1.95,-0.45);
\draw[qqwuqq!25][line width=1.2pt] ( 0.00, 2.00) -- (-1.61,-1.19);
\draw[qqwuqq!25][line width=1.2pt][dotted] ( 0.00, 2.00) -- (-0.87,-1.80);
\draw[qqwuqq!25][line width=1.2pt][dotted] ( 0.00, 2.00) -- (-0.00,-2.00);
\draw[qqwuqq!25][line width=1.2pt][dotted] ( 0.00, 2.00) -- ( 0.87,-1.80);
\draw[qqwuqq!25][line width=1.2pt] ( 0.00, 2.00) -- ( 1.61,-1.19);
\draw[qqwuqq!25][line width=1.2pt] ( 0.00, 2.00) -- ( 1.94,-0.45);
\draw[qqwuqq!25][line width=1.2pt] ( 0.00, 2.00) -- ( 1.95, 0.45);
\draw[qqwuqq] (0,2) node[above]{$p$};
\draw[orange][line width=1.2pt] (-1.56, 1.25)--(1.56, 1.25); \draw[orange] (0, 1.25) node[above]{$\tiny\pmb{c_{P(1)}}$};
\draw[blue][line width=1.2pt] ( 1.56, 1.25)--(-1.95, 0.45);
\filldraw[cyan][line width=1.2pt][opacity=0.25] (-1.95, 0.45) arc (167.0054:399.8533:2);
\draw (0,-0.5) node{\tiny $a_{\nhta}(A_{n-1})$};
\draw ( 0.00,-2.80) node{(1)};
\end{tikzpicture}
\ \ \ \
\begin{tikzpicture}[scale=0.7]
\draw[line width=1.2pt] (2,0) arc (0:360:2);
\draw[line width=2pt][white] (1.414,-1.414) arc (-45:-135:2);
\draw[line width=1.2pt][dotted] (1.414,-1.414) arc (-45:-135:2);
\filldraw[qqwuqq] ( 0.00, 2.00) circle (0.1);
\draw[ffqqqq][line width=1.2pt] (-1.56, 1.25) circle (0.1); \draw[ffqqqq] (-1.56, 1.25) node[left]{$r_1$};
\filldraw[qqwuqq] (-1.95, 0.45) circle (0.1); \draw[qqwuqq] (-1.95, 0.45) node[left]{$v_1$};
\filldraw[qqwuqq] (-1.95,-0.45) circle (0.1); \draw[qqwuqq] (-1.95,-0.45) node[left]{$v_2$};
\filldraw[qqwuqq] (-1.61,-1.19) circle (0.1); \draw[qqwuqq] (-1.61,-1.19) node[left]{$v_3$};
\filldraw[qqwuqq] (-0.87,-1.80) circle (0.1);
\filldraw[qqwuqq] (-0.00,-2.00) circle (0.1);
\filldraw[qqwuqq] ( 0.87,-1.80) circle (0.1);
\filldraw[qqwuqq] ( 1.61,-1.19) circle (0.1); \draw[qqwuqq] ( 1.61,-1.19) node[right]{$v_{n-2}$};
\filldraw[qqwuqq] ( 1.94,-0.45) circle (0.1); \draw[qqwuqq] ( 1.94,-0.45) node[right]{$v_{n-1}$};
\filldraw[qqwuqq] ( 1.95, 0.45) circle (0.1); \draw[qqwuqq] ( 1.94, 0.45) node[right]{$v_n$};
\draw[ffqqqq][line width=1.2pt] ( 1.56, 1.25) circle (0.1); \draw[ffqqqq] ( 1.56, 1.25) node[right]{$r_2$};
\draw[qqwuqq!25][line width=1.2pt] ( 0.00, 2.00) -- (-1.95, 0.45);
\draw[qqwuqq!25][line width=1.2pt] ( 0.00, 2.00) -- (-1.95,-0.45);
\draw[qqwuqq!25][line width=1.2pt] ( 0.00, 2.00) -- (-1.61,-1.19);
\draw[qqwuqq!25][line width=1.2pt][dotted] ( 0.00, 2.00) -- (-0.87,-1.80);
\draw[qqwuqq!25][line width=1.2pt][dotted] ( 0.00, 2.00) -- (-0.00,-2.00);
\draw[qqwuqq!25][line width=1.2pt][dotted] ( 0.00, 2.00) -- ( 0.87,-1.80);
\draw[qqwuqq!25][line width=1.2pt] ( 0.00, 2.00) -- ( 1.61,-1.19);
\draw[qqwuqq!25][line width=1.2pt] ( 0.00, 2.00) -- ( 1.94,-0.45);
\draw[qqwuqq!25][line width=1.2pt] ( 0.00, 2.00) -- ( 1.95, 0.45);
\draw[qqwuqq] (0,2) node[above]{$p$};
\draw[orange][line width=1.2pt] (-1.56, 1.25)--(1.56, 1.25); \draw[orange] (0, 1.25) node[above]{$\tiny\pmb{c_{P(1)}}$};
\draw[blue][line width=1.2pt] (-1.56, 1.25)--(1.56, 1.25)--(0,-2)--(-1.56, 1.25);
\draw[blue] (0,0) node{$\T_j$};
\filldraw[yellow][line width=1.2pt][opacity=0.25] (-1.56, 1.25) arc (141.295:270:2);
\draw (-1.3,-0.8) node{\tiny $C_{j-1}$};
\draw (1.2,-0.8) node{\tiny $C_{n-j}$};
\filldraw[cyan][line width=1.2pt][opacity=0.25] (1.56, 1.25) arc (38.7046:-90:2);
\draw[qqwuqq] (0,-2) node[below]{$v_j$};
\draw ( 0.00,-2.80) node{(2)};
\end{tikzpicture}
\ \ \ \
\begin{tikzpicture}[scale=0.7]
\draw[line width=1.2pt] (2,0) arc (0:360:2);
\draw[line width=2pt][white] (1.414,-1.414) arc (-45:-135:2);
\draw[line width=1.2pt][dotted] (1.414,-1.414) arc (-45:-135:2);
\filldraw[qqwuqq] ( 0.00, 2.00) circle (0.1);
\draw[ffqqqq][line width=1.2pt] (-1.56, 1.25) circle (0.1); \draw[ffqqqq] (-1.56, 1.25) node[left]{$r_1$};
\filldraw[qqwuqq] (-1.95, 0.45) circle (0.1); \draw[qqwuqq] (-1.95, 0.45) node[left]{$v_1$};
\filldraw[qqwuqq] (-1.95,-0.45) circle (0.1); \draw[qqwuqq] (-1.95,-0.45) node[left]{$v_2$};
\filldraw[qqwuqq] (-1.61,-1.19) circle (0.1); \draw[qqwuqq] (-1.61,-1.19) node[left]{$v_3$};
\filldraw[qqwuqq] (-0.87,-1.80) circle (0.1);
\filldraw[qqwuqq] (-0.00,-2.00) circle (0.1);
\filldraw[qqwuqq] ( 0.87,-1.80) circle (0.1);
\filldraw[qqwuqq] ( 1.61,-1.19) circle (0.1); \draw[qqwuqq] ( 1.61,-1.19) node[right]{$v_{n-2}$};
\filldraw[qqwuqq] ( 1.94,-0.45) circle (0.1); \draw[qqwuqq] ( 1.94,-0.45) node[right]{$v_{n-1}$};
\filldraw[qqwuqq] ( 1.95, 0.45) circle (0.1); \draw[qqwuqq] ( 1.94, 0.45) node[right]{$v_n$};
\draw[ffqqqq][line width=1.2pt] ( 1.56, 1.25) circle (0.1); \draw[ffqqqq] ( 1.56, 1.25) node[right]{$r_2$};
\draw[qqwuqq!25][line width=1.2pt] ( 0.00, 2.00) -- (-1.95, 0.45);
\draw[qqwuqq!25][line width=1.2pt] ( 0.00, 2.00) -- (-1.95,-0.45);
\draw[qqwuqq!25][line width=1.2pt] ( 0.00, 2.00) -- (-1.61,-1.19);
\draw[qqwuqq!25][line width=1.2pt][dotted] ( 0.00, 2.00) -- (-0.87,-1.80);
\draw[qqwuqq!25][line width=1.2pt][dotted] ( 0.00, 2.00) -- (-0.00,-2.00);
\draw[qqwuqq!25][line width=1.2pt][dotted] ( 0.00, 2.00) -- ( 0.87,-1.80);
\draw[qqwuqq!25][line width=1.2pt] ( 0.00, 2.00) -- ( 1.61,-1.19);
\draw[qqwuqq!25][line width=1.2pt] ( 0.00, 2.00) -- ( 1.94,-0.45);
\draw[qqwuqq!25][line width=1.2pt] ( 0.00, 2.00) -- ( 1.95, 0.45);
\draw[qqwuqq] (0,2) node[above]{$p$};
\draw[orange][line width=1.2pt] (1.56, 1.25)--(-1.56, 1.25); \draw[orange] (0, 1.25) node[above]{$\tiny\pmb{c_{P(1)}}$};
\draw[blue][line width=1.2pt] (-1.56, 1.25)--(1.95, 0.45);
\filldraw[yellow][line width=1.2pt][opacity=0.25] (1.95, 0.45) arc (12.995:-218.8533:2);
\draw (0,-0.5) node{\tiny $a_{\nhta}(A_{n-1})$};
\draw ( 0.00,-2.80) node{(3)};
\end{tikzpicture}
\caption{\textsf{In (2), the permissible curve $c_{P(1)}$ coincides with one edge of $\T$.}}
\label{fig-prop-cls:nhta triangulations}
\end{figure}
\noindent As we all know, $C_n = \sum\limits_{j=0}^{n-1} C_jC_{n-1-j}$, thus
\begin{center}
$\sum\limits_{j=2}^{n-1} C_{j-1}C_{n-j} = C_n - C_0C_{n-1}-C_{n-1}C_0=C_n-2C_{n-1}$.
\end{center}
Therefore,
\begin{align}
 a_{\nhta}(A_n)-C_n = 2 \big(a_{\nhta}(A_{n-1})-C_{n-1}\big). \nonumber
\end{align}
This shows that $\{a_{\nhta}(A_n)-C_n\}_{n\in\NN^+}$ is a geometric sequence and so
\begin{center}
$a_{\nhta}(A_n)-C_n = -2^{n-1}$.
\end{center}
Thus the formula $a_{\nhta}(A_n) = C_n - 2^{n-1}$ holds.
\end{proof}

\begin{remark} \rm
In \cite{ObaidNaumanFakiehRingel15},  Obaid, Nauman, Fakieh and Ringel showed that $C_n$ equals to the number of the tilting modules over $A_n$. In particular, the number of tilting modules with hereditary endomorphism algebras is $2^{n-1}$. By Proposition \ref{prop-tri. contains sp. tile}, we can also obtain that $a_{\nhta}(A_n) = C_n - 2^{n-1}$, see \cite{XieYangZhang22}.
\end{remark}

\begin{proposition} \label{prop-nht}
The number of the hereditary tilted algebras of type $A_n$ is:
\[a_{\hta}(A_n) = 2^{n-2} + (1+(-1)^{n-1}) 2^{\frac{n-5}{2}}.\]
\end{proposition}

\begin{proof}
By Proposition \ref{prop-non-hereditary}, any triangulation $\Gamma$ of $\SURFextra(A_n)$ which not contains a complete triangle can induce a hereditary tilted algebra $A^{\Gamma}$ of type $A_n$. Thus, we have the following three cases (up to equivalences):
\begin{align}
   & \text{Case 1}\ \xymatrix {\bullet \ar[r] & \bullet \ar@{-}[r] & \cdots \ar@{-}[r] &  \bullet \ar[r] & \bullet}, \nonumber \\
   & \text{Case 2}\ \xymatrix {\bullet \ar@{<-}[r] & \bullet \ar@{-}[r] & \cdots \ar@{-}[r] &  \bullet \ar[r] & \bullet},  \nonumber \\
   & \text{Case 3}\ \xymatrix {\bullet \ar[r] & \bullet \ar@{-}[r] & \cdots \ar@{-}[r] &  \bullet \ar@{<-}[r] & \bullet}.  \nonumber
\end{align}

\noindent The number of hereditary tilted algebra lying in Case 1 is $2^{n-3}$, and the number of hereditary tilted algebra lying in Case 2 (resp., Case 3) is $a_{\hta}(A_{n-2})$.
Thus,
\begin{align}
a_{\hta}(A_n) = 2^{n-3}+2a_{\hta}(A_{n-2}). \nonumber
\end{align}
Furthermore,
\begin{align}
a_{\hta}(A_n)-2^{n-2} = 2 \big(a_{\hta}(A_{n-2}) - 2^{n-4}\big). \nonumber
\end{align}
If $n$ is even, then $a_{\hta}(A_n) = 2^{n-2}$. If $n$ is odd, then $a_{\hta}(A_n) = 2^{n-2}+2^{(n-1)/2-1}$.
Thus we obtain that $a_{\hta}(A_n)= 2^{n-2} + (1+(-1)^{n-1}) 2^{\frac{n-5}{2}}$.
\end{proof}

\begin{theorem} \label{thm-tilt alg}
The number of tilted algebras of type $A_n$ is:
\[a_{\ta}(A_n) = C_n  + (1+(-1)^{n-1}) 2^{\frac{n-5}{2}} - 2^{n-2}. \]
\end{theorem}

\subsection{The number of the silted algebras of type $A_n$}
In this subsection, we give the formula to compute the number of the silted algebras of type $A_n$. Denote by $a_{\ncsa}(A_n)$ the number of all the non-connected silted algebras of type $A_n$. Then we have the following result.

\begin{theorem}\label{thm-slit alg}
The number of silted algebras of type $A_n$ is:
\begin{align}\label{formula-slit alg}
a_{\sa}(A_n) & = a_{\ta}(A_n) + a_{\ncsa}(A_n)
\end{align}
where
\begin{align}
a_{\ncsa}(A_n)=\frac{1}{2}\sum_{n'+n''=n \atop 1\le n'\ne n''\le n-1} a_{\ta}(A_{n'})a_{\ta}(A_{n''})+ \frac{1+(-1)^{n}}{2} \cdot \frac{a_{\ta}(A_{\lfloor \frac{n}{2}\rfloor})(a_{\ta}(A_{\lfloor \frac{n}{2}\rfloor})+1)}{2}. \nonumber
\end{align}
\end{theorem}

\begin{proof}
By Theorem \ref{coro-gamma(0,l)}, we have $a_{\sa}(A_n) = a_{\ta}(A_n) + a_{\ncsa}(A_n)$. Moreover, we have
\begin{itemize}
  \item if $n$ is even, then
    \begin{align}
      a_{\ncsa}(A_n) & = \sum\limits_{1\le l\le \frac{n}{2}-1} a_{\ta}(A_l)a_{\ta}(A_{n-l})
                     + \tbinom{a_{\ta}(A_{n/2})}{2} + a_{\ta}(A_{n/2}) \nonumber \\
                   & = \sum\limits_{1\le l\le \frac{n}{2}-1} a_{\ta}(A_l)a_{\ta}(A_{n-l})
                     + \frac{a_{\ta}(A_{n/2})(a_{\ta}(A_{n/2})+1)}{2}; \nonumber
    \end{align}
  \item if $n$ is odd, then
    \begin{align}
      a_{\ncsa}(A_n) & = \frac{1}{2}\sum\limits_{1\le l\le n-1} a_{\ta}(A_l)a_{\ta}(A_{n-l}).
                       \nonumber
    \end{align}
\end{itemize}
Thus, we have (\ref{formula-slit alg}).
\end{proof}

\begin{remark} \rm
We do not need the number of $2$-term silting complex over $A_n$ to compute the number of silted algebras of type $A_n$.
\end{remark}

\subsection{An example}

\begin{example} \label{exm-A4} \rm
Let $A_4$ be the algebra given by the following quiver
$$\begin{xy}
(-10,0)*+{1}="1",
(0,0)*+{2}="2",
(10,0)*+{3}="3",
(20,0)*+{4}="4",
\ar"1";"2", \ar"2";"3", \ar"3";"4",
\end{xy}.$$
Let $\Gamma\in \Tri(\SURFextra(A_4))$ and $c_{P(1)}\in\Gamma$. By Proposition \ref{prop-non-hereditary}, if $\Gamma$ not contains a complete triangle, then the induced tilted algebras are hereditary. Thus, we have the following four hereditary tilted algebras of type $A_{4}$ up to equivalences.
\begin{figure}[H]
\begin{center}
\definecolor{ffqqqq}{rgb}{1,0,0}
\definecolor{qqwuqq}{rgb}{0,0.5,0}
\begin{tikzpicture}[scale=0.6]
\draw[line width=1.2pt]  (2,0) arc (0:360:2);
\filldraw[qqwuqq] ( 0.00, 2.00) circle (0.1);
\draw[ffqqqq][line width=1.2pt] (-1.56, 1.25) circle (0.1);
\filldraw[qqwuqq] (-1.95,-0.45) circle (0.1);
\filldraw[qqwuqq] (-0.87,-1.80) circle (0.1);
\filldraw[qqwuqq] ( 0.87,-1.80) circle (0.1);
\filldraw[qqwuqq] ( 1.94,-0.45) circle (0.1);
\draw[ffqqqq][line width=1.2pt] ( 1.56, 1.25) circle (0.1);
\draw[qqwuqq!25][line width=1.2pt] ( 0.00, 2.00) -- (-1.95,-0.45);
\draw[qqwuqq!25][line width=1.2pt] ( 0.00, 2.00) -- (-0.87,-1.80);
\draw[qqwuqq!25][line width=1.2pt] ( 0.00, 2.00) -- ( 0.87,-1.80);
\draw[qqwuqq!25][line width=1.2pt] ( 0.00, 2.00) -- ( 1.94,-0.45);
\draw[blue][line width=1.2pt] (-1.56, 1.25)--( 1.56, 1.25); \draw[blue] (0,1.25) node[below]{\tiny$c_{P(1)}$};
\draw[blue][line width=1.2pt] ( 1.56, 1.25)--(-0.87,-1.80);
\draw[blue][line width=1.2pt] ( 1.56, 1.25)--( 0.87,-1.80);
\draw[blue][line width=1.2pt] ( 1.56, 1.25)--(-1.95,-0.45);
\draw (0,-2.5) node{$\Gamma_1$};
\end{tikzpicture}
\ \ \ \
\begin{tikzpicture}[scale=0.6]
\draw[line width=1.2pt]  (2,0) arc (0:360:2);
\filldraw[qqwuqq] ( 0.00, 2.00) circle (0.1);
\draw[ffqqqq][line width=1.2pt] (-1.56, 1.25) circle (0.1);
\filldraw[qqwuqq] (-1.95,-0.45) circle (0.1);
\filldraw[qqwuqq] (-0.87,-1.80) circle (0.1);
\filldraw[qqwuqq] ( 0.87,-1.80) circle (0.1);
\filldraw[qqwuqq] ( 1.94,-0.45) circle (0.1);
\draw[ffqqqq][line width=1.2pt] ( 1.56, 1.25) circle (0.1);
\draw[qqwuqq!25][line width=1.2pt] ( 0.00, 2.00) -- (-1.95,-0.45);
\draw[qqwuqq!25][line width=1.2pt] ( 0.00, 2.00) -- (-0.87,-1.80);
\draw[qqwuqq!25][line width=1.2pt] ( 0.00, 2.00) -- ( 0.87,-1.80);
\draw[qqwuqq!25][line width=1.2pt] ( 0.00, 2.00) -- ( 1.94,-0.45);
\draw[blue][line width=1.2pt] (-1.56, 1.25)--( 1.56, 1.25); \draw[blue] (0,1.25) node[below]{\tiny$c_{P(1)}$};
\draw[blue][line width=1.2pt] ( 1.56, 1.25)--(-0.87,-1.80);
\draw[blue][line width=1.2pt] ( 1.56, 1.25)--(-1.95,-0.45);
\draw[blue][line width=1.2pt] (-0.87,-1.80)--( 1.94,-0.45);
\draw (0,-2.5) node{$\Gamma_2$};
\end{tikzpicture}
\ \ \ \
\begin{tikzpicture}[scale=0.6]
\draw[line width=1.2pt]  (2,0) arc (0:360:2);
\filldraw[qqwuqq] ( 0.00, 2.00) circle (0.1);
\draw[ffqqqq][line width=1.2pt] (-1.56, 1.25) circle (0.1);
\filldraw[qqwuqq] (-1.95,-0.45) circle (0.1);
\filldraw[qqwuqq] (-0.87,-1.80) circle (0.1);
\filldraw[qqwuqq] ( 0.87,-1.80) circle (0.1);
\filldraw[qqwuqq] ( 1.94,-0.45) circle (0.1);
\draw[ffqqqq][line width=1.2pt] ( 1.56, 1.25) circle (0.1);
\draw[qqwuqq!25][line width=1.2pt] ( 0.00, 2.00) -- (-1.95,-0.45);
\draw[qqwuqq!25][line width=1.2pt] ( 0.00, 2.00) -- (-0.87,-1.80);
\draw[qqwuqq!25][line width=1.2pt] ( 0.00, 2.00) -- ( 0.87,-1.80);
\draw[qqwuqq!25][line width=1.2pt] ( 0.00, 2.00) -- ( 1.94,-0.45);
\draw[blue][line width=1.2pt] (-1.56, 1.25)--( 1.56, 1.25); \draw[blue] (0,1.25) node[below]{\tiny$c_{P(1)}$};
\draw[blue][line width=1.2pt] (-1.56, 1.25)--( 1.94,-0.45);
\draw[blue][line width=1.2pt] (-1.95,-0.45)--( 1.94,-0.45);
\draw[blue][line width=1.2pt] (-1.95,-0.45)--( 0.87,-1.80);
\draw (0,-2.5) node{$\Gamma_3$};
\end{tikzpicture}
\ \ \ \
\begin{tikzpicture}[scale=0.6]
\draw[line width=1.2pt]  (2,0) arc (0:360:2);
\filldraw[qqwuqq] ( 0.00, 2.00) circle (0.1);
\draw[ffqqqq][line width=1.2pt] (-1.56, 1.25) circle (0.1);
\filldraw[qqwuqq] (-1.95,-0.45) circle (0.1);
\filldraw[qqwuqq] (-0.87,-1.80) circle (0.1);
\filldraw[qqwuqq] ( 0.87,-1.80) circle (0.1);
\filldraw[qqwuqq] ( 1.94,-0.45) circle (0.1);
\draw[ffqqqq][line width=1.2pt] ( 1.56, 1.25) circle (0.1);
\draw[qqwuqq!25][line width=1.2pt] ( 0.00, 2.00) -- (-1.95,-0.45);
\draw[qqwuqq!25][line width=1.2pt] ( 0.00, 2.00) -- (-0.87,-1.80);
\draw[qqwuqq!25][line width=1.2pt] ( 0.00, 2.00) -- ( 0.87,-1.80);
\draw[qqwuqq!25][line width=1.2pt] ( 0.00, 2.00) -- ( 1.94,-0.45);
\draw[blue][line width=1.2pt] (-1.56, 1.25)--( 1.56, 1.25); \draw[blue] (0,1.25) node[below]{\tiny$c_{P(1)}$};
\draw[blue][line width=1.2pt] (-1.56, 1.25)--( 1.94,-0.45);
\draw[blue][line width=1.2pt] (-1.95,-0.45)--( 1.94,-0.45);
\draw[blue][line width=1.2pt] (-0.87,-1.80)--( 1.94,-0.45);
\draw (0,-2.5) node{$\Gamma_4$};
\end{tikzpicture}
\caption{\textsf{ }}
\label{fig-A4:tilt diss}
\end{center}
\end{figure}
\noindent In the non-hereditary case, $\Gamma$ contains a complete triangle $\T$. Thus, there are two cases as following:
\begin{itemize}
  \item[Case] 1. If $c_{P(1)}$ is an edge of $\T$,
    then $\Gamma$ is one of the following four triangulations $\Gamma_5$ -- $\Gamma_8$ shown in \Pic \ref{fig-A4:tilt diss};

\begin{figure}[H]
\begin{center}
\definecolor{ffqqqq}{rgb}{1,0,0}
\definecolor{qqwuqq}{rgb}{0,0.5,0}
\begin{tikzpicture}[scale=0.6]
\draw[line width=1.2pt]  (2,0) arc (0:360:2);
\filldraw[qqwuqq] ( 0.00, 2.00) circle (0.1);
\draw[ffqqqq][line width=1.2pt] (-1.56, 1.25) circle (0.1);
\filldraw[qqwuqq] (-1.95,-0.45) circle (0.1);
\filldraw[qqwuqq] (-0.87,-1.80) circle (0.1);
\filldraw[qqwuqq] ( 0.87,-1.80) circle (0.1);
\filldraw[qqwuqq] ( 1.94,-0.45) circle (0.1);
\draw[ffqqqq][line width=1.2pt] ( 1.56, 1.25) circle (0.1);
\draw[qqwuqq!25][line width=1.2pt] ( 0.00, 2.00) -- (-1.95,-0.45);
\draw[qqwuqq!25][line width=1.2pt] ( 0.00, 2.00) -- (-0.87,-1.80);
\draw[qqwuqq!25][line width=1.2pt] ( 0.00, 2.00) -- ( 0.87,-1.80);
\draw[qqwuqq!25][line width=1.2pt] ( 0.00, 2.00) -- ( 1.94,-0.45);
\draw[blue][line width=1.2pt] (-1.56, 1.25)--( 1.56, 1.25); \draw[blue] (0,1.25) node[below]{\tiny$c_{P(1)}$};
\draw[blue][line width=1.2pt] (-1.56, 1.25)--(-0.87,-1.80)--( 1.56, 1.25);
\draw[blue][line width=1.2pt] ( 1.56, 1.25)--( 0.87,-1.80);
\draw[blue] (0,0) node{$\T$};
\draw (0,-2.5) node{$\Gamma_5$};
\end{tikzpicture}
\ \ \ \
\begin{tikzpicture}[scale=0.6]
\draw[line width=1.2pt]  (2,0) arc (0:360:2);
\filldraw[qqwuqq] ( 0.00, 2.00) circle (0.1);
\draw[ffqqqq][line width=1.2pt] (-1.56, 1.25) circle (0.1);
\filldraw[qqwuqq] (-1.95,-0.45) circle (0.1);
\filldraw[qqwuqq] (-0.87,-1.80) circle (0.1);
\filldraw[qqwuqq] ( 0.87,-1.80) circle (0.1);
\filldraw[qqwuqq] ( 1.94,-0.45) circle (0.1);
\draw[ffqqqq][line width=1.2pt] ( 1.56, 1.25) circle (0.1);
\draw[qqwuqq!25][line width=1.2pt] ( 0.00, 2.00) -- (-1.95,-0.45);
\draw[qqwuqq!25][line width=1.2pt] ( 0.00, 2.00) -- (-0.87,-1.80);
\draw[qqwuqq!25][line width=1.2pt] ( 0.00, 2.00) -- ( 0.87,-1.80);
\draw[qqwuqq!25][line width=1.2pt] ( 0.00, 2.00) -- ( 1.94,-0.45);
\draw[blue][line width=1.2pt] (-1.56, 1.25)--( 1.56, 1.25); \draw[blue] (0,1.25) node[below]{\tiny$c_{P(1)}$};
\draw[blue][line width=1.2pt] (-1.56, 1.25)--( 0.87,-1.80)--( 1.56, 1.25);
\draw[blue][line width=1.2pt] (-1.56, 1.25)--(-0.87,-1.80);
\draw[blue] (0,0) node{$\T$};
\draw (0,-2.5) node{$\Gamma_6$};
\end{tikzpicture}
\ \ \ \
\begin{tikzpicture}[scale=0.6]
\draw[line width=1.2pt]  (2,0) arc (0:360:2);
\filldraw[qqwuqq] ( 0.00, 2.00) circle (0.1);
\draw[ffqqqq][line width=1.2pt] (-1.56, 1.25) circle (0.1);
\filldraw[qqwuqq] (-1.95,-0.45) circle (0.1);
\filldraw[qqwuqq] (-0.87,-1.80) circle (0.1);
\filldraw[qqwuqq] ( 0.87,-1.80) circle (0.1);
\filldraw[qqwuqq] ( 1.94,-0.45) circle (0.1);
\draw[ffqqqq][line width=1.2pt] ( 1.56, 1.25) circle (0.1);
\draw[qqwuqq!25][line width=1.2pt] ( 0.00, 2.00) -- (-1.95,-0.45);
\draw[qqwuqq!25][line width=1.2pt] ( 0.00, 2.00) -- (-0.87,-1.80);
\draw[qqwuqq!25][line width=1.2pt] ( 0.00, 2.00) -- ( 0.87,-1.80);
\draw[qqwuqq!25][line width=1.2pt] ( 0.00, 2.00) -- ( 1.94,-0.45);
\draw[blue][line width=1.2pt] (-1.56, 1.25)--( 1.56, 1.25); \draw[blue] (0,1.25) node[below]{\tiny$c_{P(1)}$};
\draw[blue][line width=1.2pt] (-1.56, 1.25)--(-0.87,-1.80)--( 1.56, 1.25);
\draw[blue][line width=1.2pt] (-0.87,-1.80)--( 1.94,-0.45);
\draw[blue] (0,0) node{$\T$};
\draw (0,-2.5) node{$\Gamma_7$};
\end{tikzpicture}
\ \ \ \
\begin{tikzpicture}[scale=0.6]
\draw[line width=1.2pt]  (2,0) arc (0:360:2);
\filldraw[qqwuqq] ( 0.00, 2.00) circle (0.1);
\draw[ffqqqq][line width=1.2pt] (-1.56, 1.25) circle (0.1);
\filldraw[qqwuqq] (-1.95,-0.45) circle (0.1);
\filldraw[qqwuqq] (-0.87,-1.80) circle (0.1);
\filldraw[qqwuqq] ( 0.87,-1.80) circle (0.1);
\filldraw[qqwuqq] ( 1.94,-0.45) circle (0.1);
\draw[ffqqqq][line width=1.2pt] ( 1.56, 1.25) circle (0.1);
\draw[qqwuqq!25][line width=1.2pt] ( 0.00, 2.00) -- (-1.95,-0.45);
\draw[qqwuqq!25][line width=1.2pt] ( 0.00, 2.00) -- (-0.87,-1.80);
\draw[qqwuqq!25][line width=1.2pt] ( 0.00, 2.00) -- ( 0.87,-1.80);
\draw[qqwuqq!25][line width=1.2pt] ( 0.00, 2.00) -- ( 1.94,-0.45);
\draw[blue][line width=1.2pt] (-1.56, 1.25)--( 1.56, 1.25); \draw[blue] (0,1.25) node[below]{\tiny$c_{P(1)}$};
\draw[blue][line width=1.2pt] (-1.56, 1.25)--( 0.87,-1.80)--( 1.56, 1.25);
\draw[blue][line width=1.2pt] (-1.94,-0.45)--( 0.87,-1.80);
\draw[blue] (0,0) node{$\T$};
\draw (0,-2.5) node{$\Gamma_8$};
\end{tikzpicture}
\caption{\textsf{ }}
\label{fig-A4:tilt diss}
\end{center}
\end{figure}

  \item[Case] 2. If $c_{P(1)}$ is not an edge of $\T$, then $\Gamma$ is one of 
  the following two triangulations $\Gamma_9$ and $\Gamma_{10}$ shown in \Pic \ref{fig-A4:tilt diss 2}.

\begin{figure}[H]
\begin{center}
\definecolor{ffqqqq}{rgb}{1,0,0}
\definecolor{qqwuqq}{rgb}{0,0.5,0}
\begin{tikzpicture}[scale=0.6]
\draw[line width=1.2pt]  (2,0) arc (0:360:2);
\filldraw[qqwuqq] ( 0.00, 2.00) circle (0.1);
\draw[ffqqqq][line width=1.2pt] (-1.56, 1.25) circle (0.1);
\filldraw[qqwuqq] (-1.95,-0.45) circle (0.1);
\filldraw[qqwuqq] (-0.87,-1.80) circle (0.1);
\filldraw[qqwuqq] ( 0.87,-1.80) circle (0.1);
\filldraw[qqwuqq] ( 1.94,-0.45) circle (0.1);
\draw[ffqqqq][line width=1.2pt] ( 1.56, 1.25) circle (0.1);
\draw[qqwuqq!25][line width=1.2pt] ( 0.00, 2.00) -- (-1.95,-0.45);
\draw[qqwuqq!25][line width=1.2pt] ( 0.00, 2.00) -- (-0.87,-1.80);
\draw[qqwuqq!25][line width=1.2pt] ( 0.00, 2.00) -- ( 0.87,-1.80);
\draw[qqwuqq!25][line width=1.2pt] ( 0.00, 2.00) -- ( 1.94,-0.45);
\draw[blue][line width=1.2pt] (-1.56, 1.25)--( 1.56, 1.25); \draw[blue] (0,1.25) node[above]{\tiny$c_{P(1)}$};
\draw[blue][line width=1.2pt] ( 1.56, 1.25)--(-1.95,-0.45)--( 0.87,-1.80)--( 1.56, 1.25);
\draw[blue] (0,0) node{$\T$};
\draw (0,-2.5) node{$\Gamma_9$};
\end{tikzpicture}
\ \ \ \
\begin{tikzpicture}[scale=0.6]
\draw[line width=1.2pt]  (2,0) arc (0:360:2);
\filldraw[qqwuqq] ( 0.00, 2.00) circle (0.1);
\draw[ffqqqq][line width=1.2pt] (-1.56, 1.25) circle (0.1);
\filldraw[qqwuqq] (-1.95,-0.45) circle (0.1);
\filldraw[qqwuqq] (-0.87,-1.80) circle (0.1);
\filldraw[qqwuqq] ( 0.87,-1.80) circle (0.1);
\filldraw[qqwuqq] ( 1.94,-0.45) circle (0.1);
\draw[ffqqqq][line width=1.2pt] ( 1.56, 1.25) circle (0.1);
\draw[qqwuqq!25][line width=1.2pt] ( 0.00, 2.00) -- (-1.95,-0.45);
\draw[qqwuqq!25][line width=1.2pt] ( 0.00, 2.00) -- (-0.87,-1.80);
\draw[qqwuqq!25][line width=1.2pt] ( 0.00, 2.00) -- ( 0.87,-1.80);
\draw[qqwuqq!25][line width=1.2pt] ( 0.00, 2.00) -- ( 1.94,-0.45);
\draw[blue][line width=1.2pt] (-1.56, 1.25)--( 1.56, 1.25); \draw[blue] (0,1.25) node[above]{\tiny$c_{P(1)}$};
\draw[blue][line width=1.2pt] (-1.56, 1.25)--( 1.95,-0.45)--(-0.87,-1.80)--(-1.56, 1.25);
\draw[blue] (0,0) node{$\T$};
\draw (0,-2.5) node{$\Gamma_{10}$};
\end{tikzpicture}
\caption{\textsf{ }}
\label{fig-A4:tilt diss 2}
\end{center}
\end{figure}

\end{itemize}
\noindent Thus, we have ten tilted algebras of type $A_{4}$ which corresponding to $\Gamma_1$ -- $\Gamma_{10}$, respectively.

Now we calculate the number of silted algebras of type $A_{4}$ which induced by 2-term silting complex but not tilting over $A_4$. 
To do this, we need to compute all $\gbullet$-grFFASs of $\SURFred{\F_{A_4}}(A_4)$ such that there is an arrow provided by $\gbullet$-grFFAS whose grading $\ne 0$. 

Recall that $\tgamma_{x,y}$ is the $\gbullet$-curve with endpoints $\m_x$ and $\m_y$ by Subsection \ref{sec:silt-cls}, where $0\le x\ne y\le 4$. Taking
\[\teta_{x,y}^{[z]}=\begin{cases}
\tgamma_{x,y} \text{ with intersection index } z, & \text{ if } x=0, 1\le y\le 4; \\
\tgamma_{x,y} \text{ crossing two graded $\rbullet$-arcs}, & \text{ if } x\ne 0, 1\le x\ne y\le 4, z=\varnothing. \\
\end{cases}\]
Then by Proposition \ref{prop-gamma(0,l)}, we can obtain 14 $\gbullet$-grFFASs as following: 
\begin{align}
   \Delta_1=\{\teta_{0,1}^{[1]}, \teta_{0,2}^{[1]}, \teta_{0,3}^{[1]}, \teta_{0,4}^{[0]}\} 
&& \Delta_2=\{\teta_{0,1}^{[1]}, \teta_{0,2}^{[0]}, \teta_{0,3}^{[0]}, \teta_{0,4}^{[0]}\} 
&& \Delta_3=\{\teta_{0,1}^{[1]}, \teta_{0,2}^{[1]}, \teta_{0,3}^{[0]}, \teta_{0,4}^{[0]}\}  \nonumber \\
   \Delta_4=\{\teta_{0,1}^{[1]}, \teta_{0,2}^{[0]}, \teta_{0,3}^{[0]}, \teta_{3,4}\}
&& \Delta_5=\{\teta_{0,1}^{[1]}, \teta_{0,2}^{[1]}, \teta_{0,3}^{[0]}, \teta_{3,4}\} 
&& \Delta_6=\{\teta_{1,2}, \teta_{0,2}^{[1]}, \teta_{0,3}^{[1]}, \teta_{0,4}^{[0]}\} \nonumber \\ 
   \Delta_7=\{\teta_{1,2}, \teta_{0,2}^{[1]}, \teta_{0,3}^{[0]}, \teta_{0,4}^{[0]}\}
&& \Delta_8=\{\teta_{0,1}^{[1]}, \teta_{0,2}^{[0]}, \teta_{0,4}^{[0]}, \teta_{2,3}\} 
&& \Delta_9=\{\teta_{0,1}^{[1]}, \teta_{0,3}^{[1]}, \teta_{0,4}^{[0]}, \teta_{2,3}\} \nonumber \\
   \Delta_{10}=\{\teta_{0,1}^{[1]}, \teta_{0,2}^{[0]}, \teta_{2,3}, \teta_{2,4}\}
&& \Delta_{11}=\{\teta_{0,1}^{[1]}, \teta_{0,2}^{[0]}, \teta_{2,3}, \teta_{3,4}\}
&& \Delta_{12}=\{\teta_{0,3}^{[1]}, \teta_{0,4}^{[0]}, \teta_{1,3}, \teta_{2,3}\} \nonumber \\
   \Delta_{13}=\{\teta_{0,3}^{[1]}, \teta_{0,4}^{[0]}, \teta_{1,2}, \teta_{1,3}\}
&& \Delta_{14}=\{\teta_{0,2}^{[1]}, \teta_{0,3}^{[0]}, \teta_{1,2}, \teta_{3,4}\} 
&& \nonumber 
\end{align}
Note that 14 $\gbullet$-grFFASs can induce 5 non-connected silted algebras of type $A_{4}$:
\begin{itemize}
  \item[(1)] $\xymatrix@C=0.35cm{\bullet & \bullet\ar[r] & \bullet\ar[r] & \bullet}$: 
    $\Delta_1$, $\Delta_2$,  $\Delta_{10}$, $\Delta_{12}$;
    
  \item[(2)] $\xymatrix@C=0.35cm{\bullet\ar[r] & \bullet & \bullet\ar[r] & \bullet}$: 
    $\Delta_3$, $\Delta_5$, $\Delta_7$, $\Delta_{14}$;
    
  \item[(3)] $\xymatrix@C=0.35cm{\bullet & \bullet\ar[r] & \bullet & \bullet\ar[l]}$: 
    $\Delta_4$, $\Delta_{13}$; 
    
  \item[(4)] $\xymatrix@C=0.35cm{\bullet & \bullet & \bullet \ar[l] \ar[r] & \bullet}$:
    $\Delta_{11}$, $\Delta_6$; 
    
  \item[(5)] $\xymatrix@C=0.35cm{\bullet & \bullet\ar[r] \ar@/^0.6pc/@{.}[rr] & \bullet \ar[r] & \bullet}$:
    $\Delta_8$, $\Delta_9$. 
\end{itemize}
All the $\gbullet$-grFFAS in (1) are shown as following:
\begin{figure}[H]
\centering
\tiny
\definecolor{qqwuqq}{rgb}{0,0.5,0}
\begin{tikzpicture}[scale=0.6]
\draw[line width=1.2pt] (2,0) arc (0:360:2);
\draw[red][line width=1pt] (-1.56, 1.25) -- (-1.56,-1.25) -- (0,-2) -- ( 1.56,-1.25) -- ( 1.56, 1.25);
\fill[white] ( 1.56, 1.25) circle (0.1); \draw[red][line width=1pt] ( 1.56, 1.25) circle (0.1);
\fill[white] (-1.56, 1.25) circle (0.1); \draw[red][line width=1pt] (-1.56, 1.25) circle (0.1);
\fill[white] (-1.56,-1.25) circle (0.1); \draw[red][line width=1pt] (-1.56,-1.25) circle (0.1);
\fill[white] ( 0.  ,-2.  ) circle (0.1); \draw[red][line width=1pt] ( 0.  ,-2.  ) circle (0.1);
\fill[white] ( 1.56,-1.25) circle (0.1); \draw[red][line width=1pt] ( 1.56,-1.25) circle (0.1);
\filldraw[qqwuqq] ( 0.00, 2.00) circle (0.1) node[above]{\tiny$p=m_0$};
\filldraw[qqwuqq] (-1.95,-0.45) circle (0.1) node[ left]{\tiny$\m_1$};
\filldraw[qqwuqq] (-0.87,-1.80) circle (0.1) node[below]{\tiny$\m_2$};
\filldraw[qqwuqq] ( 0.87,-1.80) circle (0.1) node[below]{\tiny$\m_3$};
\filldraw[qqwuqq] ( 1.94,-0.45) circle (0.1) node[right]{\tiny$\m_4$};
\draw[qqwuqq][opacity=0.25][line width=1.2pt] ( 0.00, 2.00) -- (-1.95,-0.45);
\draw[qqwuqq][opacity=0.25][line width=1.2pt] ( 0.00, 2.00) -- (-0.87,-1.80);
\draw[qqwuqq][opacity=0.25][line width=1.2pt] ( 0.00, 2.00) -- ( 0.87,-1.80);
\draw[qqwuqq][opacity=0.25][line width=1.2pt] ( 0.00, 2.00) -- ( 1.94,-0.45);
\draw[blue][line width=1.2pt] ( 0.00, 2.00) -- (-1.95,-0.45); \draw[cyan] (-1.56,-0.  ) node[right]{$1$};
\draw[blue][line width=1.2pt] ( 0.00, 2.00) -- (-0.87,-1.80); \draw[cyan] (-0.81,-1.61) node[above left]{$1$};
\draw[blue][line width=1.2pt] ( 0.00, 2.00) -- ( 0.87,-1.80); \draw[cyan] ( 0.81,-1.61) node[above right]{$1$};
\draw[blue][line width=1.2pt] ( 0.00, 2.00) -- ( 1.94,-0.45); \draw[cyan] ( 1.56, 0.  ) node[left]{$0$};
\draw[->] (-0.8 , 0.8 ) -- (-0.5 , 0.5 );
\draw[->] (-0.25, 0.25) -- ( 0.25, 0.25);
\draw[->] ( 0.5 , 0.5 ) -- ( 0.8 , 0.8 ) [dash pattern=on 2pt off 2pt];
\draw (0,-2.8) node{$\Delta_1$};
\end{tikzpicture}
\begin{tikzpicture}[scale=0.6]
\draw[line width=1.2pt] (2,0) arc (0:360:2);
\draw[red][line width=1pt] (-1.56, 1.25) -- (-1.56,-1.25) -- (0,-2) -- ( 1.56,-1.25) -- ( 1.56, 1.25);
\fill[white] ( 1.56, 1.25) circle (0.1); \draw[red][line width=1pt] ( 1.56, 1.25) circle (0.1);
\fill[white] (-1.56, 1.25) circle (0.1); \draw[red][line width=1pt] (-1.56, 1.25) circle (0.1);
\fill[white] (-1.56,-1.25) circle (0.1); \draw[red][line width=1pt] (-1.56,-1.25) circle (0.1);
\fill[white] ( 0.  ,-2.  ) circle (0.1); \draw[red][line width=1pt] ( 0.  ,-2.  ) circle (0.1);
\fill[white] ( 1.56,-1.25) circle (0.1); \draw[red][line width=1pt] ( 1.56,-1.25) circle (0.1);
\filldraw[qqwuqq] ( 0.00, 2.00) circle (0.1) node[above]{\tiny$p=m_0$};
\filldraw[qqwuqq] (-1.95,-0.45) circle (0.1) node[ left]{\tiny$\m_1$};
\filldraw[qqwuqq] (-0.87,-1.80) circle (0.1) node[below]{\tiny$\m_2$};
\filldraw[qqwuqq] ( 0.87,-1.80) circle (0.1) node[below]{\tiny$\m_3$};
\filldraw[qqwuqq] ( 1.94,-0.45) circle (0.1) node[right]{\tiny$\m_4$};
\draw[qqwuqq][opacity=0.25][line width=1.2pt] ( 0.00, 2.00) -- (-1.95,-0.45);
\draw[qqwuqq][opacity=0.25][line width=1.2pt] ( 0.00, 2.00) -- (-0.87,-1.80);
\draw[qqwuqq][opacity=0.25][line width=1.2pt] ( 0.00, 2.00) -- ( 0.87,-1.80);
\draw[qqwuqq][opacity=0.25][line width=1.2pt] ( 0.00, 2.00) -- ( 1.94,-0.45);
\draw[blue][line width=1.2pt] ( 0.00, 2.00) -- (-1.95,-0.45); \draw[cyan] (-1.56,-0.  ) node[right]{$1$};
\draw[blue][line width=1.2pt] ( 0.00, 2.00) -- (-0.87,-1.80); \draw[cyan] (-0.81,-1.61) node[above left]{$0$};
\draw[blue][line width=1.2pt] ( 0.00, 2.00) -- ( 0.87,-1.80); \draw[cyan] ( 0.81,-1.61) node[above right]{$0$};
\draw[blue][line width=1.2pt] ( 0.00, 2.00) -- ( 1.94,-0.45); \draw[cyan] ( 1.56, 0.  ) node[left]{$0$};
\draw[->] (-0.8 , 0.8 ) -- (-0.5 , 0.5 ) [dash pattern=on 2pt off 2pt];
\draw[->] (-0.25, 0.25) -- ( 0.25, 0.25);
\draw[->] ( 0.5 , 0.5 ) -- ( 0.8 , 0.8 );
\draw (0,-2.8) node{$\Delta_2$};
\end{tikzpicture}
\begin{tikzpicture}[scale=0.6]
\draw[line width=1.2pt] (2,0) arc (0:360:2);
\draw[red][line width=1pt] (-1.56, 1.25) -- (-1.56,-1.25) -- (0,-2) -- ( 1.56,-1.25) -- ( 1.56, 1.25);
\fill[white] ( 1.56, 1.25) circle (0.1); \draw[red][line width=1pt] ( 1.56, 1.25) circle (0.1);
\fill[white] (-1.56, 1.25) circle (0.1); \draw[red][line width=1pt] (-1.56, 1.25) circle (0.1);
\fill[white] (-1.56,-1.25) circle (0.1); \draw[red][line width=1pt] (-1.56,-1.25) circle (0.1);
\fill[white] ( 0.  ,-2.  ) circle (0.1); \draw[red][line width=1pt] ( 0.  ,-2.  ) circle (0.1);
\fill[white] ( 1.56,-1.25) circle (0.1); \draw[red][line width=1pt] ( 1.56,-1.25) circle (0.1);
\filldraw[qqwuqq] ( 0.00, 2.00) circle (0.1) node[above]{\tiny$p=m_0$};
\filldraw[qqwuqq] (-1.95,-0.45) circle (0.1) node[ left]{\tiny$\m_1$};
\filldraw[qqwuqq] (-0.87,-1.80) circle (0.1) node[below]{\tiny$\m_2$};
\filldraw[qqwuqq] ( 0.87,-1.80) circle (0.1) node[below]{\tiny$\m_3$};
\filldraw[qqwuqq] ( 1.94,-0.45) circle (0.1) node[right]{\tiny$\m_4$};
\draw[qqwuqq][opacity=0.25][line width=1.2pt] ( 0.00, 2.00) -- (-1.95,-0.45);
\draw[qqwuqq][opacity=0.25][line width=1.2pt] ( 0.00, 2.00) -- (-0.87,-1.80);
\draw[qqwuqq][opacity=0.25][line width=1.2pt] ( 0.00, 2.00) -- ( 0.87,-1.80);
\draw[qqwuqq][opacity=0.25][line width=1.2pt] ( 0.00, 2.00) -- ( 1.94,-0.45);
\draw[blue][line width=1.2pt] ( 0.00, 2.00) -- (-1.95,-0.45); \draw[cyan] (-1.56,-0.  ) node[right]{$1$};
\draw[blue][line width=1.2pt] ( 0.00, 2.00) -- (-0.87,-1.80); \draw[cyan] (-0.81,-1.61) node[above left]{$0$};
\draw[blue][line width=1.2pt] (-0.87,-1.80) -- ( 0.87,-1.80); 
\draw[blue][line width=1.2pt] (-0.87,-1.80) -- ( 1.95,-0.45); 
\draw[->] (-0.8 , 0.8 ) -- (-0.5 , 0.5 ) [dash pattern=on 2pt off 2pt];
\draw[->] ( 0.2 ,-1.75) -- ( 0.2 ,-1.33);
\draw[->] ( 0.2 ,-1.15) -- (-0.35, 0.1 );
\draw (0,-2.8) node{$\Delta_{10}$};
\end{tikzpicture}
\begin{tikzpicture}[scale=0.6]
\draw[line width=1.2pt] (2,0) arc (0:360:2);
\draw[red][line width=1pt] (-1.56, 1.25) -- (-1.56,-1.25) -- (0,-2) -- ( 1.56,-1.25) -- ( 1.56, 1.25);
\fill[white] ( 1.56, 1.25) circle (0.1); \draw[red][line width=1pt] ( 1.56, 1.25) circle (0.1);
\fill[white] (-1.56, 1.25) circle (0.1); \draw[red][line width=1pt] (-1.56, 1.25) circle (0.1);
\fill[white] (-1.56,-1.25) circle (0.1); \draw[red][line width=1pt] (-1.56,-1.25) circle (0.1);
\fill[white] ( 0.  ,-2.  ) circle (0.1); \draw[red][line width=1pt] ( 0.  ,-2.  ) circle (0.1);
\fill[white] ( 1.56,-1.25) circle (0.1); \draw[red][line width=1pt] ( 1.56,-1.25) circle (0.1);
\filldraw[qqwuqq] ( 0.00, 2.00) circle (0.1) node[above]{\tiny$p=m_0$};
\filldraw[qqwuqq] (-1.95,-0.45) circle (0.1) node[ left]{\tiny$\m_1$};
\filldraw[qqwuqq] (-0.87,-1.80) circle (0.1) node[below]{\tiny$\m_2$};
\filldraw[qqwuqq] ( 0.87,-1.80) circle (0.1) node[below]{\tiny$\m_3$};
\filldraw[qqwuqq] ( 1.94,-0.45) circle (0.1) node[right]{\tiny$\m_4$};
\draw[qqwuqq][opacity=0.25][line width=1.2pt] ( 0.00, 2.00) -- (-1.95,-0.45);
\draw[qqwuqq][opacity=0.25][line width=1.2pt] ( 0.00, 2.00) -- (-0.87,-1.80);
\draw[qqwuqq][opacity=0.25][line width=1.2pt] ( 0.00, 2.00) -- ( 0.87,-1.80);
\draw[qqwuqq][opacity=0.25][line width=1.2pt] ( 0.00, 2.00) -- ( 1.94,-0.45);
\draw[blue][line width=1.2pt] (-1.94,-0.45) -- ( 0.87,-1.80);
\draw[blue][line width=1.2pt] (-0.87,-1.80) -- ( 0.87,-1.80);
\draw[blue][line width=1.2pt] ( 0.00, 2.00) -- ( 0.87,-1.80); \draw[cyan] ( 0.71,-1.61) node[above right]{$1$};
\draw[blue][line width=1.2pt] ( 0.00, 2.00) -- ( 1.94,-0.45); \draw[cyan] ( 1.56, 0.  ) node[left]{$0$};
\draw[->] ( 0.5 , 0.5 ) -- ( 0.8 , 0.8 ) [dash pattern=on 2pt off 2pt];
\draw[->] (-0.2 ,-1.35) -- (-0.2 ,-1.75);
\draw[->] ( 0.45, 0.1 ) -- (-0.17,-1.22);
\draw (0,-2.8) node{$\Delta_{12}$};
\end{tikzpicture}
\caption{ ``$\dashrightarrow$'' is an arrow with non-zero grading, 
$\color{cyan}0$ and $\color{cyan}1$ are intersection indices.}
\label{FFASofsilt}
\end{figure}
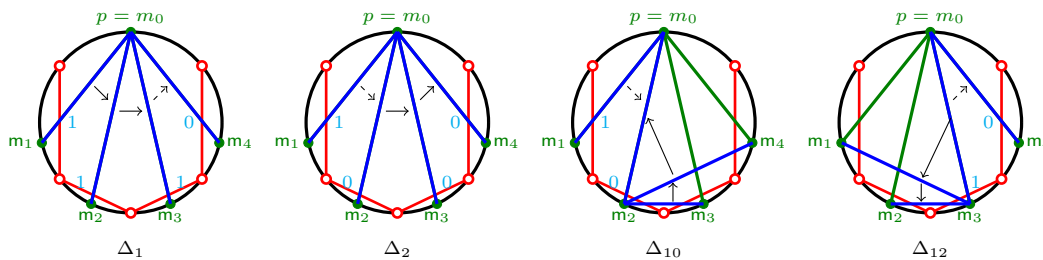
At last, we obtain that there are 15 silted algebras of type $A_{4}$, See \cite{Xing21}.

\end{example}


\def\cprime{$'$}
\providecommand{\bysame}{\leavevmode\hbox to3em{\hrulefill}\thinspace}
\providecommand{\MR}{\relax\ifhmode\unskip\space\fi MR }
\providecommand{\MRhref}[2]{%
  \href{http://www.ams.org/mathscinet-getitem?mr=#1}{#2}
}
\providecommand{\href}[2]{#2}




\end{document}